\documentclass[a4paper,11pt]{article}

\usepackage[
paperwidth=22cm,
paperheight=31cm,
body={18true cm, 27true cm},
headheight=3cm
]{geometry}
\usepackage{fancyhdr}
\usepackage{tocloft}
\usepackage[english]{babel}

\usepackage{amsmath,amssymb,amsfonts,amsthm}
\usepackage{bm}
\usepackage{mathtools}
\usepackage{cases}

\usepackage{graphicx}
\usepackage{subcaption}

\usepackage[x11names]{xcolor}

\usepackage{pgfplots}
\pgfplotsset{compat=1.18}
\usepackage{tikz}
\usetikzlibrary{arrows.meta,positioning,decorations.markings,shapes.geometric,fit}
\usetikzlibrary{calc,arrows.meta,decorations.markings}
\usepackage[colorlinks=true,    
linkcolor=blue, 
citecolor=red,  
urlcolor=magenta]{hyperref} 
\usepackage{xeCJK}

\usepackage{newtxtext}
\usepackage{mathrsfs}

\numberwithin{equation}{section}
\numberwithin{figure}{section}

\theoremstyle{plain}
\newtheorem{theorem}{Theorem}[section]
\newtheorem{lemma}[theorem]{Lemma}
\newtheorem{prop}[theorem]{Proposition}

\newtheorem{assumption}[theorem]{Assumption}

\theoremstyle{definition}
\newtheorem{remark}{Remark}[section]
\newtheorem{problem}{Riemann-Hilbert Problem}
\newtheorem{basicproblem}{Basic Riemann--Hilbert Problem}
\newtheorem{parproblem}{$\bar{\partial}$-problem}

\makeatletter

\newcommand{\Rmnum}[1]{\expandafter\@slowromancap\romannumeral #1@}
\makeatother

\definecolor{cycle-a}{HTML}{1F77B4}
\definecolor{cycle-b}{HTML}{FF7F0E}
\definecolor{PlaneWave}{HTML}{A8DADC}
\definecolor{SlowDecay}{HTML}{F1FAEE}
\definecolor{SlowDecay1}{HTML}{386641}
\definecolor{GridColor}{HTML}{BC4749}

\usepackage{enumitem}

\title{\bf Soliton gas for the derivative nonlinear Schr\"odinger equation: continuum \(\bar\partial\)-problem, genus reduction and asymptotics
}

\author{\large Deng-Shan Wang, Xinyu Wang{\thanks{E-mail
			address for corresponding author:\,\,xnyu.wang@mail.bnu.edu.cn}}
	\\
	\\{\large School of Mathematical Sciences, Beijing Normal University, Beijing 100875, China}\\
}

\date{}

\begin{document}
	\maketitle
	
	\fancyhead{} 
	\fancyfoot[C]{\thepage} 
	\pagestyle{fancy}  
	\renewcommand{\headrulewidth}{0pt}

	\begin{abstract}
		A soliton gas theory for the derivative nonlinear Schr\"odinger equation is developed, focusing on the Gerdjikov--Ivanov equation and using its gauge equivalence with the Kaup--Newell and Chen--Lee--Liu equations. Starting from the reflectionless inverse scattering problem, we formulate pure \(N\)-soliton solutions through a meromorphic Riemann--Hilbert problem and pass to the continuum limit as the discrete spectrum condenses on the planar spectral domains. Under a suitable scaling of the norming constants, the limit yields a compactly supported \(\bar\partial\)-problem. For domains admitting a Schwarz function, in particular elliptic domains, this problem reduces to a Riemann--Hilbert problem on the associated mother body. For the elliptic soliton gas, the large-\(x\) and long-time asymptotic behaviors are established: the solution decays as \(x\to+\infty\), approaches an elliptic finite-gap background as \(x\to-\infty\), and exhibits stratified long-time sectors described by one-, two-, and three-phase Riemann theta functions. A distinctive feature of derivative nonlinear Schr\"odinger equation is the \(z\mapsto -z\) symmetry, which induces a quotient reduction of the effective Abelian geometry via \(\lambda=z^2\). We further derive a kinetic equation for the effective velocity of the elliptic soliton gas and an Its–Izergin–Korepin–Slavnov type Fredholm determinant representation of the continuum \(\tau\)-function.
	\end{abstract}
	
	
	\setcounter{tocdepth}{1}
	\setlength{\cftbeforesecskip}{1pt}
	\tableofcontents

	\section{Introduction}
	The derivative nonlinear Schr\"odinger (DNLS) equation
	\begin{equation}\label{dnls1}
		i q_t+q_{xx}+i(|q|^2q)_x=0,
		\tag{DNLS-\MakeUppercase{\romannumeral1}}
	\end{equation}
	is one of the fundamental integrable models in the theory of nonlinear dispersive waves. It arises, for instance, in magnetohydrodynamics with the Hall effect and in the description of weakly nonlinear, long-wavelength, circularly polarized Alfv\'en waves propagating along a background magnetic field in a magnetized plasma \cite{Mjolhus-1974}-\cite{Ander}; it also appears in nonlinear optics through models involving self-steepening effects. The complete integrability of \eqref{dnls1} was established by Kaup and Newell \cite{Kaup-Newell}, who introduced its Lax pair and developed the associated inverse scattering formalism, thus \eqref{dnls1} is often called the Kaup-Newell equation.
	\par
	Although the DNLS equation is closely related to the standard cubic nonlinear Schr\"odinger (NLS) equation, the derivative nonlinearity introduces substantial changes in both the dynamics and the spectral theory. It destroys the usual Galilean invariance, leads to a quadratic dependence on the spectral parameter in the Lax pair, and produces additional spectral symmetries. These features make DNLS a natural but genuinely different integrable model for studying solitons, finite-gap waves, long-time dynamics, and coherent many-soliton ensembles.
	\par
	There are several gauge equivalent forms of the DNLS equation. Besides the Kaup--Newell equation \eqref{dnls1}, two important representatives are the Chen--Lee--Liu equation \cite{Chen} 
	\begin{equation}\label{dnls2}
		iQ_t+Q_{xx}+i|Q|^2Q_x=0,
		\tag{DNLS-\MakeUppercase{\romannumeral2}}
	\end{equation}
	and the Gerdjikov--Ivanov equation \cite{GI}
	\begin{equation}\label{dnls3}
		iu_t+u_{xx}-iu^2\bar u_x+\frac12|u|^4u=0.
		\tag{DNLS-\MakeUppercase{\romannumeral3}}
	\end{equation}
	They are related to the Kaup--Newell equation \eqref{dnls1} by the nonlocal gauge transformations \cite{Malomed-PRA-2007}
	\[
	Q(x,t)=q(x,t)\exp\!\left(\frac{i}{2}\int_{-\infty}^{x}|q(y,t)|^2\,dy\right),
	\qquad
	u(x,t)=q(x,t)\exp\!\left(i\int_{-\infty}^{x}|q(y,t)|^2\,dy\right).
	\]
	Thus the three equations share the same density,
	\[
	|q(x,t)|^2=|Q(x,t)|^2=|u(x,t)|^2,
	\]
	and their inverse scattering problems are closely connected. In this paper, we work mainly with the Gerdjikov--Ivanov form \eqref{dnls3}. This choice is technically convenient because its Riemann--Hilbert formulation has a symmetric residue structure and is well adapted to the spectral symmetries needed in the construction of the continuum soliton gas limit. The corresponding results for the Kaup--Newell and Chen--Lee--Liu equations can then be recovered through these gauge transformations; see Appendix~\ref{app:gauge-equivalence} for the precise gauge equivalence relations among the three equations at the level of Lax pairs and Riemann--Hilbert problems.

	\subsection{Prior work and motivation}
	
	Over the past several decades, the derivative nonlinear Schr\"odinger equation has been studied from several complementary perspectives,
	including inverse scattering transform \cite{XJChen-PRE-2004,Zhang-Yan-JNS-2020}, well-posedness \cite{Inven}-\cite{Liu} and long-time asymptotics \cite{Liu2,Liu3}, and so on.
	A fundamental distinction between DNLS and NLS is the derivative nonlinearity. For example, the
	Kaup--Newell equation \eqref{dnls1} contains the term
	\((|q|^2q)_x\), whereas the cubic NLS equation contains the
	derivative-free term \(|q|^2q\). This difference affects both the
	analytic theory and the spectral and asymptotic structures of the
	corresponding nonlinear waves.
	\par
	A first major line of research concerns the integrable and analytical
	foundations of the DNLS equation. Kaup and Newell
	\cite{Kaup-Newell} established its complete integrability by deriving the Lax pair, conservation laws, and soliton solutions. Subsequent works developed the associated inverse scattering theory through Jost solutions, scattering data, and Riemann--Hilbert formulations
	\cite{XJChen-PRE-2004,Zhang-Yan-JNS-2020}. In parallel, substantial progress has been made in the well-posedness theory. Global existence and scattering results have been obtained in weighted Sobolev spaces by
	inverse scattering methods \cite{Liu}-\cite{Peli2}, while soliton resolution was established in \cite{Liu3}. More recently, global well-posedness has been proved at the critical regularities \(H^{1/2}(\mathbb R)\) and \(L^2(\mathbb R)\) \cite{Kill-Visan}. These results provide a robust analytical framework for the study of DNLS dynamics.
	
	A second line of research concerns special solutions and their
	large-scale behavior. In addition to soliton and multi-soliton
	solutions, the DNLS equation
	admits finite-gap and hyperelliptic solutions
	\cite{Zhao,Wright}, as well as rich long-time asymptotic regimes
	\cite{Xu}. Related developments in Whitham
	modulation theory and dispersive hydrodynamics have produced periodic solutions, modulation systems, and dispersive shock wave descriptions for the gauge equivalent DNLS equations \cite{Ivanov-2017,Ivanov-2020,G-W}. These studies show that DNLS supports not only isolated solitons, but also collective finite-gap and modulated wave fields.
	
	The notion of soliton gas originates from Zakharov's kinetic description of a rarefied KdV soliton gas \cite{Zakh}. In this picture, the propagation of a trial soliton is affected by the cumulative phase shifts generated by its interactions with the surrounding solitons. This viewpoint was later extended to dense soliton gases by El and collaborators \cite{El1,El2,El3}, leading to self-consistent kinetic equations for the spectral density and the effective soliton velocity. More recently, soliton gases and breather gases have been studied extensively for the KdV equation \cite{DZZ}-\cite{Wang-SIGMA-2026}, the NLS equation \cite{Ph.D}-\cite{TRZZ2}, and the mKdV equation \cite{Grava-R,Zhang-Ling}. A rigorous Riemann--Hilbert approach has also emerged, in which soliton gases are obtained as continuum limits of pure \(N\)-soliton solutions and their macroscopic behavior is described by finite-gap backgrounds, Whitham modulation and kinetic velocity laws \cite{Grave-R,Grava-R}.
	
	Compared with the KdV, NLS, and mKdV equations, the DNLS soliton gas remains
	much less understood. This is significant both physically and
	mathematically. DNLS models describe nonlinear wave regimes in which
	self-steepening or Hall-type corrections cannot be neglected. At the
	same time, the even dependence of the spectral phase on the spectral
	parameter, the symmetry \(z\mapsto -z\), and the nonlocal gauge
	transformations among the DNLS equations lead to an inverse problem and
	a finite-gap geometry that differ essentially from those of the
	standard NLS equation.
	
	Consequently, the DNLS soliton gas cannot be obtained by a direct transfer of the existing NLS theory. Several basic questions arise: how should the infinite soliton limit be formulated as a continuum inverse problem; which finite-gap state is selected by spectral condensation; how does the resulting gas behave in the large-\(x\) and long-time regimes; what kinetic equation governs its effective
	velocity; and how is the continuum solution encoded by Fredholm determinants and \(\tau\)-functions?
	
	The purpose of this paper is to address these questions for the Gerdjikov--Ivanov equation \eqref{dnls3}. Starting from the reflectionless RHP for pure \(N\)-soliton solutions, we construct a continuum \(\bar\partial\)-problem describing the DNLS soliton gas. For admissible condensation domains, and in particular for elliptic domains admitting a Schwarz function description, the continuum problem reduces to a RHP supported on the corresponding mother body. This reduction reveals the finite-gap structure selected by spectral condensation and provides the starting point for the large-space, long-time, kinetic theory, and the Fredholm determinant analysis developed below.

	\subsection{Main results}
	
	Throughout this paper we work with the Gerdjikov--Ivanov equation \eqref{dnls3}. The corresponding results for the Kaup--Newell and Chen--Lee--Liu equations follow from the gauge equivalence described in Appendix~\ref{app:gauge-equivalence}.
	\par
	{\bf Our first result is the construction of a continuum inverse problem for the elliptic DNLS soliton gas.} Starting from the reflectionless Riemann--Hilbert problem (RHP) for pure \(N\)-soliton solutions, we let the discrete eigenvalues condense, as \(N\to\infty\), on a bounded domain \(\mathcal D_+\) in the first quadrant of the spectral plane, together with its images generated by the DNLS spectral symmetries. Under the scaling of the norming constants specified in Assumption~\ref{ass:spectral-condensation}, the meromorphic RHP converges to a compactly supported \(\bar\partial\)-problem, denoted by \(\bar\partial\)-problem~\ref{prob:dbar}. This problem provides the continuum spectral description of the DNLS soliton gas, expressed by the	 reconstruction formula  (\ref{eq:elliptic-reconstruction}) and RHP \ref{prob:elliptic-rhp} below.
	\par
	For reference, the free velocity of a single DNLS soliton with spectral parameter \(z=\eta+i\omega\) is
	\begin{equation}\label{eq:free-soliton-velocity}
		v_{\rm sol}
		=
		-4(\eta^2-\omega^2)
		=
		-4\operatorname{Re}(z^2).
	\end{equation}
	This velocity is the microscopic input that reappears in the kinetic description of the soliton gas.
	
	A further reduction is available when the boundary of \(\mathcal D_+\) admits a Schwarz function. In this case, the area integrals in the \(\bar\partial\)-problem can be reduced to contour integrals supported on the mother body of \(\mathcal D_+\). Consequently, the continuum inverse problem is transformed into the RHP. In the elliptic case considered in Section~\ref{subsec:elliptic-model}, the mother body is the segment joining the two foci of the ellipse. This reduction makes the finite-gap structure selected by the spectral condensation explicit and provides the starting point for the asymptotic analysis.
	\par
	Our second rigorous asymptotic result concerns the large-\(x\) and long-time behaviors of the elliptic DNLS soliton gas.
	
	\begin{theorem}[Large-\(x\) asymptotics]\label{spatial}
		The initial soliton gas potential \(u(x,0)\) of the Gerdjikov--Ivanov equation \eqref{dnls3} satisfies the large-\(x\) asymptotics of the form
		\[
		u(x,0)=
		\begin{cases}
			\mathcal O(e^{-c_r x}), & x\to+\infty, \\[4pt]
			\mathcal U_-(x)+\mathcal O(|x|^{-1}), & x\to-\infty,
		\end{cases}
		\qquad c_r>0,
		\]
		where 
		\[
		\mathcal U_-(x)
		=
		2i\bigl(\operatorname{Im}E_2-\operatorname{Im}E_1\bigr)
		\frac{
			\Theta(0;\widehat{\mathcal P})
			\Theta(\widehat{\mathcal{A}}_\infty-\Delta(x);\widehat{\mathcal P})
		}{
			\Theta(\Delta(x);\widehat{\mathcal P})
			\Theta(\widehat{\mathcal{A}}_\infty;\widehat{\mathcal P})
		}
		F_\infty^2 e^{2ig_\infty x},
		\]
		and $\widehat{\mathcal P}$, $\widehat{\mathcal{A}}_\infty$, $\Delta(x)$, $g_\infty$, $F_\infty$ are given by \eqref{period-matrix-x}, \eqref{limit-of-abel-map-x}, \eqref{delta-x}, \eqref{g-infty-x} and
		\eqref{F-infty}, respectively.
	\end{theorem}
	
	\begin{figure}[htbp]
		\centering
		
		\begin{subfigure}{\linewidth}
			\centering
			\includegraphics[width=\linewidth]{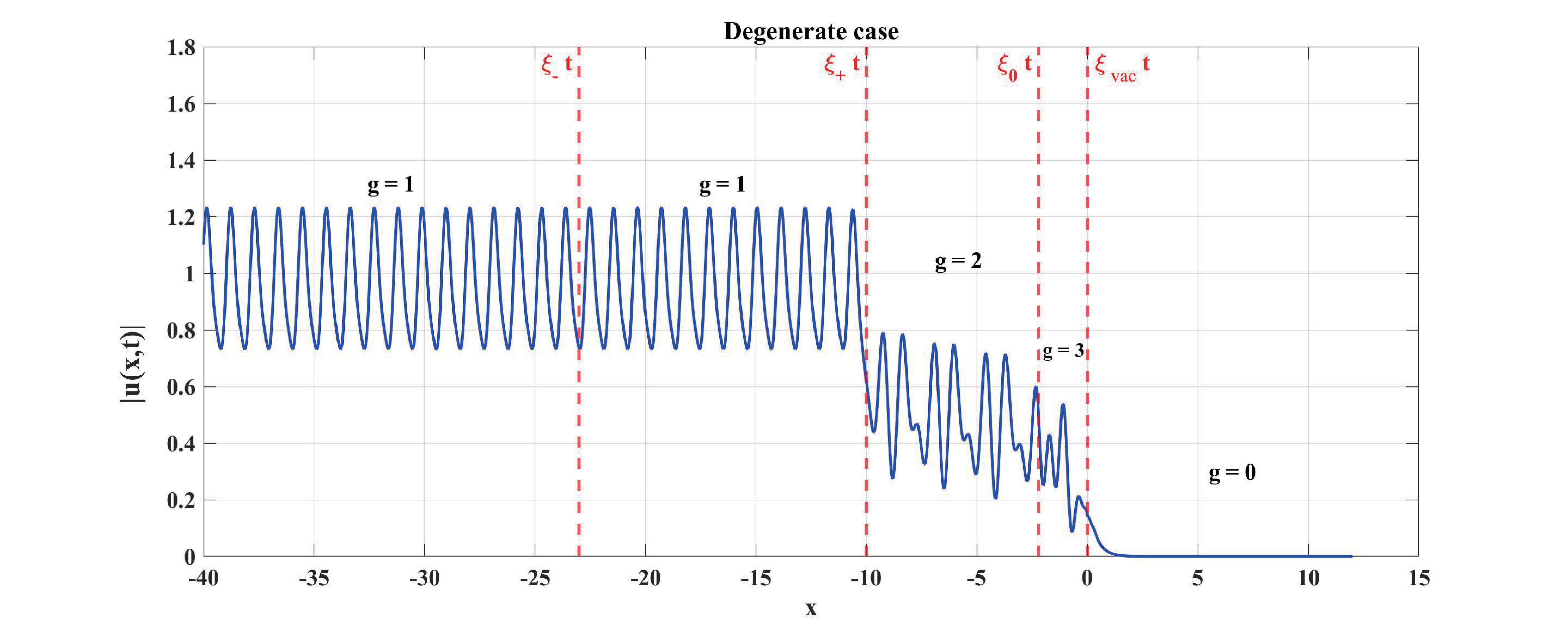}
			\caption{Degenerate case at $t=10$, $\operatorname{Im} E_2=2$, $\operatorname{Im} E_1=1.525$, $\xi_-=-2.304$, $\xi_+=-1.002$, $\xi_0=-0.223$.}
		\end{subfigure}
		
		\vspace{0.8em}
		
		\begin{subfigure}{\linewidth}
			\centering
			\includegraphics[width=\linewidth]{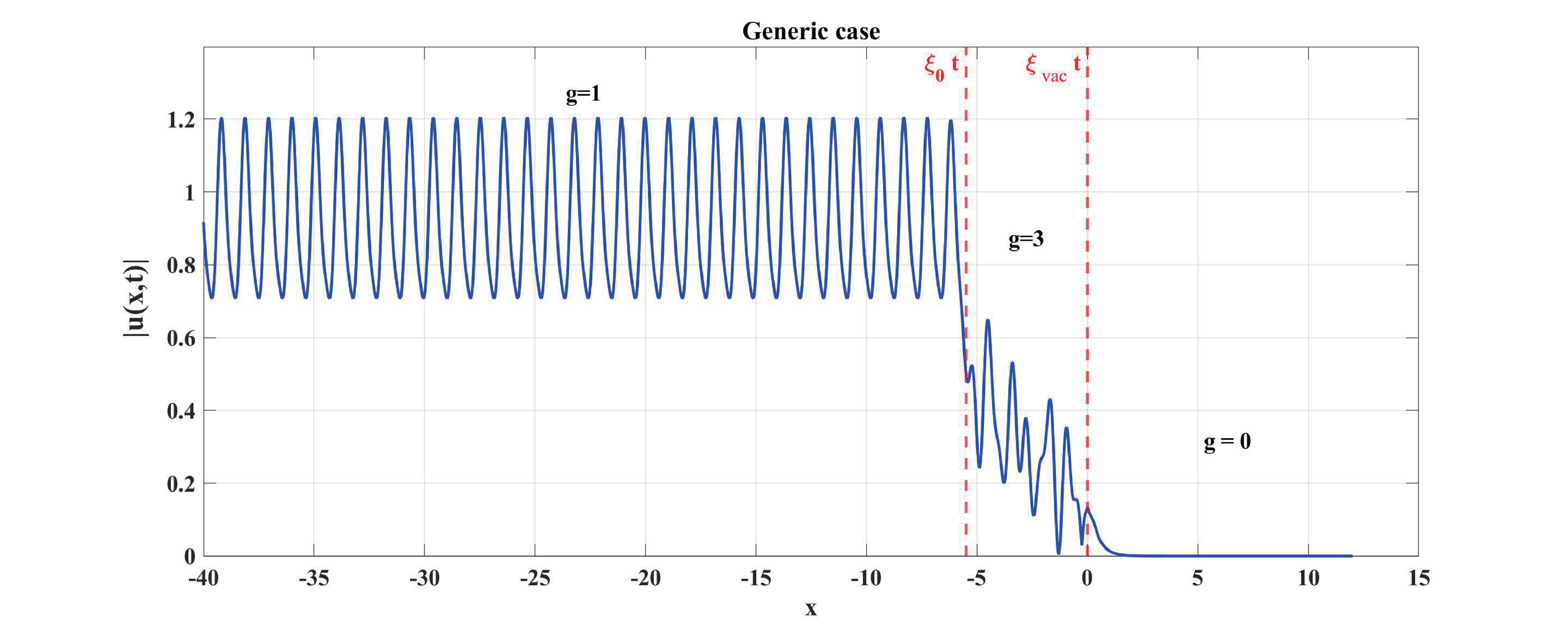}
			\caption{Generic case at $t=10$, $\operatorname{Im} E_2=0.53$, $\operatorname{Im} E_1=0.07$, $\xi_0=-0.550$.}
		\end{subfigure}
		
		\vspace{0.8em}
		\begin{subfigure}{\linewidth}
			\centering
			\begin{tikzpicture}[scale=1.15, line cap=round, line join=round]
				
				\tikzset{
					axis/.style={->, semithick},
					ray/.style={semithick},
					lab/.style={font=\small},
					reg/.style={font=\small},
					toplab/.style={font=\small, fill=white, inner sep=1.2pt},
				}
				
				\colorlet{gOne}{black!6}
				\colorlet{gOneS}{black!10}
				\colorlet{gTwo}{black!14}
				\colorlet{gThree}{black!18}
				\colorlet{vac}{black!24}
				
				\begin{scope}[shift={(0,0)}]
					
					\fill[gOne]   (0,0) -- (-3,1.5) -- (-3,0) -- cycle;
					\fill[gOneS]  (0,0) -- (-3,1.5) -- (-3,3.5) -- cycle;
					\fill[gTwo]   (0,0) --(-3,3.5)-- (-3,4) -- (-1,4) -- cycle;
					\fill[gThree] (0,0) -- (-1,4) -- (2.5,4) -- cycle;
					\fill[vac]    (0,0) -- (2.5,4) -- (3.2,4) -- (3.2,0) -- cycle;
					
					\draw[axis,dashed] (-3.2,0) -- (3.2,0) node[right] {$x$};
					\draw[axis,dashed] (0,0) -- (0,4.2) node[above] {$t$};
					
					\draw[ray] (0,0) -- (-3,1.5);
					\draw[ray] (0,0) -- (-3,3.5);
					\draw[ray] (0,0) -- (-1,4);
					\draw[ray] (0,0) -- (2.5,4);
					
					\node[toplab, above left] at (-3,1.5) {$\xi=\xi_-\,t$};
					\node[toplab, above left] at (-3,3.5) {$\xi=\xi_+\,t$};
					\node[toplab, above left] at (-1,4) {$\xi=\xi_0\,t$};
					\node[toplab, above right] at (2.1,4) {$\xi=\xi_{\mathrm{vac}}\,t$};
					
					\node[reg] at (-2.05,0.55) {genus 1};
					\node[reg] at (-2.10,1.75) {genus $1_s$};
					\node[reg] at (-1.60,2.95) {genus 2};
					\node[reg] at (0.2,3.30) {genus 3};
					\node[reg] at (1.95,1.90) {$\mathrm{vacuum}$};
					
					\node[font=\small] at (0,-0.45) {(c.1)};
					
				\end{scope}
				
				\begin{scope}[shift={(7.5,0)}]
					
					\fill[gOne]   (0,0) -- (-3,4) -- (-3,0) -- cycle;
					\fill[gThree] (0,0) -- (-3,4) -- (2.0,4) -- cycle;
					\fill[vac]    (0,0) -- (2.0,4) -- (3.2,4) -- (3.2,0) -- cycle;
					
					\draw[axis,dashed] (-3.2,0) -- (3.2,0) node[right] {$x$};
					\draw[axis,dashed] (0,0) -- (0,4.2) node[above] {$t$};
					
					\draw[ray] (0,0) -- (-3,4);
					\draw[ray] (0,0) -- (2.0,4);
					
					\node[toplab, above left] at (-3,4) {$\xi=\xi_0\,t$};
					\node[toplab, above right] at (2.0,4) {$\xi=\xi_{\mathrm{vac}}\,t$};
					
					\node[reg] at (-1.90,1.25) {genus 1};
					\node[reg] at (-0.70,2.65) {genus 3};
					\node[reg] at (1.90,1.90) {$\mathrm{vacuum}$};
					
					\node[font=\small] at (0,-0.45) {(c.2)};
					
				\end{scope}
				
			\end{tikzpicture}
			\caption{The solution regions in the $(x,t)$ half-plane: degenerate case on the left and generic case on the right.}
		\end{subfigure}
		
		\caption{
			Long-time asymptotic structure of the elliptic DNLS soliton gas.
			Panels (a) and (b) display representative solution profiles in the degenerate and generic regimes, respectively.
			Panel (c) displays the induced decomposition of the upper \((x,t)\)-half plane into asymptotic sectors, with the degenerate case on the left and the generic case on the right.
			As the self-similar variable \(\xi=x/t\) varies, the leading-order terms are given by unmodulated one-phase finite-gap waves, modulated multi-phase finite-gap waves, or an exponentially small field (vacuum region).
		}
		
		\label{fig:soliton-gas-all}
	\end{figure}
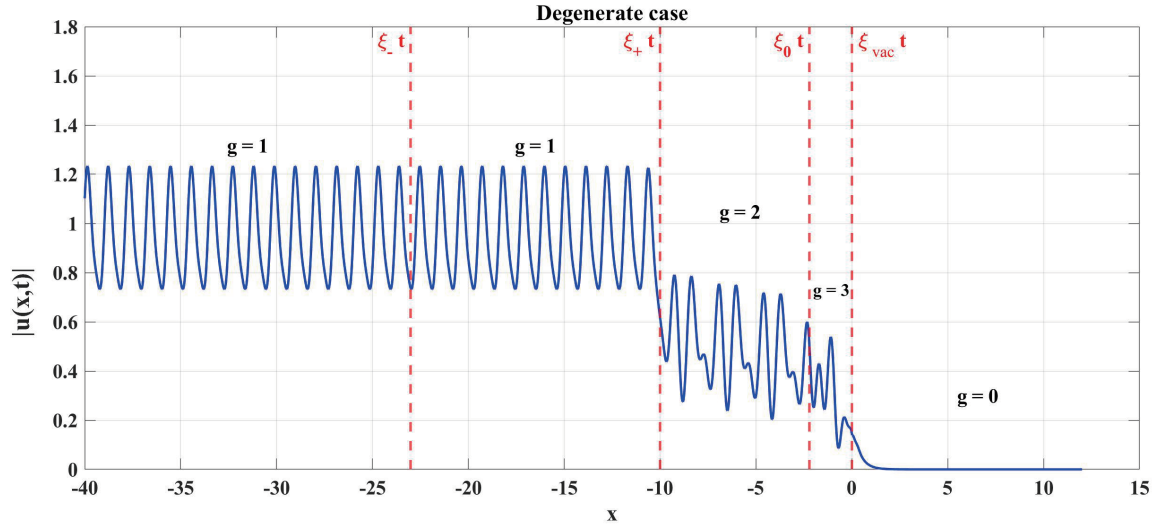
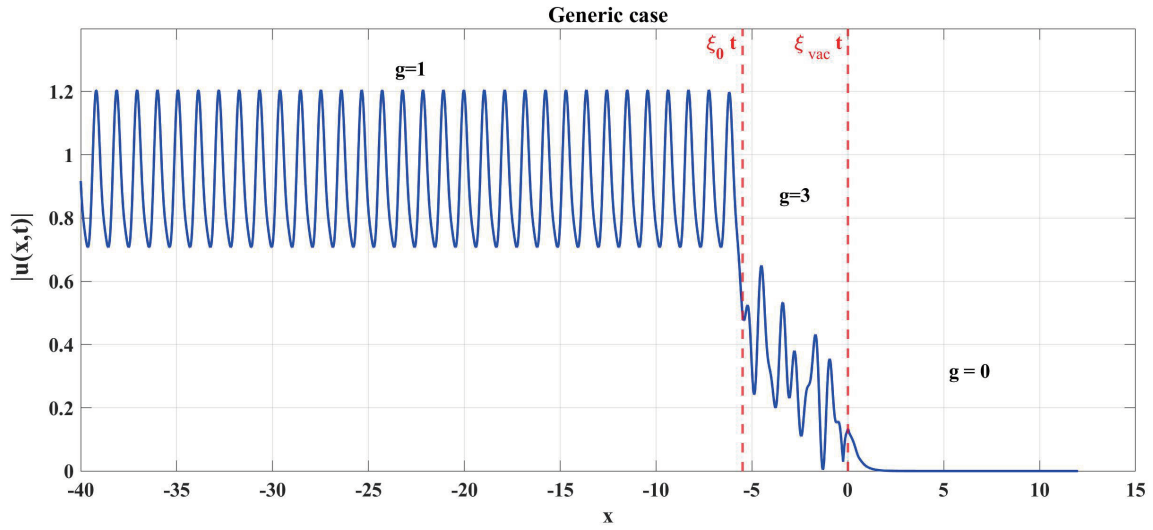
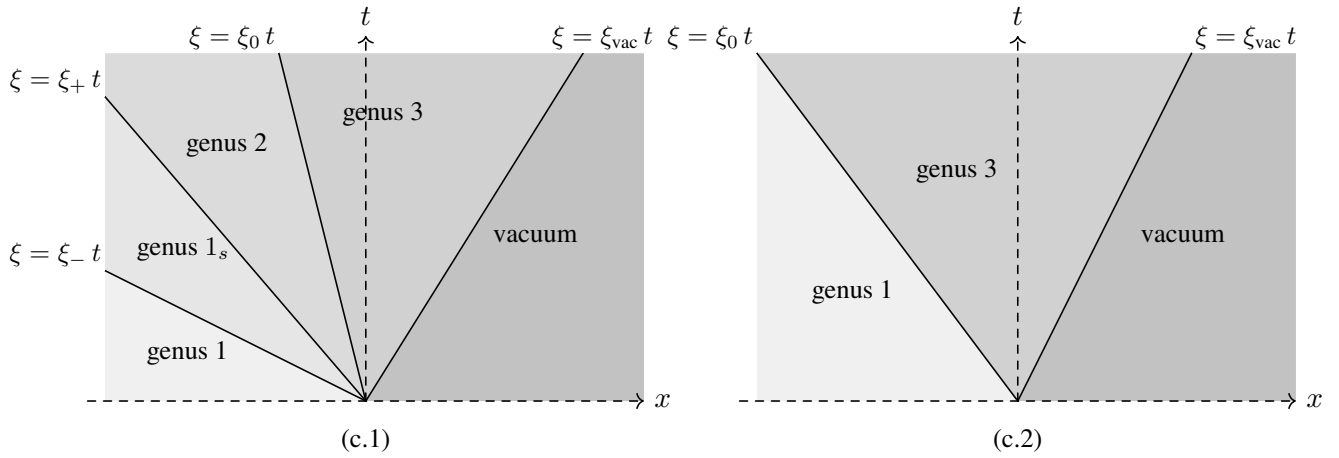
	
	Thus the limiting profile is spatially asymmetric, with a decaying state on the right and an elliptic finite-gap background on the left.
	We next consider the long-time regime
	\[
	\xi=\frac{x}{t},\qquad t\to+\infty .
	\]
	The resulting asymptotic structure depends on the interaction between the saddle points and the zero level set of 
	\(\operatorname{Im}g\). Two cases occur: a degenerate case and a generic case. In both cases, the upper \((x,t)\)-half plane is divided into asymptotic sectors in which the leading term is either exponentially small or described by Riemann theta functions of different effective genera as depicted in Figure~\ref{fig:soliton-gas-all}.
	
	\begin{theorem}[Long-time asymptotics: degenerate case]\label{time degen}
		Assume that the saddle point configuration is degenerate. Then, as \(t\to+\infty\), the elliptic DNLS soliton gas $u(x,t)$ reconstructed from  (\ref{eq:elliptic-reconstruction}) has the following asymptotic behaviors:
		\[
		u(x,t)=
		\begin{cases}
			\mathcal U_1(x,t)+\mathcal O(t^{-1}),
			& -\infty<\xi<\xi_-,\\[1mm]
			\mathcal U_1(x,t)+\mathcal O(t^{-1}),
			& \xi_-<\xi<\xi_+,\\[1mm]
			\mathcal U_2(x,t)+\mathcal O(t^{-1}),
			& \xi_+<\xi<\xi_0,\\[1mm]
			\mathcal U_3(x,t)+\mathcal O(t^{-1}),
			& \xi_0<\xi<\xi_{\mathrm{vac}},\\[1mm]
			\mathcal O(e^{-ct}),
			& \xi>\xi_{\mathrm{vac}},
		\end{cases}
		\]
		where
		\[
		\begin{aligned}
			&\mathcal{U}_1(x,t)
			=
			2i
			\bigl(
			\operatorname{Im}E_2-\operatorname{Im}E_1
			\bigr)
			T_1(\xi)
			(F_{\infty,1}(\xi))^2
			e^{2it g_{\infty,1}(\xi)},\\
			&\mathcal{U}_2(x,t)=
			2i
			\bigl(
			\operatorname{Im}E_2-\operatorname{Im}E_1
			-\operatorname{Im}c_1
			\bigr)
			T_2(\xi)
			\bigl(F_{\infty,2}(\xi)\bigr)^2
			e^{2it g_{\infty,2}(\xi)},\\
			&\mathcal{U}_3(x,t)=
			2i
			\bigl(
			\operatorname{Im}E_2-\operatorname{Im}E_1
			+\operatorname{Im}b_2
			-\operatorname{Im}b_3
			\bigr)
			T_3(\xi)
			\bigl(F_{\infty,3}(\xi)\bigr)^2
			e^{2it g_{\infty,3}(\xi)},
		\end{aligned}
		\]
		are respectively one-phase, two-phase, and three-phase theta-functional expressions. Here 
		\[
		T_n(\xi)
		=
		\frac{
			\Theta(0;\widehat{\mathcal P})
			\Theta\bigl(
			\widehat{\mathcal A}_{n,\infty}
			-\boldsymbol{\Delta}_n(\xi);
			\widehat{\mathcal P}
			\bigr)
		}{
			\Theta\bigl(
			\boldsymbol{\Delta}_n(\xi);
			\widehat{\mathcal P}
			\bigr)
			\Theta\bigl(
			\widehat{\mathcal A}_{n,\infty};
			\widehat{\mathcal P}
			\bigr),
		}\quad n=1,2,3,
		\]
		and $\widehat{\mathcal P}$, $\widehat{\mathcal A}_{n,\infty}$, $\boldsymbol{\Delta}_n$ are given by \eqref{period-matrix-t}, \eqref{limit-of-abel-map}, \eqref{eq:Delta-n-vector} respectively.
		The parameters $g_{\infty,n}$ and $F_{\infty,n}$ are defined by \eqref{g-function-n} and \eqref{eq:Fn-infinity},respectively. The critical values \(\xi_\mathrm{vac}\), \(\xi_-\), \(\xi_+\), and \(\xi_0\) are determined by the \(g\)-function mechanism in Section~\ref{long-time- asymptotics}.
	\end{theorem}
	
	\begin{remark}
		The first two sectors in Theorem~\ref{time degen} (see also Figure \ref{fig:soliton-gas-all}(a) and (c)) have the same leading-order expression of theta function but correspond to different signature charts of the \(g\)-function. For this reason they are treated separately in the detailed asymptotic analysis.
	\end{remark}

	\begin{theorem}[Long-time asymptotics: generic case]\label{time gener}
		Assume that the saddle point configuration is generic. Then, as \(t\to+\infty\), the elliptic DNLS soliton gas $u(x,t)$ reconstructed from  (\ref{eq:elliptic-reconstruction}) has the following asymptotic behaviors:
		\[
		u(x,t)=
		\begin{cases}
			\mathcal U_1(x,t)+\mathcal O(t^{-1}),
			& -\infty<\xi<\xi_0,\\[1mm]
			\mathcal U_3(x,t)+\mathcal O(t^{-1}),
			& \xi_0<\xi<\xi_{\mathrm{vac}},\\[1mm]
			\mathcal O(e^{-ct}),
			& \xi>\xi_{\mathrm{vac}},
		\end{cases}
		\]
		where \(\mathcal U_1\) and \(\mathcal U_3\) denote the one-phase and three-phase theta-functional leading terms, respectively, given above.
	\end{theorem}
	
	Theorems~\ref{time degen} and~\ref{time gener} show that the DNLS soliton gas has a stratified long-time asymptotic structure. Depending on the self-similar variable \(\xi\), the leading behavior is either an exponentially small field or a finite-gap oscillatory wave with different effective genera. The resulting decomposition of the upper \((x,t)\)-plane is determined by the limiting spectral distribution and the associated \(g\)-function.
	\par
	A specific DNLS feature of this finite-gap description is the reduction of the effective Abelian geometry induced by the symmetry \(z\mapsto -z\). The natural symmetric spectral curve in the original \(z\)-plane has higher genus, but the RHP involves only the invariant part of its holomorphic differentials and period data. Equivalently, the relevant Abelian data are pulled back from the quotient curve under the map \(\lambda=z^2\). This is a quotient reduction of the effective finite-gap geometry, rather than a degeneration of the spectral curve itself. The detailed construction is given in Section~\ref{sec:genus-reduction}.
	
	The asymptotic analysis also leads to a kinetic description of the DNLS soliton gas on the quotient spectral plane.
	Fix a genus-\(n\) asymptotic sector, \(n\in\{1,2,3\}\), and let 
	\(\Lambda_n(\xi)\) denote the active quotient spectral support in the variable 
	\(\lambda=z^2\). Let \(\mathcal N_n(x,t,\lambda)\) be the cumulative spectral distribution 
	defined in Section~\ref{sec:dnls-kinetic}.
	The effective velocity $v_{{\rm eff},n}(x,t,\lambda)$ satisfies the kinetic equation
	\begin{equation}\label{eq:soliton-gas-kinetic}
		v_{{\rm eff},n}(x,t,\lambda)
		=
		-4\operatorname{Re}\lambda
		+
		\frac{1}{\operatorname{Im}\lambda}
		\int_{\Lambda_n(\xi)}
		\log\left|
		\frac{\lambda-\overline{\mu}}
		{\lambda-\mu}
		\right|
		\Bigl[
		v_{{\rm eff},n}(x,t,\mu)
		-
		v_{{\rm eff},n}(x,t,\lambda)
		\Bigr]
		\mathfrak f_n(x,t,\mu)\,d\mu,
		\qquad
		\lambda\in\Lambda_n(\xi),
	\end{equation}
	where the density function $\mathfrak f_n(x,t,\mu)$ is given in \eqref{kinetic-density}.
	The first term \(-4\operatorname{Re}\lambda\) is precisely the velocity of the free single soliton shown in \eqref{eq:free-soliton-velocity}. The integral term describes the renormalization of this velocity caused by interactions with the surrounding spectral components of the soliton gas. 
	\par
	Finally, we give an operator theoretic characterization of the continuum limit solution. The \(\bar\partial\)-problem~\ref{prob:dbar} in (\ref{eq:dbar-problem}) can be reformulated in terms of an Its--Izergin--Korepin--Slavnov integrable operator \cite{IIKS}. After the change of variable \(\lambda=z^2\), one obtains an operator \(\mathcal K^+\) defined in \eqref{K-operator} whose kernel has an integrable representation. Under the corresponding trace class assumptions, the continuum \(\tau\)-function is given by
	\[
	\tau^+(x,t)
	=
	\det\left(\operatorname{Id}-\mathcal K^+\right).
	\]
	Moreover, it is found that
	\[
	|u(x,t)|^2
	=
	2i\,\partial_x
	\log
	\frac{\tau^+(x,t)}{\tau^-(x,t)},
	\qquad
	\tau^-(x,t)=\overline{\tau^+(x,t)}.
	\]
	Thus, in contrast to the standard NLS case, the \(\tau\)-function of the DNLS equation is generally complex-valued, and the density \(|u|^2\) is encoded in the \(x\)-variation of the phase of \(\tau^+\).
	\par
	The present work provides, to our knowledge, the first rigorous inverse-scattering construction of a soliton gas for the DNLS equation. Our analysis establishes a direct connection between the condensation of discrete soliton spectra and finite-gap geometry through a continuum \(\bar\partial\)-problem and its reduction to a mother-body RHP. A distinctive feature of the DNLS equation is the involution 
	\(z\mapsto -z\), which induces a quotient reduction of the underlying Abelian geometry and leads to a reduction of the effective genus of the associated hyperelliptic curve. These results reveal a new geometric mechanism underlying soliton gases beyond the standard AKNS-type formalism and provide new perspectives for the study of asymptotic behaviors of other soliton gases.

	\section{From discrete solitons to the continuum DNLS soliton gas}
	\label{Continuum Limit}
	
	In this section we construct the continuum soliton gas limit for the equation \eqref{dnls3}.  The construction has three steps.  First, we give the reflectionless RHP by inverse scattering analysis and the corresponding pure $N$-soliton solution.  Second, we remove the poles by enclosing them with fixed contours and pass to the limit in which the discrete spectrum condenses on two dimensional domains.  Under the $1/N$ scaling of the norming constants, the finite pole-removal problem converges to a compactly supported $\bar\partial$-problem.  Third, for condensation domains admitting a Schwarz function, we identify the exterior large-$z$ expansion, and hence the reconstructed potential.
	
	Throughout this section, for a set $A\subset\mathbb C$ we write
	$
	A^*:=\{\bar z:z\in A\}.
	$
	For a scalar or matrix-valued function $h$, the Schwarz conjugate is denoted by
	$
	h^*(z):=\overline{h(\bar z)},
	$
	with entrywise conjugation in the matrix case.
	
	\subsection{The reflectionless RHP}
	\label{subsec:basic-rhp}
	
	We do not rederive the inverse scattering transform here.  Instead, we record the RHP that serves as the starting point for the continuum limit.  Let the initial datum $u(x,0)=u_0(x)$ belong to the Schwartz class $\mathcal S(\mathbb R)$.  For the $x$-part of the Lax pair \eqref{eq:lax3} associated with \eqref{dnls3}, let
	\[
	\mathcal J_\pm(x,z)=\bigl(\mathcal J_\pm^{(1)}(x,z),\mathcal J_\pm^{(2)}(x,z)\bigr)
	\]
	be the Jost solutions normalized by
	$
	\mathcal J_\pm(x,z)e^{iz^2x\sigma_3}\to \mathbb I$ as $x\to\pm\infty$.
	They are related by the scattering matrix $S(z)$ as
	\[
	\mathcal J_+(x,z)=\mathcal J_-(x,z)S(z),
	\qquad z\in\mathbb R\cup i\mathbb R .
	\]
	The symmetries of the \eqref{dnls3} spectral problem imply
	\begin{equation}
		\label{eq:S-form}
		S(z)=
		\begin{pmatrix}
			a^*(z)&b(z)\\
			-b^*(z)&a(z)
		\end{pmatrix},
		\qquad
		S(z)=\sigma_2S^*(z)\sigma_2,
		\qquad
		S(z)=\sigma_3S(-z)\sigma_3 .
	\end{equation}
	Define the reflection coefficient by
	$
	r(z)=\frac{b(z)}{a(z)}.
	$
	We assume that $a(z)$ has $N$ simple zeros
	\[
	z_j\in\mathbb C_{\mathrm I}:=\{z\in\mathbb C:\operatorname{Re}z>0,\ \operatorname{Im}z>0\},
	\qquad j=1,\ldots,N,
	\]
	with $z_j\neq z_k$ for $j\neq k$ and $z_j\notin \mathbb R\cup i\mathbb R$.  By the spectral symmetries, the full discrete spectrum is
	$
	\{z_j,-z_j,\bar z_j,-\bar z_j\}_{j=1}^N .
	$
	Let the phase function be
	\begin{equation}
		\label{eq:theta}
		\theta(x,t,z)=z^2x+2z^4t .
	\end{equation}
	
	\begin{basicproblem}\label{prob:basic-rhp}
		Find a $2\times2$ matrix-valued function $M(x,t,z)$ with the following properties:
		\begin{enumerate}
			\item \textbf{Analyticity.}
			$M(x,t,z)$ is meromorphic for $z\in \mathbb C\setminus(\mathbb R\cup i\mathbb R)$ and has simple poles at
			$
			\{z_j,-z_j,\bar z_j,-\bar z_j\}_{j=1}^N .
			$
			
			\item \textbf{Jump condition.}
			For $z\in\mathbb R\cup i\mathbb R$, the following jump condition holds
			\begin{equation}
				\label{eq:basic-jump}
				M_+(x,t,z)=M_-(x,t,z)J(x,t,z),
			\end{equation}
			where
			\begin{equation}
				\label{eq:basic-jump-matrix}
				J(x,t,z)=
				\begin{pmatrix}
					1+|r(z)|^2&r^*(z)\,e^{-2i\theta(x,t,z)}\\
					r(z)e^{2i\theta(x,t,z)}&1
				\end{pmatrix}.
			\end{equation}
			
			\item \textbf{Residue conditions.}
			At the discrete spectrum, the residue conditions satisfy
			\begin{equation}
				\label{eq:basic-residues}
				\begin{aligned}
					\operatorname*{Res}_{z=z_j} M(x,t,z)
					&=
					\lim_{z\to z_j} M(x,t,z)
					\begin{pmatrix}
						0 & 0 \\
						c_j e^{2i\theta(x,t,z_j)} & 0
					\end{pmatrix},\\[1mm]
					\operatorname*{Res}_{z=\bar z_j} M(x,t,z)
					&=
					\lim_{z\to \bar z_j} M(x,t,z)
					\begin{pmatrix}
						0 & -\bar c_j e^{-2i\theta(x,t,\bar z_j)} \\
						0 & 0
					\end{pmatrix},\\[1mm]
					\operatorname*{Res}_{z=-z_j} M(x,t,z)
					&=
					-\sigma_3
					\bigl(\operatorname*{Res}_{z=z_j} M(x,t,z)\bigr)
					\sigma_3,
					\qquad
					\operatorname*{Res}_{z=-\bar z_j} M(x,t,z)
					=
					-\sigma_3
					\bigl(\operatorname*{Res}_{z=\bar z_j} M(x,t,z)\bigr)
					\sigma_3.
				\end{aligned}
			\end{equation}

			\item \textbf{Symmetries.}
			\begin{equation}
				\label{eq:basic-symmetry}
				M(x,t,z)=\sigma_2M^*(x,t,z)\sigma_2,
				\qquad
				M(x,t,z)=\sigma_3M(x,t,-z)\sigma_3 .
			\end{equation}
			
			\item \textbf{Normalization.}
			\begin{equation}
				\label{eq:basic-normalization}
				M(x,t,z)=\mathbb I+\mathcal{O}(z^{-1}),
				\qquad z\to\infty .
			\end{equation}
		\end{enumerate}
	\end{basicproblem}
	
	The solution of \eqref{dnls3} is reconstructed from the large-$z$ expansion by
	\begin{equation}
		\label{eq:reconstruction-basic}
		u(x,t)=2i\lim_{z\to\infty}z[M(x,t,z)]_{12}.
	\end{equation}
	
	We now restrict to the reflectionless case, namely 
	$
	r(z)\equiv0 .
	$
	Then the jump on $\mathbb R\cup i\mathbb R$ disappears, and Basic RHP~\ref{prob:basic-rhp} becomes a meromorphic RHP with prescribed simple poles.
	
	\begin{prop}
		\label{prop:pure-N-soliton}
		In the reflectionless case, the solution of Basic RHP~\ref{prob:basic-rhp} has the rational representation
		\begin{equation}
			\label{eq:soliton-ansatz}
			\begin{aligned}
				M^{(N)}(x,t,z)=\mathbb I
				&+\sum_{j=1}^N\frac{1}{z-z_j}
				\begin{pmatrix}
					p_j&0\\ q_j&0
				\end{pmatrix}
				+\sum_{j=1}^N\frac{1}{z-\bar z_j}
				\begin{pmatrix}
					0&-\bar q_j\\ 0&\bar p_j
				\end{pmatrix} \
				&+\sum_{j=1}^N\frac{1}{z+z_j}
				\begin{pmatrix}
					-p_j&0\\ q_j&0
				\end{pmatrix}
				+\sum_{j=1}^N\frac{1}{z+\bar z_j}
				\begin{pmatrix}
					0&-\bar q_j\\0&-\bar p_j
				\end{pmatrix},
			\end{aligned}
		\end{equation}
		where $p_j=p_j(x,t)$ and $q_j=q_j(x,t)$ are determined by the finite algebraic system
		\begin{subequations}
			\label{eq:finite-algebraic-system}
			\begin{align}
				\begin{pmatrix}p_j\\q_j\end{pmatrix}
				&=
				c_j e^{2i\theta(x,t,z_j)}
				\left[
				\begin{pmatrix}0\\1\end{pmatrix}
				+\sum_{k=1}^N
				\frac{1}{z_j-\bar z_k}
				\begin{pmatrix}
					-\bar q_k\\ \bar p_k
				\end{pmatrix}
				+\sum_{k=1}^N
				\frac{1}{z_j+\bar z_k}
				\begin{pmatrix}
					-\bar q_k\\-\bar p_k
				\end{pmatrix}
				\right],
				\\[1mm]
				\begin{pmatrix}-\bar q_j\\ \bar p_j\end{pmatrix}
				&=
				-\bar c_j e^{-2i\theta(x,t,\bar z_j)}
				\left[
				\begin{pmatrix}1\\0\end{pmatrix}
				+\sum_{k=1}^N
				\frac{1}{\bar z_j-z_k}
				\begin{pmatrix}
					p_k\\ q_k
				\end{pmatrix}
				+\sum_{k=1}^N
				\frac{1}{\bar z_j+z_k}
				\begin{pmatrix}
					-p_k\\ q_k
				\end{pmatrix}
				\right].
			\end{align}
		\end{subequations}
		Consequently, it follows that
		\begin{equation}
			\label{eq:N-soliton-reconstruction}
			u_N(x,t)=-4i\sum_{j=1}^N\bar q_j(x,t).
		\end{equation}
	\end{prop}
	
	\begin{proof}
		The representation \eqref{eq:soliton-ansatz} is the most general rational ansatz compatible with the pole set, the normalization, and the two symmetries in \eqref{eq:basic-symmetry}.  Substitution into the residue conditions \eqref{eq:basic-residues} gives \eqref{eq:finite-algebraic-system}.  Expanding \eqref{eq:soliton-ansatz} at infinity and using \eqref{eq:reconstruction-basic} gives \eqref{eq:N-soliton-reconstruction}.
	\end{proof}
	In particular, when \(N=1\), the pure soliton formula (\ref{eq:N-soliton-reconstruction}) reduces to the one-soliton solution
	\[
	u_1(x,t)
	=
	\frac{4\eta\omega}{\eta+i\omega}\,
	\operatorname{sech}(\Xi)\,
	\exp\left[-i\left(2\psi+\phi+\frac{\pi}{2}\right)\right],
	\]
	where \(z=\eta+i\omega\) with \(\eta,\omega\in\mathbb R\), and \(\phi\in\mathbb R\) is the phase parameter. Here
	\[
	\psi
	=
	(\eta^2-\omega^2)x
	+
	2(\eta^4-6\eta^2\omega^2+\omega^4)t,
	\quad
	\Xi
	=
	-4\eta\omega\bigl[x+4(\eta^2-\omega^2)t\bigr]
	+
	\ln \left(\frac{|c_1|(\eta+i\omega)}{2\eta\omega}\right).
	\]
	The center of the soliton is determined by \(\Xi=\mathrm{constant}\), and hence its propagation velocity is
	\begin{equation}\label{eq:free-soliton-velocity-a}
		v_{\mathrm{sol}}
		=
		-4(\eta^2-\omega^2)
		=
		-4\operatorname{Re}(z^2).
	\end{equation}

	\subsection{From pole removal to the continuum \texorpdfstring{$\bar\partial$}{dbar}-problem}
	\label{subsec:pole-removal}
	
	We now convert the reflectionless meromorphic RHP into a pole-free RHP on fixed enclosing contours.  Since $N$ will vary, the discrete data are denoted by
	\[
	\{z_{j,N},c_{j,N}\}_{j=1}^N,
	\]
	although the subscript $N$ will be suppressed when no confusion is possible.
	Let $\gamma_+$ be a simple positively oriented contour enclosing the poles $\{z_j\}_{j=1}^N$ in the first quadrant.  Set
	\[
	\gamma_-:=-\gamma_+,\qquad
	\gamma_+^*:=\overline{\gamma_+},\qquad
	\gamma_-^*:=-\gamma_+^*,
	\]
	each with the positive orientation with respect to its bounded interior.  We define the pole-removal contour
	\[
	\Gamma_{\rm p}:=\gamma_+\cup\gamma_-\cup\gamma_+^*\cup\gamma_-^* .
	\]
	On $\gamma_+$ the symbol $z\mp z_j$ means $z-z_j$, while on $\gamma_-$ it means $z+z_j$; the same convention is used for $z\mp\bar z_j$.
	
	Define the triangular matrix $T^{(N)}(x,t,z)$ by
	\[
	T^{(N)}(x,t,z)=
	\begin{cases}
		\begin{pmatrix}
			1&0\\[1mm]
			-\displaystyle\sum_{j=1}^N
			\frac{c_{j,N}e^{2i\theta(x,t,z_j)}}{z\mp z_j}
			&1
		\end{pmatrix},
		& z \text{ inside } \gamma_\pm,\\[7mm]
		\begin{pmatrix}
			1&
			\displaystyle\sum_{j=1}^N
			\frac{\bar c_{j,N}e^{-2i\theta(x,t,\bar z_j)}}{z\mp\bar z_j}\\[1mm]
			0&1
		\end{pmatrix},
		& z \text{ inside } \gamma_\pm^*,\\[7mm]
		\mathbb I,
		& z \text{ outside all four contours}.
	\end{cases}
	\]
	If $M^{(N)}(x,t,z)$ denotes the reflectionless meromorphic solution, then
	\[
	M_{\rm p}^{(N)}(x,t,z):=M^{(N)}(x,t,z)T^{(N)}(x,t,z)
	\]
	is analytic in $\mathbb C\setminus\Gamma_{\rm p}$ and satisfies
	\[
	M_{{\rm p},+}^{(N)}(x,t,z)=M_{{\rm p},-}^{(N)}(x,t,z)J_{\rm p}^{(N)}(x,t,z),
	\qquad z\in\Gamma_{\rm p},
	\]
	where
	\begin{equation}\label{discrete-jump}
		J_{\rm p}^{(N)}(x,t,z)=\mathbb I+W_N(x,t,z),
	\end{equation}
	and
	\begin{equation}
		\label{eq:WN}
		W_N(x,t,z)=
		\begin{pmatrix}
			0&
			\displaystyle
			\sum_{j=1}^N
			\frac{\bar c_{j,N}e^{-2i\theta(x,t,\bar z_j)}}{z\mp\bar z_j}
			\chi_{\gamma_\pm^*}(z)
			\\[5mm]
			\displaystyle
			-\sum_{j=1}^N
			\frac{c_{j,N}e^{2i\theta(x,t,z_j)}}{z\mp z_j}
			\chi_{\gamma_\pm}(z)
			&0
		\end{pmatrix}.
	\end{equation}
	
	We now specify the condensation regime in the limit $N\to\infty$.  For each $N$, let
	$
	\{z_{j,N}\}_{j=1}^N\subset \mathbb C_{\mathrm I}
	$
	be the discrete spectrum in the first quadrant.  We assume that, as $N\to\infty$, these points condense onto a bounded simply connected domain $\mathcal D_+$ satisfying
	$
	\operatorname{cl}(\mathcal D_+)\Subset \mathbb C_{\mathrm I}.
	$
	The remaining condensation domains are generated by the spectral symmetries:
	\[
	\mathcal D_-:=-\mathcal D_+,\qquad
	\mathcal D_+^*:=\overline{\mathcal D_+},\qquad
	\mathcal D_-^*:=-\mathcal D_+^* .
	\]
	Let $\mathcal A:=|\mathcal D_+|$ be the area of $\mathcal D_+$.  The pole-removal contours are chosen independently of $N$ so that, if $D_\gamma$ denotes the bounded interior of a contour $\gamma$, then
	\[
	\overline{\mathcal D_\pm}\Subset D_{\gamma_\pm},
	\qquad
	\overline{\mathcal D_\pm^*}\Subset D_{\gamma_\pm^*}.
	\]
	Thus the domains stay a positive distance away from the pole-removal contours.  In particular, for all sufficiently large $N$, the discrete spectra $\{z_{j,N}\}$, together with their symmetry images, are enclosed by the corresponding contours.

	\begin{assumption}\label{ass:spectral-condensation}
		The empirical measures
		$
		\mu_N:=\frac1N\sum_{j=1}^N\delta_{z_{j,N}}
		$
		converge weakly to the normalized area measure
		$
		\mu:=\frac{\chi_{\mathcal D_+}(\lambda)}{\mathcal A}d^2\lambda .
		$
		The norming constants satisfy
		\begin{equation}
			\label{eq:norming-scaling}
			c_{j,N}=\frac{\mathcal A}{\pi N}f(z_{j,N},\bar z_{j,N}),
			\qquad j=1,\ldots,N,
		\end{equation}
		where the function $f\in C^\alpha(\operatorname{cl}(\mathcal D_+))$ for some $0<\alpha<1$.  The density function can be extended to the symmetry-related domains by
		\[
		f(-z,-\bar z)=f(z,\bar z),
		\qquad
		f^*(z,\bar z):=\overline{f(\bar z,z)}
		\quad\text{on }\mathcal D_+^*\cup\mathcal D_-^* .
		\]
	\end{assumption}
	
	\begin{lemma}
		\label{lem:Cauchy-sum-limit}
		Let $K\Subset\mathbb R^2$ be compact.  Under Assumption~\ref{ass:spectral-condensation}, for every compact set $V\Subset\mathbb C\setminus\operatorname{cl}(\mathcal D_+)$, one has
		\[
		\sum_{j=1}^N
		\frac{c_{j,N}e^{2i\theta(x,t,z_{j,N})}}{z-z_{j,N}}
		\longrightarrow
		\frac1\pi
		\iint_{\mathcal D_+}
		\frac{f(\lambda,\bar\lambda)e^{2i\theta(x,t,\lambda)}}{z-\lambda}d^2\lambda
		\]
		uniformly for $(x,t,z)\in K\times V$.  The analogous convergence holds in the other three symmetry-related domains.
	\end{lemma}
	
	\begin{proof}
		For $(x,t,z)\in K\times V$, the functions
		$
		\lambda\mapsto
		\frac{f(\lambda,\bar\lambda)e^{2i\theta(x,t,\lambda)}}{z-\lambda}
		$
		form a bounded and equicontinuous family on $\operatorname{cl}(\mathcal D_+)$.  By \eqref{eq:norming-scaling}, we have
		\[
		\sum_{j=1}^N
		\frac{c_{j,N}e^{2i\theta(x,t,z_{j,N})}}{z-z_{j,N}}
		=
		\frac{\mathcal A}{\pi}
		\int_{\mathbb C}
		\frac{f(\lambda,\bar\lambda)e^{2i\theta(x,t,\lambda)}}{z-\lambda}d\mu_N(\lambda).
		\]
		The weak convergence $\mu_N\rightharpoonup\mu$, together with compactness of the above family in $C(\operatorname{cl}\mathcal D_+)$, gives the claimed locally uniform convergence.  The remaining cases follow by the transformations $z\mapsto -z$ and $z\mapsto\bar z$.
	\end{proof}
	
	By Lemma~\ref{lem:Cauchy-sum-limit}, the jump matrices \eqref{discrete-jump} on the fixed contour $\Gamma_{\rm p}$ converge to
	\[
	J_{\rm p}^{(\infty)}(x,t,z)=\mathbb I+W(x,t,z),
	\]
	where
	\begin{equation}
		\label{eq:W-limit}
		W(x,t,z)=
		\begin{pmatrix}
			0&
			\displaystyle
			\frac{\chi_{\gamma_\pm^*}(z)}{\pi}
			\iint_{\mathcal D_\pm^*}
			\frac{f^*(\lambda,\bar\lambda)e^{-2i\theta(x,t,\lambda)}}{z-\lambda}
			d^2\lambda 
			\\[3mm]
			\displaystyle
			-\frac{\chi_{\gamma_\pm}(z)}{\pi}
			\iint_{\mathcal D_\pm}
			\frac{f(\lambda,\bar\lambda)e^{2i\theta(x,t,\lambda)}}{z-\lambda}
			d^2\lambda
			&0
		\end{pmatrix}.
	\end{equation}
	Let $M_{\rm p}^{(\infty)}$ denote the solution of the limiting RHP
	\[
	M_{{\rm p},+}^{(\infty)}=
	M_{{\rm p},-}^{(\infty)}
	(\mathbb I+W),
	\qquad
	M_{\rm p}^{(\infty)}(z)=\mathbb I+\mathcal{O}(z^{-1}),
	\qquad z\to\infty.
	\]
	The jump conditions can be removed in the limit.  Define
	\[
	T^{(\infty)}(x,t,z)=
	\begin{cases}
		\begin{pmatrix}
			1&0\\[2mm]
			\displaystyle
			\frac1\pi
			\iint_{\mathcal D_\pm}
			\frac{f(\lambda,\bar\lambda)e^{2i\theta(x,t,\lambda)}}{z-\lambda}d^2\lambda
			&1
		\end{pmatrix},
		& z \text{ inside } \gamma_\pm,\\[7mm]
		\begin{pmatrix}
			1&
			\displaystyle
			-\frac1\pi
			\iint_{\mathcal D_\pm^*}
			\frac{f^*(\lambda,\bar\lambda)e^{-2i\theta(x,t,\lambda)}}{z-\lambda}d^2\lambda
			\\[2mm]
			0&1
		\end{pmatrix},
		& z \text{ inside } \gamma_\pm^*,\\[7mm]
		\mathbb I,
		& z \text{ outside all four contours}.
	\end{cases}
	\]
	Then $T^{(\infty)}$ has the inverse jump to $\mathbb I+W$ on $\Gamma_{\rm p}$, and therefore
	\[
	M_2(x,t,z):=M_{\rm p}^{(\infty)}(x,t,z)T^{(\infty)}(x,t,z)
	\]
	has no jump across $\Gamma_{\rm p}$.  However, it is no longer analytic.  Using the Cauchy--Green identity, one obtains
	\[
	\partial_{\bar z}T^{(\infty)}(x,t,z)=T^{(\infty)}(x,t,z)N(x,t,z),
	\]
	where
	\begin{equation}
		\label{eq:dbar-matrix-N}
		N(x,t,z)=
		\begin{pmatrix}
			0&
			-f^*(z,\bar z)e^{-2i\theta(x,t,z)}
			\chi_{\mathcal D_\pm^*}(z)
			\\[1mm]
			f(z,\bar z)e^{2i\theta(x,t,z)}
			\chi_{\mathcal D_\pm}(z)
			&0
		\end{pmatrix}.
	\end{equation}
	Thus $M_2(x,t,z)$ satisfies the following compactly supported $\bar\partial$-problem.
	
	\begin{parproblem}
		\label{prob:dbar}
		Find a $2\times2$ matrix-valued function $M_2(x,t,z)$ such that
		\begin{equation}
			\label{eq:dbar-problem}
			\partial_{\bar z}M_2(x,t,z)=M_2(x,t,z)N(x,t,z),
		\end{equation}
		with $N(x,t,z)$ given by \eqref{eq:dbar-matrix-N}, and such that
		\[
		M_2(x,t,z)=\mathbb I+\mathcal{O}(z^{-1}),
		\qquad z\to\infty.
		\]
		Moreover, we have
		\[
		M_2(x,t,z)=\sigma_2M_2^*(x,t,z)\sigma_2,
		\qquad
		M_2(x,t,z)=\sigma_3M_2(x,t,-z)\sigma_3 .
		\]
	\end{parproblem}
	
	The continuum potential is reconstructed by
	\begin{equation}
		\label{eq:reconstruction-dbar}
		u(x,t)=2i\lim_{z\to\infty}z[M_2(x,t,z)]_{12}.
	\end{equation}
	
	\begin{theorem}
		\label{thm:finite-to-dbar}
		Let $K\Subset\mathbb R^2$ be compact.  Assume that Assumption~\ref{ass:spectral-condensation} holds and that
		\[
		\varepsilon_N(K):=
		\sup_{(x,t)\in K}
		\|W_N(x,t,\cdot)-W(x,t,\cdot)\|_{L^2(\Gamma_{\rm p})\cap L^\infty(\Gamma_{\rm p})}
		\longrightarrow0 .
		\]
		For each $(x,t)\in K$, let
		$
		\mathcal C_{W(x,t)}h:=\mathcal C_-\bigl(hW(x,t,\cdot)\bigr)
		$
		be the Beals--Coifman operator on $L^2(\Gamma_{\rm p})$.  Suppose that
		\[
		\sup_{(x,t)\in K}
		\bigl\|(\mathbb I-\mathcal C_{W(x,t)})^{-1}\bigr\|_{L^2(\Gamma_{\rm p})\to L^2(\Gamma_{\rm p})}
		<\infty .
		\]
		Then, for all sufficiently large $N$, the finite-$N$ pole-removal RHP is uniquely solvable.  Moreover, if $V\Subset\mathbb C\setminus\Gamma_{\rm p}$, then
		\[
		\sup_{(x,t)\in K}\sup_{z\in V}
		\bigl\|
		M_{\rm p}^{(N)}(x,t,z)-M_{\rm p}^{(\infty)}(x,t,z)
		\bigr\|
		=
		\mathcal{O}(\varepsilon_N(K)).
		\]
		The function
		$
		M_2=M_{\rm p}^{(\infty)}T^{(\infty)}
		$
		solves $\bar{\partial}$-problem~\ref{prob:dbar}.  Consequently, with
		\[
		u_N(x,t):=2i\lim_{z\to\infty}z[M_{\rm p}^{(N)}(x,t,z)]_{12},
		\qquad
		u(x,t):=2i\lim_{z\to\infty}z[M_2(x,t,z)]_{12},
		\]
		one has
		$
		u_N\to u
		$
		uniformly on $K$.
	\end{theorem}
	
	\begin{proof}
		After pole removal, the finite inverse problem has jump $\mathbb I+W_N$ on $\Gamma_{\rm p}$, while the limiting contour problem has jump $\mathbb I+W$. By assumption,
		\[
		W_N\to W
		\quad\text{in }L^2(\Gamma_{\rm p})\cap L^\infty(\Gamma_{\rm p}),
		\]
		uniformly for $(x,t)\in K$. Since $\mathcal C_-$ is bounded on $L^2(\Gamma_{\rm p})$, we have
		\[
		\|\mathcal C_{W_N}-\mathcal C_W\|_{L^2(\Gamma_{\rm p})\to L^2(\Gamma_{\rm p})}
		=
		\mathcal O(\varepsilon_N(K)).
		\]
		The uniform invertibility of $\mathbb I-\mathcal C_W$ and the Neumann perturbation theorem imply that
		$\mathbb I-\mathcal C_{W_N}$ is invertible for all sufficiently large $N$, uniformly for $(x,t)\in K$.
		
		Let $\mu_N$ and $\mu$ be the corresponding Beals--Coifman densities:
		\[
		(\mathbb I-\mathcal C_{W_N})\mu_N=\mathbb I,
		\qquad
		(\mathbb I-\mathcal C_W)\mu=\mathbb I .
		\]
		By the resolvent identity, one has
		\[
		\mu_N-\mu
		=
		(\mathbb I-\mathcal C_{W_N})^{-1}
		(\mathcal C_{W_N}-\mathcal C_W)\mu .
		\]
		Hence
		\[
		\sup_{(x,t)\in K}
		\|\mu_N-\mu\|_{L^2(\Gamma_{\rm p})}
		=
		\mathcal O(\varepsilon_N(K)).
		\]
		
		Using the Cauchy representations
		\[
		M_{\rm p}^{(N)}(z)
		=
		\mathbb I+\frac1{2\pi i}
		\int_{\Gamma_{\rm p}}
		\frac{\mu_N(s)W_N(s)}{s-z}\,ds,
		\qquad
		M_{\rm p}^{(\infty)}(z)
		=
		\mathbb I+\frac1{2\pi i}
		\int_{\Gamma_{\rm p}}
		\frac{\mu(s)W(s)}{s-z}\,ds,
		\]
		we get
		\[
		M_{\rm p}^{(N)}(z)-M_{\rm p}^{(\infty)}(z)
		=
		\frac1{2\pi i}
		\int_{\Gamma_{\rm p}}
		\frac{(\mu_N(s)-\mu(s))W_N(s)+\mu(s)(W_N(s)-W(s))}{s-z}\,ds .
		\]
		If $V\Subset \mathbb C\setminus\Gamma_{\rm p}$, then $\operatorname{dist}(V,\Gamma_{\rm p})>0$. By the Cauchy--Schwarz inequality, the uniform bounds on $\mu$, $\mu_N$, and $W_N$, and the estimates above, it follows that
		\[
		\sup_{(x,t)\in K}\sup_{z\in V}
		\left\|
		M_{\rm p}^{(N)}(x,t,z)-M_{\rm p}^{(\infty)}(x,t,z)
		\right\|
		=
		\mathcal O(\varepsilon_N(K)).
		\]
		
		By construction, $T^{(\infty)}$ has the inverse jump to $\mathbb I+W$ on $\Gamma_{\rm p}$. Thus
		$
		M_2=M_{\rm p}^{(\infty)}T^{(\infty)}
		$
		has no jump across $\Gamma_{\rm p}$. Moreover,
		$
		\partial_{\bar z}M_2
		=
		M_{\rm p}^{(\infty)}\partial_{\bar z}T^{(\infty)}
		=
		M_{\rm p}^{(\infty)}T^{(\infty)}N
		=
		M_2N.
		$
		The normalization follows since $T^{(\infty)}=\mathbb I$ near infinity. Therefore $M_2$ solves $\bar{\partial}$-problem~\ref{prob:dbar}.
		
		It remains to prove the convergence of the reconstructed potentials. Expanding the Cauchy representations at infinity gives
		\[
		M_{\rm p}^{(N)}(z)
		=
		\mathbb I+\frac{M_{{\rm p},1}^{(N)}}{z}
		+\mathcal O(z^{-2}),
		\qquad
		M_{\rm p}^{(\infty)}(z)
		=
		\mathbb I+\frac{M_{{\rm p},1}^{(\infty)}}{z}
		+\mathcal O(z^{-2}),
		\]
		thus
		\[
		M_{{\rm p},1}^{(N)}-M_{{\rm p},1}^{(\infty)}=
		-\frac1{2\pi i}
		\int_{\Gamma_{\rm p}}
		\bigl((\mu_N(s)-\mu(s))W_N(s)+\mu(s)(W_N(s)-W(s))\bigr)\,ds .
		\]
		Again by the Cauchy--Schwarz inequality, we have
		\[
		\sup_{(x,t)\in K}
		\left\|
		M_{{\rm p},1}^{(N)}(x,t)
		-
		M_{{\rm p},1}^{(\infty)}(x,t)
		\right\|
		=
		\mathcal O(\varepsilon_N(K)).
		\]
		Since $T^{(\infty)}(z)=\mathbb I$ in a neighbourhood of $z=\infty$, the functions
		$M_2=M_{\rm p}^{(\infty)}T^{(\infty)}$ and $M_{\rm p}^{(\infty)}$ have the same $z^{-1}$ coefficient at infinity. Since the pole-removal transformations are also identity near infinity, the reconstruction formulas are
		\[
		u_N(x,t)
		=
		2i\,[M_{{\rm p},1}^{(N)}(x,t)]_{12},
		\qquad
		u(x,t)
		=
		2i\,[M_{{\rm p},1}^{(\infty)}(x,t)]_{12}.
		\]
		Then we conclude that
		$
		u_N(x,t)\to u(x,t)
		$
		uniformly on $K$.
	\end{proof}

	\subsection{Schwarz function reduction to the mother body}
	\label{subsec:schwarz-function}
	
	We now describe the reduction of the exterior large-$z$ expansion of $\bar{\partial}$-problem~\ref{prob:dbar} to a RHP. 
	For the Schwarz function reduction we impose the following additional regularity: in the upper domains $\mathcal D_\pm$, the density function $f$ is the restriction of a function holomorphic in a neighbourhood of $\operatorname{cl}(\mathcal D_\pm)$; in the conjugate domains the density is given by Schwarz conjugation.
	
	Write
	$
	M_2(x,t,z)=\bigl(M_{2,1}(x,t,z),M_{2,2}(x,t,z)\bigr)
	$
	in terms of its column vectors.  From \eqref{eq:dbar-problem} we obtain
	\begin{equation}
		\label{eq:dbar-decoupled}
		\begin{aligned}
			\partial_{\bar z}M_{2,1}
			&=f(z)e^{2i\theta(x,t,z)}M_{2,2},
			&
			\partial_{\bar z}M_{2,2}&=0,
			&& z\in\mathcal D_\pm,\\
			\partial_{\bar z}M_{2,1}
			&=0,
			&
			\partial_{\bar z}M_{2,2}
			&=-f^*(z)e^{-2i\theta(x,t,z)}M_{2,1},
			&& z\in\mathcal D_\pm^* .
		\end{aligned}
	\end{equation}
	The Cauchy--Green formula gives, for exterior $z$,
	\begin{equation}
		\label{eq:columns-CG}
		\begin{aligned}
			M_{2,1}(x,t,z)
			&=
			\begin{pmatrix}1\\0\end{pmatrix}
			-
			\iint_{\mathcal D_\pm}
			\frac{
				M_{2,2}(x,t,\lambda)
				f(\lambda)e^{2i\theta(x,t,\lambda)}
			}{\lambda-z}
			\frac{d\bar\lambda\wedge d\lambda}{2\pi i},
			\\[1mm]
			M_{2,2}(x,t,z)
			&=
			\begin{pmatrix}0\\1\end{pmatrix}
			+
			\iint_{\mathcal D_\pm^*}
			\frac{
				M_{2,1}(x,t,\lambda)
				f^*(\lambda)e^{-2i\theta(x,t,\lambda)}
			}{\lambda-z}
			\frac{d\bar\lambda\wedge d\lambda}{2\pi i}.
		\end{aligned}
	\end{equation}
	Assume that $\mathcal D_+$ is a simply connected quadrature domain whose boundary admits a Schwarz function $\mathcal S(z)$:
	\[
	\bar z=\mathcal S(z),\qquad z\in\partial\mathcal D_+ .
	\]
	Let $\mathscr L_+$ be the singularity set of the analytic continuation of $\mathcal S$ inside $\mathcal D_+$, called the mother body.  Set
	\[
	\mathscr L_-:=-\mathscr L_+,\qquad
	\mathscr L_+^*:=\overline{\mathscr L_+},\qquad
	\mathscr L_-^*:=-\mathscr L_+^*,
	\]
	and
	$
	\Sigma_{\rm mb}:=
	\mathscr L_+\cup\mathscr L_-\cup\mathscr L_+^*\cup\mathscr L_-^* .
	$
	
	Let
	$
	\Delta\mathcal S(z):=\mathcal S_-(z)-\mathcal S_+(z)
	$
	denote the jump of the Schwarz function across the oriented mother body. Define the density function on the mother body by
	\begin{equation}
		\label{eq:mother-density}
		\varrho(z):=f(z)\Delta\mathcal S(z),
		\qquad z\in\mathscr L_+\cup\mathscr L_-,
	\end{equation}
	and define $\varrho^*$ on $\mathscr L_+^*\cup\mathscr L_-^*$ by Schwarz conjugation.
	
	\begin{prop}
		\label{prop:mother-body-reduction}
		Under the assumptions above, the coefficient of $z^{-1}$ in the exterior large-$z$ expansion of the solution of $\bar{\partial}$-problem~\ref{prob:dbar} coincides with the coefficient of $z^{-1}$ generated by the following RHP on $\Sigma_{\rm mb}$:
		\[
		\widetilde m_+(x,t,z)
		=
		\widetilde m_-(x,t,z)\widetilde J(x,t,z),
		\qquad
		\widetilde m(x,t,z)=\mathbb I+O(z^{-1}),\quad z\to\infty,
		\]
		where
		\begin{equation}
			\label{eq:mother-jump}
			\widetilde J(x,t,z)
			=
			\mathbb I
			+
			e^{-i\theta(x,t,z)\sigma_3}
			\begin{pmatrix}
				0&\varrho^*(z)\chi_{\mathscr L_+^*\cup\mathscr L_-^*}(z)\\
				-\varrho(z)\chi_{\mathscr L_+\cup\mathscr L_-}(z)&0
			\end{pmatrix}
			e^{i\theta(x,t,z)\sigma_3}.
		\end{equation}
		Consequently, the potential function reconstructed from this RHP agrees with the potential function reconstructed from the exterior expansion of the $\bar{\partial}$-problem~\ref{prob:dbar}.
	\end{prop}
	
	\begin{proof}
		We give the argument for $\mathcal D_+$; the other domains follow by symmetry.  Since $M_{2,2}(\lambda)f(\lambda)e^{2i\theta(\lambda)}$ is analytic in $\mathcal D_+\setminus\mathscr L_+$, Stokes' theorem applied to the one-form
		\[
		\frac{M_{2,2}(\lambda)f(\lambda)e^{2i\theta(\lambda)}\bar\lambda}
		{\lambda-z}d\lambda
		\]
		gives, for exterior $z$,
		\[
		\iint_{\mathcal D_+}
		\frac{M_{2,2}(\lambda)f(\lambda)e^{2i\theta(\lambda)}}{\lambda-z}
		\frac{d\bar\lambda\wedge d\lambda}{2\pi i}
		=
		\int_{\partial\mathcal D_+}
		\frac{M_{2,2}(\lambda)f(\lambda)\mathcal S(\lambda)e^{2i\theta(\lambda)}}{\lambda-z}
		\frac{d\lambda}{2\pi i}.
		\]
		The analytic continuation of the Schwarz function $\mathcal S(z)$ is single-valued away from the mother body.  Deforming the boundary to the two sides of $\mathscr L_+$ therefore gives
		\[
		\int_{\partial\mathcal D_+}
		\frac{M_{2,2}f\mathcal Se^{2i\theta}}{\lambda-z}
		\frac{d\lambda}{2\pi i}
		=
		\int_{\mathscr L_+}
		\frac{M_{2,2}(\lambda)f(\lambda)\Delta\mathcal S(\lambda)e^{2i\theta(\lambda)}}{\lambda-z}
		\frac{d\lambda}{2\pi i}.
		\]
		The formula in the conjugate domain can also be obtained in the same way.  Hence the exterior Cauchy representation \eqref{eq:columns-CG} has the same large-$z$ expansion as the Cauchy singular integral equation associated with the jump matrix \eqref{eq:mother-jump}.  The equality of the reconstruction coefficients follows from \eqref{eq:reconstruction-dbar}.
	\end{proof}

	\subsection{The elliptic model and its reduced RHP}
	\label{subsec:elliptic-model}
	
	We now specialize to the elliptic domain used in the large-$x$ and long-time asymptotic analysis.  Let
	\[
	E_1=a+ib_1,\qquad E_2=a+ib_2,
	\qquad a>0,\quad 0<b_1<b_2.
	\]
	Put
	$
	b_0:=\frac{b_1+b_2}{2},
	\ 
	\ell:=\frac{b_2-b_1}{2}.
	$
	Choose $\rho>\ell$ and impose
	$
	a>\sqrt{\rho^2-\ell^2},
	\ 
	b_0>\rho .
	$
	These inequalities ensure that the ellipse below lies strictly in the first quadrant.  The domain $\mathcal D_+$ is the ellipse
	\begin{equation}
		\label{eq:ellipse}
		\frac{(x-a)^2}{\rho^2-\ell^2}
		+
		\frac{(y-b_0)^2}{\rho^2}
		<1,
		\qquad z=x+iy.
	\end{equation}
	Equivalently, its boundary is characterized by
	\[
	\sqrt{(x-a)^2+(y-b_1)^2}
	+
	\sqrt{(x-a)^2+(y-b_2)^2}
	=2\rho .
	\]
	The Schwarz function of this ellipse is
	\begin{equation}
		\label{eq:ellipse-Schwarz}
		\mathcal S_{\rm ell}(z)
		=
		\frac{2\rho^2}{\ell^2}(a+ib_0-z)
		+
		(z-2ib_0)
		+
		2\frac{\rho}{\ell^2}\sqrt{\rho^2-\ell^2}
		\sqrt{(z-E_1)(z-E_2)} .
	\end{equation}
	One checks directly from \eqref{eq:ellipse} that $\mathcal S_{\rm ell}(z)=\bar z$ on $\partial\mathcal D_+$.  The branch of the square root is chosen so that $\mathcal S_{\rm ell}$ is analytic outside the segment
	$
	\mathscr I:=[E_1,E_2],
	$
	which is oriented from $E_1$ to $E_2$.  This segment is the mother body of the ellipse.  With this orientation, we have
	\begin{equation}
		\label{eq:ellipse-Schwarz-jump}
		\Delta\mathcal S_{\rm ell}(z)
		=
		4\frac{\rho}{\ell^2}\sqrt{\rho^2-\ell^2}
		\sqrt{(z-E_1)(z-E_2)},
		\qquad z\in\mathscr I .
	\end{equation}
	
	The symmetry-related mother-body segments are
	$
	-\mathscr I,\  \mathscr I^*,\  -\mathscr I^*,
	$
	with orientations inherited from the positively oriented boundaries of the corresponding ellipses.  We set
	\begin{equation}\label{jump-contours-ell}
		\Sigma_{\rm ell}:=
		\mathscr I\cup(-\mathscr I)\cup\mathscr I^*\cup(-\mathscr I^*) .
	\end{equation}
	Then the density function is
	\begin{equation}
		\label{eq:elliptic-density}
		\varrho(z):=f(z)\Delta\mathcal S_{\rm ell}(z),
		\qquad z\in\mathscr I\cup(-\mathscr I),
	\end{equation}
	and $\varrho^*(z)=\overline{\varrho(\bar z)}$ on $\mathscr I^*\cup(-\mathscr I^*)$.

	\begin{figure}[ht]
		\centering
		\begin{tikzpicture}[scale=1.5]
			
			\colorlet{cmpaxis}{black!75}
			\colorlet{cmpgrid}{black!7}
			\colorlet{cmpdomainfill}{blue!6}
			\colorlet{cmpdomainedge}{blue!45!black}
			\colorlet{cmpsegment}{red!65!black}
			\colorlet{cmppoint}{black}
			\colorlet{cmplabel}{black!90}
			\colorlet{cmpsublabel}{red!55!black}
			\colorlet{cmpdomainlabel}{blue!45!black}
			
			\tikzset{
				axis/.style={->, line width=0.65pt, cmpaxis},
				gridline/.style={step=0.5, cmpgrid, very thin},
				ellregion/.style={fill=cmpdomainfill, draw=cmpdomainedge, line width=1.05pt},
				segment/.style={cmpsegment, line width=1.25pt},
				point/.style={cmppoint},
				lbl/.style={cmplabel, font=\small},
				sublbl/.style={cmpsublabel, font=\small},
				domlbl/.style={cmpdomainlabel, font=\small},
			}
			
			\begin{scope}
				\draw[gridline] (-1.9,-1.9) grid (1.9,1.9);
				
				\draw[axis] (-2,0) -- (2,0);
				\draw[axis] (0,-2) -- (0,2);
				
				\node[below left, lbl] at (0,0) {$0$};
				\node[below, lbl] at (2,0) {$\mathrm{Re}(z)$};
				\node[right, lbl] at (0,2) {$\mathrm{Im}(z)$};
				
				\draw[ellregion] (1,0.9) ellipse (0.5 cm and 0.8 cm);
				\draw[ellregion] (-1,-0.9) ellipse (0.5 cm and 0.8 cm);
				\draw[ellregion] (1,-0.9) ellipse (0.5 cm and 0.8 cm);
				\draw[ellregion] (-1,0.9) ellipse (0.5 cm and 0.8 cm);
				
				\draw[segment] (1,0.5)--(1,1.3);
				\draw[segment] (1,-0.5)--(1,-1.3);
				\draw[segment] (-1,0.5)--(-1,1.3);
				\draw[segment] (-1,-0.5)--(-1,-1.3);
				
				\filldraw[point] (1,0.5) circle (0.8pt);
				\node[below, lbl] at (1,0.5) {${\small E_1}$};
				
				\filldraw[point] (1,1.3) circle (0.8pt);
				\node[above, lbl] at (1,1.3) {$E_2$};
				
				\filldraw[point] (-1,0.5) circle (0.8pt);
				\node[below, lbl] at (-1,0.5) {$-\bar E_1$};
				
				\filldraw[point] (-1,1.3) circle (0.8pt);
				\node[above, lbl] at (-1,1.3) {$-\bar E_2$};
				
				\filldraw[point] (-1,-0.5) circle (0.8pt);
				\node[above, lbl] at (-1,-0.5) {$-E_1$};
				
				\filldraw[point] (-1,-1.3) circle (0.8pt);
				\node[below, lbl] at (-1,-1.3) {$-E_2$};
				
				\filldraw[point] (1,-0.5) circle (0.8pt);
				\node[above, lbl] at (1,-0.5) {$\bar E_1$};
				
				\filldraw[point] (1,-1.3) circle (0.8pt);
				\node[below, lbl] at (1,-1.3) {$\bar E_2$};
				
				\node[domlbl] at (1.73,1.0) {$\mathcal D_+$};
				\node[domlbl] at (-1.73,-1.0) {$\mathcal D_-$};
				\node[domlbl] at (1.73,-1.0) {$\mathcal D_+^*$};
				\node[domlbl] at (-1.73,1.0) {$\mathcal D_-^*$};
				
				\node[below, lbl] at (0,-2) {(a)};
			\end{scope}
			
			\begin{scope}
				\draw[gridline] (3,-1.9) grid (7,1.9);
				
				\draw[axis] (3,0) -- (7,0);
				\draw[axis] (5,-2) -- (5,2);
				
				\node[below left, lbl] at (5,0) {$0$};
				\node[below, lbl] at (7,0) {$\mathrm{Re}(z)$};
				\node[right, lbl] at (5,2) {$\mathrm{Im}(z)$};
				
				\draw[
				segment,
				postaction={decorate},
				decoration={markings, mark=at position 0.55 with {\arrow{Stealth}}}
				] (6,0.5)--(6,1.3);
				
				\draw[
				segment,
				postaction={decorate},
				decoration={markings, mark=at position 0.55 with {\arrow{Stealth}}}
				] (6,-1.3)--(6,-0.5);
				
				\draw[
				segment,
				postaction={decorate},
				decoration={markings, mark=at position 0.55 with {\arrow{Stealth}}}
				] (4,0.5)--(4,1.3);
				
				\draw[
				segment,
				postaction={decorate},
				decoration={markings, mark=at position 0.55 with {\arrow{Stealth}}}
				] (4,-1.3)--(4,-0.5);
				
				\filldraw[point] (6,0.5) circle (0.8pt);
				\node[below, lbl] at (6,0.5) {$E_1$};
				
				\filldraw[point] (6,1.3) circle (0.8pt);
				\node[above, lbl] at (6,1.3) {$E_2$};
				\node[right, sublbl] at (6.08,0.9) {$\mathscr I$};
				
				\filldraw[point] (4,0.5) circle (0.8pt);
				\node[below, lbl] at (4,0.5) {$-\bar E_1$};
				
				\filldraw[point] (4,1.3) circle (0.8pt);
				\node[above, lbl] at (4,1.3) {$-\bar E_2$};
				\node[left, sublbl] at (3.92,0.9) {$-\mathscr I^*$};
				
				\filldraw[point] (4,-0.5) circle (0.8pt);
				\node[above, lbl] at (4,-0.5) {$-E_1$};
				
				\filldraw[point] (4,-1.3) circle (0.8pt);
				\node[below, lbl] at (4,-1.3) {$-E_2$};
				\node[left, sublbl] at (3.92,-0.9) {$-\mathscr I$};
				
				\filldraw[point] (6,-0.5) circle (0.8pt);
				\node[above, lbl] at (6,-0.5) {$\bar E_1$};
				
				\filldraw[point] (6,-1.3) circle (0.8pt);
				\node[below, lbl] at (6,-1.3) {$\bar E_2$};
				\node[right, sublbl] at (6.08,-0.9) {$\mathscr I^*$};
				
				\node[below, lbl] at (5,-2) {(b)};
			\end{scope}
			
		\end{tikzpicture}
		\caption{Mother body reduction for an elliptic condensation domain.  Panel (a) shows the four symmetry-related elliptic domains.  Panel (b) shows the reduced contour $\Sigma_{\rm ell}$ consisting of the mother body segments $\pm\mathscr I$ and $\pm\mathscr I^*$.}
		\label{fig:ellipse-collapse}
	\end{figure}
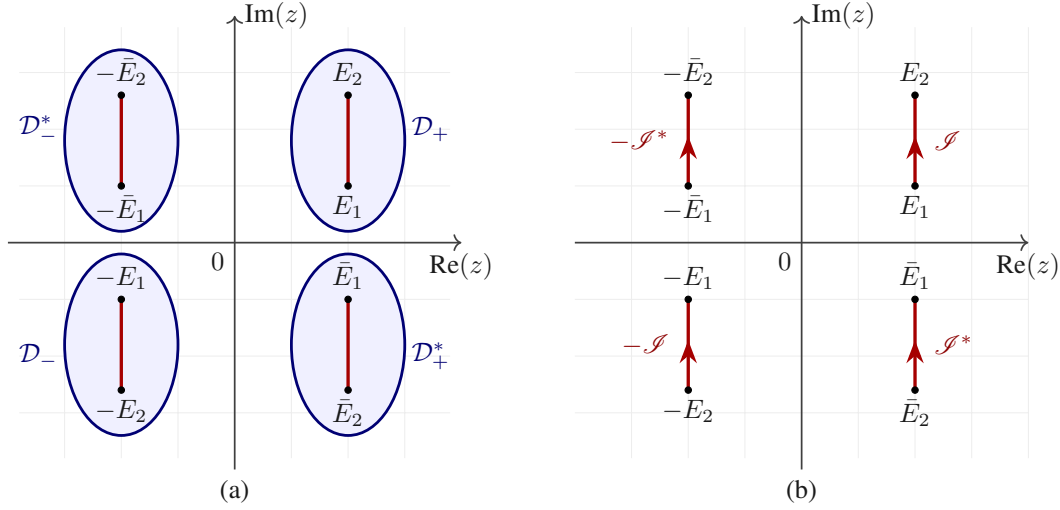
	
	For the elliptic DNLS soliton gas, 
	the exterior mother body representation takes the form of the
	following RHP:
	\begin{problem}
		\label{prob:elliptic-rhp}
		Find a $2\times2$ matrix-valued function $m(x,t,z)$ analytic in
		$
		\mathbb C\setminus\Sigma_{\rm ell}
		$
		such that:
		\begin{enumerate}
			\item \textbf{Jump condition.}
			For $z\in\Sigma_{\rm ell}$, the following jump condition holds
			\begin{equation}
				\label{eq:elliptic-jump}
				m_+(x,t,z)=m_-(x,t,z)J(x,t,z),
			\end{equation}
			where
			\begin{equation}
				\label{eq:elliptic-jump-matrix}
				J(x,t,z)=
				\begin{cases}
					\begin{pmatrix}
						1&0\\
						-\varrho(z)e^{2i\theta(x,t,z)}&1
					\end{pmatrix},
					& z\in\mathscr I\cup(-\mathscr I),\\[5mm]
					\begin{pmatrix}
						1&\varrho^*(z)e^{-2i\theta(x,t,z)}\\
						0&1
					\end{pmatrix},
					& z\in\mathscr I^*\cup(-\mathscr I^*).
				\end{cases}
			\end{equation}

			\item \textbf{Normalization.}
			\begin{equation}
				\label{eq:elliptic-normalization}
				m(x,t,z)=\mathbb I+\mathcal{O}(z^{-1}),
				\qquad z\to\infty .
			\end{equation}
			
			\item \textbf{Symmetries.}
			\begin{equation}
				\label{eq:elliptic-symmetry}
				m(x,t,z)=\sigma_2m^*(x,t,z)\sigma_2,
				\qquad
				m(x,t,z)=\sigma_3m(x,t,-z)\sigma_3 .
			\end{equation}
		\end{enumerate}

	\end{problem}
	
	By Proposition~\ref{prop:mother-body-reduction}, its reconstruction coefficient agrees with that of the exterior solution of the continuum $\bar{\partial}$-problem~\ref{prob:dbar}. The elliptic soliton gas potential is reconstructed by
	\begin{equation}
		\label{eq:elliptic-reconstruction}
		u(x,t)=2i\lim_{z\to\infty}z[m(x,t,z)]_{12}.
	\end{equation}
	RHP~\ref{prob:elliptic-rhp} is the
	RHP used in the rest of the paper. The large-$x$ asymptotics and the long-time asymptotics are therefore obtained by applying the Deift-Zhou nonlinear steepest descent analysis to RHP~\ref{prob:elliptic-rhp}.

	\section{Spectral symmetry and genus reduction of the Riemann Surface}
	\label{sec:genus-reduction}
	
	The RHP~\ref{prob:elliptic-rhp} is supported on branch cuts which are invariant under the spectral involution $z\mapsto -z$.  This symmetry is not a cosmetic feature of the formulation.  It is the mechanism which lowers the Abelian dimension of the model problem.  A model problem which appears, in the original $z$-plane, to be associated with a hyperelliptic curve of genus-$(2n+1)$ actually uses only the invariant part of the Abelian geometry.  Equivalently, it descends to a quotient curve (also called quotient algebraic curve) of genus-$n$ under the map
	\begin{equation}\label{map}
		\lambda=z^2.
	\end{equation}
	This section gives the intrinsic form of this reduction.  The point is that the reduction is a quotient reduction, not a degeneration of the spectral curve.  No branch points coalesce; rather, the symmetric RHP factors through the quotient by the involution $z\mapsto -z$.
	
	\subsection{The symmetric spectral curve and its quotient}
	\label{subsec:symmetric-curve-quotient}
	
	Let $n\geq1$ be the effective genus which will appear in the theta function solution of the model problem.  Let
	$
	E_1,E_2,\ldots,$
	$E_{n+1}\in\mathbb C_{\mathrm I}
	$
	be distinct points in the first quadrant, and assume that the $2n+2$ points
	\[
	E_1^2,\ldots,E_{n+1}^2,
	\qquad
	\bar E_1^{2},\ldots,\bar E_{n+1}^{2}
	\]
	are distinct and nonzero.  In the original $z$-plane, the symmetric branch point set is
	$
	\{\pm E_j,\pm\bar E_j\}_{j=1}^{n+1} .
	$
	The associated hyperelliptic curve \cite{Farkas} is
	\begin{equation}\label{hyperelliptic-curve-on-z}
		\mathcal X_n:
		\ 
		\mathcal{R}^2=\prod_{j=1}^{n+1}
		(z^2-E_j^2)(z^2-\bar E_j^{2}).
	\end{equation}
	It is a two-sheeted covering of the $z$-sphere with $4n+4$ finite simple branch points.  Therefore, the genus is
	\[
	g(\mathcal X_n)=\frac{4n+4-2}{2}=2n+1 .
	\]
	
	The curve $\mathcal X_n$ admits the holomorphic involution
	\[
	\rho:\mathcal X_n\longrightarrow\mathcal X_n,
	\qquad
	\rho(z,\mathcal{R})=(-z,\mathcal{R}).
	\]
	This involution is the geometric counterpart of the spectral symmetry \eqref{eq:elliptic-symmetry} of the RHP \ref{prob:elliptic-rhp}.
	Taking the quotient by $\rho$ gives the curve
	\[
	\widehat{\mathcal X}_n:=\mathcal X_n/\rho.
	\]
	It is described by the quotient coordinate
	\[
	\lambda=z^2,
	\qquad
	\widehat{\mathcal{R}}= \mathcal{R}.
	\]
	Consequently, we have
	\begin{equation}\label{hyperelliptic-curve-on-zeta}
		\widehat{\mathcal X}_n:
		\ 
		\widehat { \mathcal{R}}^{2}
		=
		\prod_{j=1}^{n+1}
		(\lambda-E_j^2)(\lambda-\bar E_j^{2}).
	\end{equation}
	This is a hyperelliptic curve with $2n+2$ finite simple branch points, and hence the genus is reduced to
	\[
	g(\widehat{\mathcal X}_n)=n .
	\]
	The quotient projection is
	\[
	\Pi:\mathcal X_n\longrightarrow\widehat{\mathcal X}_n,
	\qquad
	\Pi(z, \mathcal{R})=(z^2, \mathcal{R}).
	\]
	The four ramification points of $\Pi$ are the two points over $z=0$ and the two points at infinity.  The Riemann--Hurwitz formula \cite{Farkas} gives
	\[
	2g(\mathcal X_n)-2=
	2\bigl(2g(\widehat{\mathcal X}_n)-2\bigr)+4,
	\]
	which is consistent with
	\[
	g(\mathcal X_n)=2n+1,
	\qquad
	g(\widehat{\mathcal X}_n)=n .
	\]
	Thus the natural curve in the $z$-plane has genus $2n+1$, but the quotient curve selected by the symmetry of the DNLS equation \eqref{dnls3} has genus $n$.
	
	The genus reduction is most transparent at the level of holomorphic differentials.  On the full curve $\mathcal X_n$, a basis of holomorphic differentials is
	\begin{equation}\label{holomorphic-differentials}
		\eta_k=\frac{z^{k-1}}{\mathcal{R}(z)}dz,
		\qquad
		k=1,\ldots,2n+1.
	\end{equation}
	Since $\mathcal{R}(-z)=\mathcal{R}(z)$ and $d(-z)=-dz$, the pullback by $\rho$ satisfies
	\[
	\rho^*\eta_k=(-1)^k\eta_k .
	\]
	Therefore the $\rho$-invariant holomorphic differentials are precisely those with even index:
	\[
	\eta_2,\eta_4,\ldots,\eta_{2n} .
	\]
	Equivalently, the invariant subspace is spanned by
	\begin{equation}\label{invariant-subspace}
		\frac{z}{\mathcal{R}(z)}dz,
		\quad
		\frac{z^3}{\mathcal{R}(z)}dz,
		\quad
		\ldots,
		\quad
		\frac{z^{2n-1}}{\mathcal{R}(z)}dz.
	\end{equation}
	It has dimension $n$, exactly the genus of the quotient curve $\widehat{\mathcal X}_n$.
	
	On the quotient curve $\widehat{\mathcal X}_n$, a basis of holomorphic differentials is
	\[
	\widehat\eta_m=\frac{\lambda^{m-1}}{\widehat {\mathcal{R}}(\lambda)}d\lambda,
	\qquad
	m=1,\ldots,n .
	\]
	Pulling back by $\Pi$ gives
	\[
	\Pi^*\widehat\eta_m=\frac{(z^2)^{m-1}}{\mathcal{R}(z)}d(z^2)=
	2\frac{z^{2m-1}}{\mathcal{R}(z)}dz,
	\qquad
	m=1,\ldots,n .
	\]
	Hence the pullbacks of holomorphic differentials on $\widehat{\mathcal X}_n$ are precisely the invariant holomorphic differentials on $\mathcal X_n$.
	
	\begin{prop}
		\label{prop:invariant-differentials-revised}
		The pullback map
		$
		\Pi^*:H^{1,0}(\widehat{\mathcal X}_n)
		\longrightarrow
		H^{1,0}(\mathcal X_n)
		$
		is an isomorphism from $H^{1,0}(\widehat{\mathcal X}_n)$ onto the invariant subspace
		\[
		H^{1,0}(\mathcal X_n)^\rho
		:=
		\{\omega\in H^{1,0}(\mathcal X_n):\rho^*\omega=\omega\}.
		\]
		In particular,
		$
		\dim H^{1,0}(\mathcal X_n)^\rho=n .
		$
	\end{prop}
	
	\begin{proof}
		The calculation above shows that $\Pi^*H^{1,0}(\widehat{\mathcal X}_n)$ is contained in $H^{1,0}(\mathcal X_n)^\rho$ and is spanned by \eqref{invariant-subspace}. These $n$ differentials are linearly independent.  Since the invariant subspace of the basis displayed above is also $n$-dimensional, the pullback map is onto $H^{1,0}(\mathcal X_n)^\rho$.  It is injective because $\Pi$ is surjective.  Therefore $\Pi^*$ gives the stated isomorphism.
	\end{proof}
	The complementary anti-invariant subspace is
	\[
	H^{1,0}(\mathcal X_n)^{-\rho}
	:=
	\{\omega\in H^{1,0}(\mathcal X_n):\rho^*\omega=-\omega\}.
	\]
	It is spanned by
	$
	\eta_1,\eta_3,\ldots,\eta_{2n+1},
	$
	and has dimension $n+1$.
	
	\subsection{Descent of the Abelian differentials used in the model problem}
	\label{subsec:descent-abelian-differentials}
	
	The same quotient principle applies to the meromorphic Abelian differentials used to construct the $g$-function, the phase constants, and the normalized model parametrices.  The differentials appearing in the nonlinear steepest descent analysis of the DNLS equation have the symmetric form
	\[
	d\Omega(z)=
	\frac{zQ(z^2)}{\mathcal{R}(z)}dz,
	\]
	where $Q$ is a polynomial, possibly together with compatible principal parts at points belonging to the same $\rho$-orbit.  Such differentials satisfy
	\[
	\rho^*d\Omega=d\Omega .
	\]
	Since $d\lambda=2zdz$, every differential of the form displayed above can be written as the pullback of a differential on the quotient curve:
	\[
	d\Omega=\frac12\Pi^*\left(\frac{Q(\lambda)}{\widehat{\mathcal{R}}(\lambda)}d\lambda\right).
	\]
	Thus the $g$-function differentials used later are naturally quotient differentials.
	
	\begin{prop}
		\label{prop:descent-differentials-revised}
		Let $d\Omega$ be a meromorphic differential on $\mathcal X_n$ satisfying
		\[
		\rho^*d\Omega=d\Omega .
		\]
		Assume that the principal parts of $d\Omega$ at points in the same $\rho$-orbit are identified by the involution.  Then there exists a unique meromorphic differential $d\widehat\Omega$ on $\widehat{\mathcal X}_n$ such that
		\[
		d\Omega=\Pi^*d\widehat\Omega .
		\]
		Conversely, the pullback of any meromorphic differential on $\widehat{\mathcal X}_n$ is a $\rho$-invariant meromorphic differential on $\mathcal X_n$.
	\end{prop}
	
	\begin{proof}
		Away from the fixed points of $\rho$, the map $\Pi$ identifies each two-point orbit ${P,\rho(P)}$ with one point on $\widehat{\mathcal X}_n$.  Since $d\Omega$ is invariant, its local expressions at $P$ and $\rho(P)$ agree after projection to the quotient coordinate.  Therefore $d\Omega$ defines a meromorphic differential on the quotient away from the branch points of $\Pi$.

		At the fixed points, namely the two points over $z=0$ and the two points at infinity, one uses a local quotient coordinate.  The assumed compatibility of the principal parts ensures that the projected object has at worst meromorphic singularities.  Hence the projected differential is a meromorphic differential $d\widehat\Omega$ on $\widehat{\mathcal X}_n$ and satisfies $d\Omega=\Pi^*d\widehat\Omega$.  Uniqueness follows from the surjectivity of $\Pi$, and the converse is immediate from $\Pi\circ\rho=\Pi$.

	\end{proof}
	
	\begin{remark}
		\label{rem:g-differential-descends}
		In the later nonlinear steepest descent analysis, the normalized $g$-function differential has numerator depending on $z$ through $z$ times a polynomial in $z^2$, divided by the symmetric square root $\mathcal{R}(z)$.  Hence it is $\rho$-invariant and descends to $\widehat{\mathcal X}_n$.  The modulation constants and period conditions can therefore be imposed on the quotient curve from the beginning.
	\end{remark}
	
	We now formulate the Abelian data used in the model solution directly on the quotient curve.  Let
	\[
	\{\widehat a_j,\widehat b_j\}_{j=1}^n
	\]
	be a canonical homology basis on $\widehat{\mathcal X}_n$ as depicted in Figure~\ref{fig:genus}.  Let
	\begin{equation}\label{homology-basis-on-zeta}
		\widehat\nu_j=\frac{\sum_{m=0}^{n-1}d_{jm}\lambda^m}{\widehat {\mathcal{R}}(\lambda)}d\lambda,
		\qquad
		j=1,\ldots,n,
	\end{equation}
	be the normalized holomorphic differentials satisfying
	\[
	\oint_{\widehat a_i}\widehat\nu_j=\delta_{ij},
	\qquad
	i,j=1,\ldots,n .
	\]
	Their pullbacks
	$
	\nu_j:=\Pi^*\widehat\nu_j
	$
	are the normalized invariant holomorphic differentials on the full curve, with respect to cycles which project one-to-one onto $\widehat a_j$ and $\widehat b_j$.  Explicitly,
	\[
	\nu_j=
	2\frac{z\sum_{m=0}^{n-1}d_{jm}z^{2m}}{\mathcal{R}(z)}dz,
	\qquad
	j=1,\ldots,n.
	\]
	Thus the normalized differentials used in the symmetric model problem are precisely the invariant differentials.  No anti-invariant differential contributes.
	
	The reduced period matrix is
	\begin{equation}\label{reduced-period-matrix}
		\widehat{\mathcal P}_{ij}
		=
		\oint_{\widehat b_i}\widehat\nu_j,
		\qquad
		i,j=1,\ldots,n.
	\end{equation}
	By the Riemann bilinear relations on $\widehat{\mathcal X}_n$,
	\begin{equation}\label{period-matrix-conditions}
		\widehat{\mathcal P}^{T}=\widehat{\mathcal P},
		\qquad
		\operatorname{Im}\widehat{\mathcal P}>0.
	\end{equation}
	The function entering the model problem is therefore the $n$-dimensional Riemann theta function
	\[
	\Theta(\mathbf z;\widehat{\mathcal P})
	=
	\sum_{\mathbf n\in\mathbb Z^n}
	\exp\left\{
	2\pi i\mathbf n^T\mathbf z
	+
	\pi i\mathbf n^T\widehat{\mathcal P}\mathbf n
	\right\}.
	\]
	
	Fix a base point $\widehat P_0\in\widehat{\mathcal X}_n$.  The Abel map used in the model parametrix is
	\begin{equation}\label{Abel-map}
		\widehat{\mathcal A}(\widehat P)
		=
		\int_{\widehat P_0}^{\widehat P}
		(\widehat\nu_1,\ldots,\widehat\nu_n)^T,
		\qquad
		\widehat P\in\widehat{\mathcal X}_n.
	\end{equation}
	If $P\in\mathcal X_n$ and $\widehat P=\Pi(P)$, then, for any lift $P_0$ of $\widehat P_0$, we have
	\[
	\widehat{\mathcal A}(\Pi(P))=
	\int_{P_0}^{P}(\nu_1,\ldots,\nu_n)^T
	\quad
	\operatorname{mod}\ \mathbb Z^n+
	\widehat{\mathcal P}\mathbb Z^n .
	\]
	In the original spectral variable, the quotient point is represented by
	\eqref{map}. Thus the two points with spectral coordinates $z$ and $-z$ project to the same point on $\widehat{\mathcal X}_n$.  This is exactly the geometric content of the symmetry $m(x,t,z)=\sigma_3m(x,t,-z)\sigma_3$: the solution, expressed in terms of Riemann theta function, depends on $z$ only through the quotient point $\lambda=z^2$, together with the elementary matrix factors which encode the sheet exchange and the $\sigma_3$ symmetry.
	\begin{figure}[htbp]
		\centering
		\begin{tikzpicture}[scale=1.2,
			transform shape,
			x=1.45cm,
			y=1.45cm,
			every node/.style={font=\small},
			point/.style={circle, fill=black, inner sep=1.35pt},
			cut/.style={black, line width=2.6pt, line cap=round},
			acycle/.style={red!75!black, line width=0.95pt, line cap=round},
			bcycle/.style={blue!70!black, line width=0.95pt, line cap=round},
			breturn/.style={blue!55!black, line width=0.8pt, line cap=round,
				dash pattern=on 2.6pt off 2.6pt},
			cyclearrow/.style={
				postaction={decorate},
				decoration={
					markings,
					mark=at position #1 with {
						\arrow{Stealth[length=2.1mm,width=1.45mm]}
					}
				}
			}
			]
			\coordinate (Enp1) at (-4.25,0); 
			\coordinate (En)   at (-3.25,0); 
			\coordinate (Enm1) at (-2.05,0); 
			\coordinate (Enm2) at (-1.05,0); 
			\coordinate (E4)   at ( 1.05,0); 
			\coordinate (E3)   at ( 2.05,0); 
			\coordinate (E2)   at ( 3.25,0); 
			\coordinate (E1)   at ( 4.25,0); 
			
			\draw[cut] (Enp1) -- (En);
			\draw[cut] (Enm1) -- (Enm2);
			\draw[cut] (E4)   -- (E3);
			\draw[cut] (E2)   -- (E1);
			
			\foreach \P in {Enp1,En,Enm1,Enm2,E4,E3,E2,E1}{
				\node[point] at (\P) {};
			}
			
			\node[anchor=north east, font=\scriptsize] at ($(Enp1)+(-0.03,0.08)$) {$E_{n+1}^2$};
			\node[anchor=north west, font=\scriptsize] at ($(En)+( 0.03,0.08)$) {$E_n^2$};
			
			\node[anchor=north east, font=\scriptsize] at ($(Enm1)+(-0.03,0.08)$) {$E_{n-1}^2$};
			\node[anchor=north west, font=\scriptsize] at ($(Enm2)+( 0.03,0.08)$) {$E_{n-2}^2$};
			
			\node[anchor=north east, font=\scriptsize] at ($(E4)+(-0.03,0.08)$) {$E_4^2$};
			\node[anchor=north west, font=\scriptsize] at ($(E3)+( 0.03,0.08)$) {$E_3^2$};
			
			\node[anchor=north, font=\scriptsize] at ($(E2)+(-0.20,0.08)$) {$E_2^2$};
			\node[anchor=west, font=\scriptsize]  at ($(E1)+(0.06,-0.02)$) {$E_1^2$};
			
			\node at (0,0.04) {$\cdots$};
			
			\draw[acycle, cyclearrow=.23]
			(-3.75,0) ellipse [x radius=0.62, y radius=0.30];
			
			\draw[acycle, cyclearrow=.23]
			(-1.55,0) ellipse [x radius=0.62, y radius=0.30];
			
			\draw[acycle, cyclearrow=.23]
			( 1.55,0) ellipse [x radius=0.62, y radius=0.30];
			
			\node[red!75!black, anchor=south, font=\small] at (-3.75,0.33) {$\widehat a_n$};
			\node[red!75!black, anchor=south, font=\small] at (-1.55,0.33) {$\widehat a_{n-1}$};
			\node[red!75!black, anchor=south, font=\small] at ( 1.55,0.33) {$\widehat a_1$};
			
			\draw[bcycle, cyclearrow=.55]
			(E1) .. controls (3.75,0.60) and (2.20,0.60) .. (1.55,0.03);
			
			\draw[bcycle, cyclearrow=.55]
			(E1) .. controls (3.25,1.10) and (-0.45,1.10) .. (-1.55,0.03);
			
			\draw[bcycle, cyclearrow=.57]
			(E1) .. controls (2.45,1.50) and (-2.55,1.50) .. (-3.75,0.03);
			
			\draw[breturn, cyclearrow=.55]
			(1.55,-0.03) .. controls (2.15,-0.55) and (3.65,-0.50) .. (E1);
			
			\draw[breturn, cyclearrow=.45]
			(-1.55,-0.03) .. controls (-0.25,-0.88) and (3.00,-0.80) .. (E1);
			
			\draw[breturn, cyclearrow=.45]
			(-3.75,-0.03) .. controls (-2.25,-1.18) and (2.85,-1.12) .. (E1);
			
			\node[blue!70!black, font=\small] at (2.85,0.25) {$\widehat b_1$};
			\node[blue!70!black, font=\small] at (0.65,0.55) {$\widehat b_{n-1}$};
			\node[blue!70!black, font=\small] at (-0.65,0.90) {$\widehat b_n$};
			
		\end{tikzpicture}
		
		\caption{
			The homology basis \(\{\widehat{a}_i,\widehat{b}_i\}_{i=1}^{n}\) is drawn on the
			quotient curve \(\widehat{\mathcal X}_n\) associated with the symmetry
			\(z\mapsto -z\) of the DNLS equation. The red cycles are the \(\widehat{a}\)-cycles
			surrounding the branch cuts, and the blue cycles are the corresponding
			\(\widehat{b}\)-cycles. The branch points \(E_j^2\) have been placed on the same line
			for clarity.}
		\label{fig:genus}
	\end{figure}
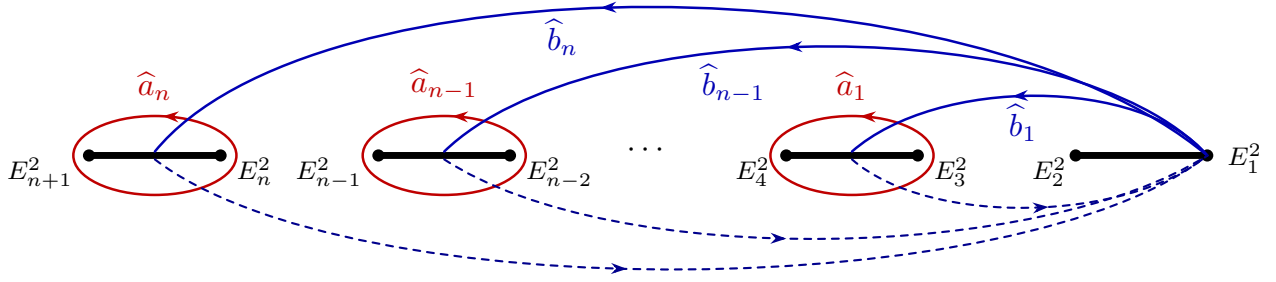

	\subsection{Consequences for the model RHP of the DNLS equation}
	\label{consequences-for-DNLS}
	
	We now summarize how the quotient construction enters the RHP of the DNLS equation used in the sequel.
	The reduced elliptic RHP in Section~\ref{subsec:elliptic-model}, and the higher genus model problems arising after the 
	\(g\)-function deformation, have branch cuts which are invariant under the involution
	\[
	z\mapsto -z .
	\]
	Although the corresponding symmetric curve in the original \(z\)-plane has genus-$(2n+1)$, 
	the jump matrices, normalization, and Abelian differentials entering the model problem are invariant under this involution. 
	Therefore the anti-invariant Abelian directions do not enter the solution.
	
	Consequently, the model problem is constructed on the quotient curve
	\[
	\widehat{\mathcal X}_n=\mathcal X_n/\rho,
	\qquad
	\lambda=z^2,
	\]
	whose genus is \(n\). The normalized holomorphic differentials, period matrix, Abel map, 
	and Riemann theta function used in the solution of model problem are those of \(\widehat{\mathcal X}_n\). 
	After pullback by
	\[
	\Pi:\mathcal X_n\to\widehat{\mathcal X}_n,
	\qquad
	\Pi(z,\mathcal R)=(z^2,\mathcal R),
	\]
	they give the symmetric solution of the original \(z\)-plane model problem.
	
	In particular, a direct construction on the full hyperellipti curve would formally involve a Riemann theta function of dimension \(2n+1\), 
	whereas the symmetry of the DNLS equation reduces to an effective dimension \(n\). 
	Thus, it is concluded that
	\[
	\text{effective genus}=n,
	\qquad
	\text{unreduced symmetric genus}=2n+1.
	\]
	For the elliptic soliton gas given by RHP~\ref{prob:elliptic-rhp} in Section~\ref{subsec:elliptic-model}, consider \(n=1\), then the full curve in the \(z\)-plane has genus \(3\), 
	but the quotient curve is elliptic, and the large-\(x\) background is described by a one-dimensional theta function. 
	Similarly, in the long-time analysis, the two-phase and three-phase regions correspond to quotient genera 
	\(n=2\) and \(n=3\), respectively, although the unreduced symmetric curves have genera \(5\) and \(7\), respectively.
	\par
	This reduction is not caused by a collision of branch points. The full curve remains smooth; the reduction occurs 
	because the RHP excites only the \(\rho\)-invariant part of the Abelian geometry, which is precisely the pullback of 
	the Abelian geometry of the quotient curve.

	\section{Large-$x$ asymptotics for the elliptic soliton gas}
	\label{sec:large-x}
	
	This section proves Theorem \ref{spatial} by studying the large-space behavior of the elliptic soliton gas at $t=0$, i.e., $u(x,0)$, based on the RHP~\ref{prob:elliptic-rhp}.  Thus the phase in the jump matrix is
	$
	\theta(x,0,z)=xz^2,
	$
	and the jump contour $\Sigma_{\rm ell}$ is given by \eqref{jump-contours-ell}.
	The purpose of this section is twofold.  First, we show that the soliton gas is exponentially small as $x\to+\infty$.  Second, we show that as $x\to-\infty$ the solution approaches a one-phase finite-gap background.  The second statement is the first place in this paper where the genus reduction mechanism in Section~\ref{sec:genus-reduction} enters the asymptotic analysis in an essential way.
	\par
	Throughout this section, we assume that the intrinsic data associated with the elliptic condensation domain are admissible in the following sense. The density function $\varrho$ in \eqref{eq:elliptic-density} is nonzero in the interior of the mother body segments, and near every endpoint $E$ of $\Sigma_{\rm ell}$ it has the form
	\begin{equation}\label{density-rho}
		\varrho(z)=(z-E)^{1/2}\varrho_E(z),
		\qquad
		\varrho_E(E)\neq0,
	\end{equation}
	with the analogous Schwarz conjugate behavior on $\mathscr I^*\cup(-\mathscr I^*)$.
	
	\subsection{Exponential decay as \texorpdfstring{$x\to+\infty$}{x to +infinity}}
	\label{subsec:large-x-right}
	
	The right-hand asymptotics follows from a direct small norm argument. On 
	\(\mathscr I\cup(-\mathscr I)\), one has \(\operatorname{Im} z^2>0\), while on 
	\(\mathscr I^*\cup(-\mathscr I^*)\), one has \(\operatorname{Im} z^2<0\). Hence, for \(x>0\), we have
	\[
	\left|e^{2ixz^2}\right|
	=
	e^{-2x\operatorname{Im}z^2},
	\qquad z\in \mathscr I\cup(-\mathscr I),
	\]
	whereas
	\[
	\left|e^{-2ixz^2}\right|
	=
	e^{2x\operatorname{Im}z^2},
	\qquad z\in \mathscr I^*\cup(-\mathscr I^*).
	\]
	Since \(\operatorname{Im}z^2\) is bounded away from zero on these contour segments, the jump matrix of RHP~\ref{prob:elliptic-rhp} satisfies
	\[
	\|J(x,0,\cdot)-\mathbb I\|_{L^2(\Sigma_{\rm ell})\cap L^\infty(\Sigma_{\rm ell})}
	=
	\mathcal{O}(e^{-c_rx}),
	\qquad x\to+\infty,
	\]
	for some \(c_r>0\).
	
	The Beals--Coifman small norm theory then gives
	\[
	m(x,0,z)=\mathbb I+\mathcal{O}(e^{-c_rx}),
	\qquad x\to+\infty,
	\]
	uniformly for \(z\) on compact subsets of \(\mathbb C\setminus\Sigma_{\rm ell}\). Moreover, expanding the Cauchy representation at infinity yields
	\[
	m(x,0,z)
	=
	\mathbb I+\frac{m_1(x)}{z}+\mathcal{O}(z^{-2}),
	\qquad
	m_1(x)=\mathcal{O}(e^{-c_rx}).
	\]
	Therefore, by the reconstruction formula, it follows that
	\[
	u(x,0)=2i[m_1(x)]_{12}
	=
	\mathcal{O}(e^{-c_rx}),
	\qquad x\to+\infty .
	\]
	
	\subsection{Asymptotic behavior as $x \rightarrow -\infty$}
	\label{subsec:large-x-g}
	
	The asymptotics as $x\to-\infty$ is qualitatively different. The
	exponentials in the jump matrix \eqref{eq:elliptic-jump-matrix} no longer
	decay on the original contour. Thus introduce a $g$-function which converts
	the jumps on the main bands into constant jumps. The jumps on the
	lens boundaries are exponentially small away from the band endpoints; the
	endpoint neighborhoods are treated by local Bessel parametrices.
	
	Introduce the symmetric curve
	\[
	\mathcal X_1:\ 
	\mathcal R_1(z)^2
	=
	(z^2-E_1^2)(z^2-\bar E_1^{\,2})
	(z^2-E_2^2)(z^2-\bar E_2^{\,2}),
	\]
	with branch cuts on $\Sigma_{\rm ell}$, and choose the branch such that
	\[
	\mathcal R_1(z)=z^4+\mathcal O(z^2),
	\qquad z\to\infty.
	\]
	Set $S_E:=\operatorname{Re}(E_1^2+E_2^2)$ and define
	\begin{equation}\label{d-Omega-0}
		d\Omega_0(z)
		=
		\frac{z^5-S_Ez^3+\tau z}{\mathcal R_1(z)}\,dz,
	\end{equation}
	where $\tau$ is determined by
	\[
	\int_{\bar E_1}^{E_1}d\Omega_0=0,
	\qquad
	\tau
	=-
	\frac{\displaystyle\int_{\bar E_1}^{E_1}
		\frac{z^5-S_Ez^3}{\mathcal R_1(z)}\,dz}
	{\displaystyle\int_{\bar E_1}^{E_1}
		\frac{z}{\mathcal R_1(z)}\,dz}.
	\]
	We assume that the denominator above and the numerator of
	$d\Omega_0$ at every branch point are nonzero. Define
	\begin{equation}\label{g-fun-x}
		g(z)=-z^2+2\int_{E_2}^{z}d\Omega_0,
		\qquad
		\Phi(z):=g(z)+z^2.
	\end{equation}
	The additive constant is chosen so that
	\begin{equation}\label{large x of g}
		\begin{cases}
			g_+(z)+g_-(z)=-2z^2,
			&z\in\Sigma_{\rm ell},\\
			g_+(z)-g_-(z)=\Omega,
			&z\in\Sigma_{\rm gap},
		\end{cases}
		\qquad
		\Omega:=2\int_{E_2}^{E_1}d\Omega_0.
	\end{equation}
	Moreover, it is seen that $g=g^*$ and
	\begin{equation}\label{g-infty-x}
		g(z)=g_\infty+\mathcal O(z^{-2}),
		\qquad z\to\infty.
	\end{equation}
	At a branch point $E\in\{\pm E_j,\pm\bar E_j\}_{j=1}^2$, the function
	$g$ is bounded and
	\begin{equation}\label{eq:fixed-endpoint-phase-x}
		\Phi(z)-\Phi(E)
		=\beta_E(z-E)^{1/2}
		\bigl(1+\mathcal O(z-E)\bigr),
		\qquad \beta_E\neq0.
	\end{equation}
	
	Take the first transformation
	\begin{equation}\label{transformation-1}
		m^{(1)}(x,z)
		=e^{-ig_\infty x\sigma_3}
		m(x,0,z)e^{ig(z)x\sigma_3}.
	\end{equation}
	On $\Sigma_{\rm ell}$, the condition \eqref{large x of g} gives $m^{(1)}_{+}(x,z)=m^{(1)}_{-}(x,z)J^{(1)}(x,z)$ with
	\[
	J^{(1)}(x,z)
	=
	\begin{cases}
		\begin{pmatrix}
			e^{ix(g_+-g_-)}&0\\
			-\varrho(z)&e^{-ix(g_+-g_-)}
		\end{pmatrix},
		&z\in\mathscr I\cup(-\mathscr I),\\[4mm]
		\begin{pmatrix}
			e^{ix(g_+-g_-)}&\varrho^*(z)\\
			0&e^{-ix(g_+-g_-)}
		\end{pmatrix},
		&z\in\mathscr I^*\cup(-\mathscr I^*).
	\end{cases}
	\]
	
	We next introduce a scalar function $F$, analytic and nonzero away from
	the bands and gaps, satisfying
	\begin{equation}\label{eq:F-large-x-jumps}
		\begin{cases}
			F_+F_-=\varrho^{-1},
			&z\in\mathscr I\cup(-\mathscr I),\\
			F_+F_-=\varrho^*,
			&z\in\mathscr I^*\cup(-\mathscr I^*),\\
			F_+/F_-=e^\kappa,
			&z\in\Sigma_{\rm gap}.
		\end{cases}
	\end{equation}
	The constant $\kappa$ is fixed by the full period condition
	\[
	\oint_{\widehat a}d\log F=0,
	\]
	which makes $F$ single-valued. At infinity, it follows that
	\[
	F(z)=F_\infty+\mathcal O(z^{-1}),
	\qquad F_\infty\neq0,
	\]
	and, since the function $\varrho$ has a square-root zero at every fixed endpoint,
	$F$ has the fourth-root behavior: it is of order
	$(z-E)^{-1/4}$ on the bands carrying $\varrho^{-1}$ and of order
	$(z-E)^{1/4}$ on their Schwarz conjugate bands.
	
	Define
	\begin{equation}\label{transformation-2}
		m^{(2)}(x,z)
		=F_\infty^{-\sigma_3}m^{(1)}(x,z)F(z)^{\sigma_3},
	\end{equation}
	then it is derived that $m^{(2)}_{+}(x,z)=m^{(2)}_{-}(x,z)J^{(2)}(x,z)$ with
	\[
	J^{(2)}(x,z)
	=
	\begin{cases}
		\begin{pmatrix}
			e^{ix(g_+-g_-)}F_+F_-^{-1}&0\\
			-1&e^{-ix(g_+-g_-)}F_-F_+^{-1}
		\end{pmatrix},
		&z\in\mathscr I\cup(-\mathscr I),\\[5mm]
		\begin{pmatrix}
			e^{ix(g_+-g_-)}F_+F_-^{-1}&1\\
			0&e^{-ix(g_+-g_-)}F_-F_+^{-1}
		\end{pmatrix},
		&z\in\mathscr I^*\cup(-\mathscr I^*),\\[3mm]
		e^{(ix\Omega+\kappa)\sigma_3},
		&z\in\Sigma_{\rm gap}.
	\end{cases}
	\]
	For $z\in\mathscr I\cup(-\mathscr I)$ this jump factors as
	\[
	\begin{pmatrix}
		1&-\dfrac{e^{-2ix\Phi_-(z)}}{\varrho(z)F_-(z)^2}\\
		0&1
	\end{pmatrix}
	\begin{pmatrix}0&1\\-1&0\end{pmatrix}
	\begin{pmatrix}
		1&-\dfrac{e^{-2ix\Phi_+(z)}}{\varrho(z)F_+(z)^2}\\
		0&1
	\end{pmatrix},
	\]
	and the conjugate bands have the corresponding lower-triangular
	factorization. Define
	\begin{equation}\label{transformation-3}
		m^{(3)}(x,z)=m^{(2)}(x,z)G(x,z),
	\end{equation}
	where
	\begin{equation}
		G(x,z) =
		\begin{cases}
			\begin{pmatrix}
				1 & -\displaystyle \frac{e^{-2 i x (g(z) + z^2)}}{\varrho(z) \, F(z)^2} \\[6pt]
				0 & 1
			\end{pmatrix}
			& z \in \pm \mathscr{U}_{+}, \quad \begin{pmatrix}
				1 & 0 \\[4pt]
				-\displaystyle \frac{e^{2 i x (g(z) + z^2)} \, F(z)^2}{\varrho^*(z)} & 1
			\end{pmatrix}
			\ \ \ \ \ z \in \pm \mathscr{U}_{+}^*, \\[6pt]
			\begin{pmatrix}
				1 & \displaystyle \frac{e^{-2 i x (g(z) + z^2)}}{\varrho(z) \, F(z)^2} \\[6pt]
				0 & 1
			\end{pmatrix}
			& z \in \pm \mathscr{U}_{-}, \quad
			\begin{pmatrix}
				1 & 0 \\[4pt]
				\displaystyle \frac{e^{2 i x (g(z) + z^2)} \, F(z)^2}{\varrho^*(z)} & 1
			\end{pmatrix}
			\ \ \ \ \ \ \ \ z \in \pm \mathscr{U}_{-}^*, \\[6pt]
			\mathbb{I},
			& \text{otherwise}.
		\end{cases}
	\end{equation}
	The lenses are chosen according to the sign of $\operatorname{Im}\Phi$.
	Away from fixed endpoint disks, the resulting lens jumps are
	$\mathcal O(e^{-c|x|})$ as $x\to-\infty$. The lens configuration is the
	one depicted in Figure~\ref{figure:lenses}.
	
	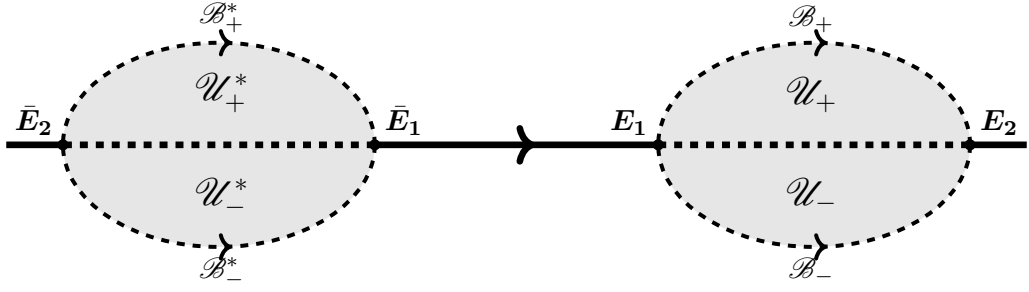
\begin{figure}
		\centering
		\begin{tikzpicture}[scale=1.5]
			\draw[-, thick, line width=2.5pt, black] (-4.5,0) -- (-4,0);
			\draw[-, postaction={decorate}, decoration={markings, mark=at position 0.5 with {\arrow{<}}}, thick,line width=2.5pt, black] (1.25,0) -- (-1.25,0);
			\draw[-, thick,line width=2.5pt, black] (4.5,0) -- (4,0);

			\fill[gray!20] (4,0) arc (0:180: 1.375 and 0.9) -- cycle;
			\draw[-,dashed, postaction={decorate}, decoration={markings, mark=at position 0.5 with {\arrow{<}}},thick, line width=1.5pt, black] (4,0) arc (0:180: 1.375 and 0.9);
			
			\fill[gray!20] (4,0) arc (360:180: 1.375 and 0.9) -- cycle;
			\draw[-,dashed, postaction={decorate}, decoration={markings, mark=at position 0.5 with {\arrow{<}}},thick, line width=1.5pt, black] (4,0) arc (360:180: 1.375 and 0.9);
			
			\fill[gray!20] (-1.25,0) arc (0:180: 1.375 and 0.9) -- cycle;
			\draw[-,dashed, postaction={decorate}, decoration={markings, mark=at position 0.5 with {\arrow{<}}},thick, line width=1.5pt, black] (-1.25,0) arc (0:180: 1.375 and 0.9);
			
			\fill[gray!20] (-1.25,0) arc (360:180: 1.375 and 0.9) -- cycle;
			\draw[-,dashed, postaction={decorate}, decoration={markings, mark=at position 0.5 with {\arrow{<}}},thick, line width=1.5pt, black] (-1.25,0) arc (360:180: 1.375 and 0.9);
			
			\draw[-, dashed,line width=2.5pt, black] (1.25,0) -- (4,0);
			\draw[-, dashed,line width=2.5pt, black] (-1.25,0) -- (-4,0);

			\filldraw[black](1.25,0) circle (1.25pt);
			\node at (1.25,0) [above left, black] {$\bm{E_1}$};
			\filldraw[black](4,0) circle (1.25pt);
			\node at (4,0) [above right, black] {$\bm{E_2}$};
			\filldraw[black](-1.25,0) circle (1.25pt);
			\node at (-1.25,0) [above right, black] {$\bm{\bar{E}_1}$};
			\filldraw[black](-4,0) circle (1.25pt);
			\node at (-4,0) [above left, black] {$\bm{\bar{E}_2}$};
			
			\node at (2.6,0.2) [above, black,font=\Large] {$\mathscr{U}_+$};
			\node at (2.6,-0.2) [below, black,font=\Large] {$\mathscr{U}_-$};
			\node at (-2.6,0.2) [above, black,font=\Large] {$\mathscr{U}_+^*$};
			\node at (-2.6,-0.2) [below, black,font=\Large] {$\mathscr{U}_-^*$};
			
			\node at (2.6,0.9) [above, black] {$\mathscr{B}_+$};
			\node at (2.6,-0.9) [below, black] {$\mathscr{B}_-$};
			\node at (-2.6,0.9) [above, black] {$\mathscr{B}_+^*$};
			\node at (-2.6,-0.9) [below, black] {$\mathscr{B}_-^*$};
		\end{tikzpicture}
		\caption{The lenses of segments $\mathscr{I}$ and $\mathscr{I}^*$ (the original image was rotated 90 degrees clockwise for visual clarity).}
		\label{figure:lenses}
	\end{figure}
	
	The remaining outer model has constant jumps
	\begin{equation}\label{model-prob-x}
		J^{\rm mod}(x,z)
		=
		\begin{cases}
			\begin{pmatrix}0&1\\-1&0\end{pmatrix},
			&z\in\Sigma_{\rm ell},\\[3mm]
			e^{(ix\Omega+\kappa)\sigma_3},
			&z\in\Sigma_{\rm gap}.
		\end{cases}
	\end{equation}
	By the genus-reduction mechanism given in Section~\ref{sec:genus-reduction},
	this model descends to the quotient elliptic curve
	\[
	\widehat{\mathcal X}_1:\ 
	\widehat{\mathcal R}_1(\lambda)^2
	=
	(\lambda-E_1^2)(\lambda-\bar E_1^{\,2})
	(\lambda-E_2^2)(\lambda-\bar E_2^{\,2}),
	\qquad \lambda=z^2.
	\]
	Let $\widehat\nu$ be normalized by
	$\oint_{\widehat a}\widehat\nu=1$ and set
	\begin{equation}\label{period-matrix-x}
		\widehat{\mathcal P}
		=\oint_{\widehat b}\widehat\nu.
	\end{equation}
	The Riemann theta function is
	\[
	\Theta(w;\widehat{\mathcal P})
	=\sum_{k\in\mathbb Z}
	\exp\left(2\pi ikw+\pi ik^2\widehat{\mathcal P}\right).
	\]
	Let
	$\widehat{\mathcal A}(P)
	=\int_{\infty^+}^{P}\widehat\nu.$
	We set
	\begin{equation}\label{limit-of-abel-map-x}
		\widehat{\mathcal A}_\infty
		:=\lim_{z\to\infty}\widehat{\mathcal A}(z^-),
	\end{equation}
	and define
	\begin{equation}\label{delta-x}
		\Delta(x)=\frac{x\Omega-i\kappa}{2\pi}.
	\end{equation}
	
	Introduce
	\[
	\alpha_1(z)
	=
	\left[
	\frac{(z-E_1)(z-\bar E_2)(z+\bar E_1)(z+E_2)}
	{(z+E_1)(z+\bar E_2)(z-\bar E_1)(z-E_2)}
	\right]^{1/4},
	\qquad \alpha_1(\infty)=1,
	\]
	and
	\[
	Q^\pm(z)
	:=
	\frac{\Theta(\widehat{\mathcal A}(z^\pm)-\Delta(x);
		\widehat{\mathcal P})}
	{\Theta(\widehat{\mathcal A}(z^\pm);\widehat{\mathcal P})}.
	\]
	The unnormalized matrix expressed by Riemann theta function is
	\[
	\mathcal N(x,z)
	=
	\frac{\Theta(0;\widehat{\mathcal P})}
	{2\Theta(\Delta(x);\widehat{\mathcal P})}
	\begin{pmatrix}
		Q^+(z)(\alpha_1+\alpha_1^{-1})
		&-iQ^-(z)(\alpha_1-\alpha_1^{-1})\\[1mm]
		iQ^+(z)(\alpha_1-\alpha_1^{-1})
		&Q^-(z)(\alpha_1+\alpha_1^{-1})
	\end{pmatrix}.
	\]
	It has the jumps in \eqref{model-prob-x}, but
	\[
	\mathcal N(x,\infty)
	=\operatorname{diag}(1,T(x)),
	\qquad
	T(x)
	=
	\frac{\Theta(0;\widehat{\mathcal P})
		\Theta(\widehat{\mathcal A}_\infty-\Delta(x);
		\widehat{\mathcal P})}
	{\Theta(\Delta(x);\widehat{\mathcal P})
		\Theta(\widehat{\mathcal A}_\infty;
		\widehat{\mathcal P})}.
	\]
	Hence the normalized outer model problem is
	\begin{equation}\label{eq:large-x-normalized-outer}
		m^{\rm out}(x,z)
		=(\mathcal N(x,\infty))^{-1}\mathcal N(x,z),
		\qquad
		m^{\rm out}(x,z)=\mathbb I+\mathcal O(z^{-1}).
	\end{equation}
	
	It remains to construct the local parametrices. Put $X:=-x>0$ and let
	\[
	\mathcal P_{\rm fix}
	=\{\pm E_1,\pm E_2,
	\pm\bar E_1,\pm\bar E_2\}.
	\]
	For each $E\in\mathcal P_{\rm fix}$, choose a small fixed disk $D_E$,
	with all disks mutually disjoint. By
	\eqref{eq:fixed-endpoint-phase-x}, one can choose a local branch
	\[
	\phi_E(z)=-i\varepsilon_E
	\bigl(\Phi(z)-\Phi(E)\bigr),
	\qquad \varepsilon_E\in\{\pm1\},
	\]
	so that the relevant lens exponential is decaying, and
	\begin{equation}\label{eq:large-x-bessel-coordinate}
		\zeta_E(z)=X^2\phi_E(z)^2
	\end{equation}
	is conformal in $D_E$. Since $\varrho(z)\sim(z-E)^{1/2}$, the local
	problem is the standard Bessel model of order $1/2$. Denote its solution
	by $\Psi_{\rm Bes}(\zeta)$. We use the normalization
	\[
	\Psi_{\rm Bes}(\zeta)
	=
	\zeta^{-\sigma_3/4}M_{\rm B}
	\left(\mathbb I+\mathcal O(\zeta^{-1/2})\right)
	e^{-2\zeta^{1/2}\sigma_3},
	\qquad \zeta\to\infty,
	\]
	where $M_{\rm B}$ is a constant unimodular matrix. The local parametrix has the form
	\begin{equation}\label{eq:large-x-bessel-parametrix}
		P_E(z)
		=E_E(z)
		\Psi_{\rm Bes}(\zeta_E(z))
		e^{2X\phi_E(z)\sigma_3}C_E(z),
		\qquad z\in D_E,
	\end{equation}
	where $E_E$ and $C_E$ are analytic and invertible in $D_E$. They are
	chosen so that $P_E$ has the exact local jumps and the same endpoint
	behavior as the transformed solution. Since
	$|\zeta_E|\sim X^2$ on $\partial D_E$, we have
	\begin{equation}\label{eq:large-x-bessel-matching}
		P_E(z)
		=m^{\rm out}(x,z)
		\left(\mathbb I+\mathcal O(X^{-1})\right),
		\qquad z\in\partial D_E.
	\end{equation}
	
	Define the global approximation
	\[
	m^{\rm app}(x,z)
	=
	\begin{cases}
		P_E(z),&z\in D_E,
		\quad E\in\mathcal P_{\rm fix},\\
		m^{\rm out}(x,z),&z\notin\displaystyle\bigcup_E D_E,
	\end{cases}
	\]
	and the error matrix
	\[
	\mathcal E(x,z)
	=m^{(3)}(x,z)\bigl(m^{\rm app}(x,z)\bigr)^{-1}.
	\]
	On the lens contours outside the disks,
	$J_{\mathcal E}=\mathbb I+\mathcal O(e^{-cX})$, whereas on the Bessel
	circles \eqref{eq:large-x-bessel-matching} gives
	$J_{\mathcal E}=\mathbb I+\mathcal O(X^{-1})$. Consequently, one has
	\[
	\|J_{\mathcal E}-\mathbb I\|_{L^2\cap L^\infty}
	=\mathcal O(X^{-1}),
	\]
	and the small-norm theorem yields
	\begin{equation}\label{small-norm-x}
		\mathcal E(x,z)
		=\mathbb I+
		\mathcal O\left(\frac{X^{-1}}{1+|z|}\right),
		\qquad x\to-\infty.
	\end{equation}
	In particular, the coefficient of $z^{-1}$ in the expansion of
	$\mathcal E$ is $\mathcal O(X^{-1})$.
	
	Tracing back \eqref{transformation-1}--\eqref{transformation-3} and using
	the reconstruction formula give
	\begin{equation}\label{asympotics-x-l}
		\mathcal U_-(x)
		=
		2i\bigl(\operatorname{Im}E_2-
		\operatorname{Im}E_1\bigr)
		T(x)F_\infty^2e^{2ig_\infty x},
	\end{equation}
	where
	\begin{equation}\label{F-infty}
		F_\infty
		=
		\exp\Bigg\{
		\frac{i}{\pi}
		\Bigg(
		-\int_{E_1}^{E_2}
		\frac{s\log\varrho(s)}{\mathcal R_{1+}(s)}\,ds
		+\int_{\bar E_2}^{\bar E_1}
		\frac{s\log\varrho^*(s)}{\mathcal R_{1+}(s)}\,ds
		+\int_{E_1}^{\bar E_1}
		\frac{s\kappa}{\mathcal R_1(s)}\,ds
		\Bigg)
		\Bigg\}.
	\end{equation}
	The logarithm branches in \eqref{F-infty} are those fixed in the scalar RHP.
	Therefore, it follows that
	\[
	u(x,0)
	=\mathcal U_-(x)+\mathcal O(|x|^{-1}),
	\qquad x\to-\infty.
	\]
	
	\begin{remark}
		The full hyperelliptic curve in the $z$-plane has genus $3$, but the theta
		function in $\mathcal U_-(x)$ is one-dimensional because the model RHP
		descends to the quotient curve $\widehat{\mathcal X}_1$. This is the
		large-$x$ manifestation of the genus-reduction theorem in
		Section~\ref{sec:genus-reduction}.
	\end{remark}

	\section{Long-time asymptotics of the elliptic soliton gas}\label{long-time- asymptotics}
	
	This section proves the Theorem \ref{time degen} and Theorem \ref{time gener} by studying the long-time asymptotic behaviors of the elliptic soliton gas potential (\ref{eq:elliptic-reconstruction}) in the self-similar regime
	\begin{equation}
		\xi=\frac{x}{t},
		\qquad
		t\to+\infty.
	\end{equation}
	Our aim is to determine the leading asymptotic behavior of the solution in the different sectors of the upper \((x,t)\)-plane selected by the parameter \(\xi\).
	The analysis is based on the Deift--Zhou nonlinear steepest descent method \cite{Deift-Zhou} for the RHP~\ref{prob:elliptic-rhp}. The oscillatory phase in the jump matrix \eqref{eq:elliptic-jump-matrix} is
	\begin{equation}
		\theta(x,t,z)
		=
		xz^2+2tz^4
		=
		t\bigl(\xi z^2+2z^4\bigr).
	\end{equation}
	Based on the transformation $\lambda=z^2$ given in Section~\ref{sec:genus-reduction}, it is natural to introduce the quotient spectral variable
	\begin{equation}
		\widehat{\theta}(\xi,\lambda)
		=
		\xi\lambda+2\lambda^2.
	\end{equation}
	This transformation incorporates the even spectral symmetry of the spectral problem associated with the DNLS equation and reduces the long-time analysis to an effective quotient genus problem on the quotient spectral plane.
	\par
	The asymptotic behavior is determined by the critical graph of an appropriate \(g\)-function in the \(\lambda\)-plane. As \(\xi\) varies, the topology of this graph changes, giving rise to different finite-gap sectors described by Riemann theta functions, together with a right-most exponentially decaying sector. In the following subsections, we construct the relevant \(g\)-functions, derive the associated model RHP, and obtain the long-time asymptotic formulas through the corresponding local parametrices and small norm error analysis.

	\begin{remark}
		All asymptotic formulas are uniform for \(\xi=x/t\) in compact subsets of each open sector. The transition regimes near the critical rays are not considered here, since they require a separate local analysis.
	\end{remark}

	Denote the image of the upper mother body segment \(\mathscr I=[E_1,E_2]\) under the map \(z\mapsto z^2\) as
	\begin{equation}
		\Lambda=\{z^2:z\in \mathscr I\},
	\end{equation}
	and let
	$
	\Lambda^*=\overline{\Lambda}
	$
	be its Schwarz conjugate image. We have
	\begin{equation}\label{eq:im-quotient-phase}
		\operatorname{Im}\widehat{\theta}(\xi,\lambda)
		=
		\operatorname{Im}(\xi\lambda+2\lambda^2)
		=
		\operatorname{Im}\lambda\bigl(\xi+4\operatorname{Re}\lambda\bigr).
	\end{equation}
	Since \(\Lambda\) lies in the upper half-plane, one has \(\operatorname{Im}\lambda>0\) on \(\Lambda\). Hence
	$
	\operatorname{Im}\widehat{\theta}(\xi,\lambda)>0,
	$
	provided that
	\[
	\xi+4\operatorname{Re}\lambda>0,
	\quad
	\lambda\in\Lambda.
	\]
	This leads to the definition of the right vacuum boundary
	\begin{equation}
		\xi_{\mathrm{vac}}
		:=
		-4\min_{\lambda\in\Lambda}\operatorname{Re}\lambda.
		\label{eq:vacuum-boundary}
	\end{equation}

	\paragraph{The Vacuum Sector:}
	
	We first consider the sector
	$
	\xi>\xi_{\mathrm{vac}}.
	$
	By \eqref{eq:im-quotient-phase} and the definition of
	\(\xi_{\mathrm{vac}}\) in (\ref{eq:vacuum-boundary}), the phase funciton has the favorable sign on
	\(\Lambda\). More precisely, for every compact set
	$
	K\Subset(\xi_{\mathrm{vac}},+\infty),
	$
	there exists a constant \(c_K>0\) such that
	$
	\operatorname{Im}\widehat{\theta}(\xi,\lambda)
	\geq c_K.
	$
	Hence,
	\begin{equation}
		\left|e^{2it\widehat{\theta}(\xi,\lambda)}\right|
		\leq e^{-2c_Kt},
		\qquad
		\xi\in K,\quad \lambda\in\Lambda.
	\end{equation}
	By the Schwarz symmetry, the same exponential decay holds for the
	corresponding jump entries on \(\Lambda^*\). Therefore,
	\[
	\|J(\xi,z)-\mathbb I\|_{L^2(\Sigma_{\mathrm{ell}})
		\cap L^\infty(\Sigma_{\mathrm{ell}})}
	=\mathcal{O}(e^{-2c_Kt}),
	\qquad
	\xi\in K.
	\]
	The RHP is therefore a small norm problem.  Hence, uniformly for $\xi\in K$, it follows that
	\[
	m(\xi,z)
	=
	\mathbb{I}+\mathcal{O}(e^{-2c_Kt}),
	\qquad
	t\to+\infty,
	\]
	for $z$ in compact subsets of $\mathbb C\setminus\Sigma_{\rm ell}$.  Moreover, the expansion at infinity satisfies
	\[
	m(\xi,z)
	=
	\mathbb{I}+\frac{m_1(\xi,z)}{z}+\mathcal{O}(z^{-2}),
	\qquad
	m_1(\xi,z)=\mathcal{O}(e^{-2c_Kt}).
	\]
	Using the reconstruction formula \eqref{eq:elliptic-reconstruction}, we obtain that
	\begin{equation}
		u(x,t)
		=
		\mathcal{O}(e^{-2c_Kt}),
		\qquad
		\xi\in K.
		\label{eq:vacuum-asymptotics}
	\end{equation}
	Thus, \(\xi>\xi_{\mathrm{vac}}\) is a vacuum sector, whose outer model is
	the identity matrix, and no nontrivial \(g\)-function or
	parametrix expressed by Riemann theta function is required.

	\paragraph{The initial genus-1 \(g\)-function:}
	
	We now turn to the oscillatory region
	$\xi<\xi_{\mathrm{vac}}$.
	In this case, the original jump matrix is no longer uniformly close to the identity. Indeed, by \eqref{eq:im-quotient-phase},
	there is a portion of \(\Lambda\) on which
	$
	\operatorname{Im}\widehat{\theta}(\xi,\lambda)<0.
	$
	Hence the exponential factor \(e^{2it\widehat{\theta}}\) grows on that part of the contour. A nonlinear phase deformation is therefore needed.
	
	The first possible deformation is given by a genus-1 quotient \(g\)-function. Set
	$
	A_j=E_j^2$,
	$j=1,2,
	$
	and introduce the quotient curve
	\begin{equation}\label{quotient-curve-t1}
		\widehat{\mathcal X}_1:
		\ 
		\widehat {\mathcal{R}}_1(\lambda)^2
		=
		(\lambda-A_1)(\lambda-A_2)
		(\lambda-\overline{A}_1)(\lambda-\overline{A}_2).
	\end{equation}
	We choose the branch normalized at infinity by
	\begin{equation}
		\widehat {\mathcal{R}}_1(\lambda)
		=
		\lambda^2-S_1\lambda+S_2+\mathcal{O}(\lambda^{-1}),
		\qquad
		\lambda\to\infty,
	\end{equation}
	where
	$
	S_1=\operatorname{Re}(A_1+A_2)
	$ and
	$	S_2
	=
	\frac{1}{2}
	\left[
	|A_1|^2+|A_2|^2
	+2\operatorname{Re}(A_1A_2)
	+2\operatorname{Re}(A_1\overline A_2)
	-
	\bigl(\operatorname{Re}(A_1+A_2)\bigr)^2
	\right].
	$

	Seek a meromorphic differential of the form
	\begin{equation}
		d\widehat\Omega_1(\xi,\lambda)
		=
		4\frac{P_3(\lambda,\xi)}{\widehat {\mathcal{R}}_1(\lambda)}d\lambda,
		\label{eq:genus-one-differential}
	\end{equation}
	where $P_3(\lambda,\xi)$ is a monic cubic polynomial given by
	\begin{equation}
		P_3(\lambda,\xi)
		=
		\lambda^3+a(\xi)\lambda^2+b(\xi)\lambda+w(\xi).
	\end{equation}
	The associated phase function is defined by Abelian integral as
	\begin{equation}
		\widehat\Phi_1(\xi,\lambda)
		:=
		\int_{A_2}^{\lambda}d\widehat\Omega_1(\xi,s)
		+
		\ell_1(\xi),
	\end{equation}
	where the constant \(\ell_1(\xi)\) is chosen to impose the normalization
	\[
	\widehat\Phi_{1,+}(\xi,\lambda)
	+
	\widehat\Phi_{1,-}(\xi,\lambda)
	=
	0
	\]
	on the main band.
	The coefficients of polynomial \(P_3(\lambda,\xi)\) are determined as follows. First, \(d\widehat\Omega_1\) must match the phase function in quotient space at infinity:
	\[
	d\widehat\Omega_1(\xi,\lambda)
	=
	\bigl(4\lambda+\xi+\mathcal{O}(\lambda^{-2})\bigr)d\lambda,
	\qquad
	\lambda\to\infty.
	\]
	Equivalently, the Abelian integral satisfies
	\begin{equation}
		\widehat\Phi_1(\xi,\lambda)
		=
		\widehat\theta(\xi,\lambda)
		+
		\ell_1(\xi)
		+
		\mathcal{O}(\lambda^{-1}),
		\qquad
		\lambda\to\infty.
	\end{equation}
	Expanding \eqref{eq:genus-one-differential} at infinity gives
	$a(\xi)=\frac{\xi}{4}-S_1$, and $b(\xi)=S_2-\frac{S_1\xi}{4}$.
	The remaining coefficient \(w(\xi)\) is fixed by the Boutroux condition
	\begin{equation}
		\operatorname{Im}
		\oint_{\widehat{a}_1}
		d\widehat\Omega_1(\xi,\lambda)
		=
		0,
		\label{eq:genus-one-boutroux}
	\end{equation}
	where \(\widehat{a}_1\) denotes the \(a\)-cycle of the quotient curve $\widehat{\mathcal X}_1$. This condition makes the imaginary part of \(\widehat\Phi_1\) compatible with the required sign distribution in the nonlinear steepest descent analysis.
	
	The genus-1 \(g\)-function in the original spectral variable is defined by
	\begin{equation}
		g_1(\xi,z)
		:=
		\widehat\Phi_1(\xi,z^2)
		-
		\theta(\xi,z),
		\label{eq:genus-one-g-function}
	\end{equation}
	which cancels the leading oscillatory phase on the main band and produces a model problem with constant jump. By the normalization at $z\to\infty$, we have
	$
	g_1(\xi,z)
	=
	g_{\infty,1}(\xi)+\mathcal{O}(z^{-2}).
	$

	\begin{figure}[htbp]
		\centering
		
		\begin{subfigure}[t]{0.48\textwidth}
			\centering
			\resizebox{\linewidth}{!}{%
				\begin{tikzpicture}[scale=0.75]
					\draw[thick] (-3,0) .. controls (-3.25,2) and (-2.75,3) .. (-3,4);
					\draw[thick] (-3,0) .. controls (-2.75,-2) and (-3.25,-3) .. (-3,-4);
					
					\draw[thick] (3,0) .. controls (2.75,2) and (3.25,3) .. (3,4);
					\draw[thick] (3,0) .. controls (3.25,-2) and (2.75,-3) .. (3,-4);

					\draw[->, line width=1pt] (-5,0) -- (5,0);
					\draw[->, line width=1pt] (0,-4) -- (0,4);
					
					\draw[dashed, line width=1pt] (1,-3.5) -- (1,3.5);
					\draw[dashed, line width=1pt] (-1,-3.5) -- (-1,3.5);
					
					\draw[postaction={decorate},thick] (1,1) .. controls (2,1.5) and (2.5,2.5) .. (1,3);
					
					\draw[postaction={decorate},thick] (1,-1) .. controls (2,-1.5) and (2.5,-2) .. (1,-3);
					
					\draw[postaction={decorate},line width=1pt] (1,-1) -- (1,1);
					
					\draw[postaction={decorate},thick] (-1,1) .. controls (-2,1.5) and (-2.5,2) .. (-1,3);
					
					\draw[postaction={decorate},thick] (-1,-1) .. controls (-2,-1.5) and (-2.5,-2) .. (-1,-3);
					
					\draw[postaction={decorate},line width=1pt] (-1,-1) -- (-1,1);
					
					\node at (0,0) [below left, font=\small] {$\mathbf{0}$};
					\filldraw[black] (0,0) circle (1.25pt);
					\node at (5,0) [below, font=\small] {$\operatorname{Re}(z)$};
					
					\node at (0,4) [right, font=\small] {$\operatorname{Im}(z)$};
					\node at (3,0) [below right, font=\scriptsize] {$\zeta_1$};
					\node at (-3,0) [below left, font=\scriptsize] {$-\zeta_1$};
					\filldraw[black] (3,0) circle (1.25pt);
					\filldraw[black] (-3,0) circle (1.25pt);
					
					\node at (1,1) [left,font=\scriptsize] {$E_1$};
					\node at (1,3) [left,font=\scriptsize] {$E_2$};
					\filldraw[black] (1,1) circle (1.25pt);
					\filldraw[black] (1,3) circle (1.25pt);
					
					\node at (1,-1) [left,font=\scriptsize] {$\bar{E}_1$};
					\node at (1,-3) [left,font=\scriptsize] {$\bar{E}_2$};
					\filldraw[black] (1,-1) circle (1.25pt);
					\filldraw[black] (1,-3) circle (1.25pt);
					
					\node at (-1,1) [right,font=\scriptsize] {$-\bar{E}_1$};
					\node at (-1,3) [right,font=\scriptsize] {$-\bar{E}_2$};
					\filldraw[black] (-1,1) circle (1.25pt);
					\filldraw[black] (-1,3) circle (1.25pt);
					
					\node at (-1,-1) [right,font=\scriptsize] {$-E_1$};
					\node at (-1,-3) [below right,font=\scriptsize] {$-E_2$};
					\filldraw[black] (-1,-1) circle (1.25pt);
					\filldraw[black] (-1,-3) circle (1.25pt);
				\end{tikzpicture}
			}
			\caption{$-\infty<\xi<\xi_0$.}
			\label{fig:sign-genus-1-gen}
		\end{subfigure}
		\hfill
		\begin{subfigure}[t]{0.48\textwidth}
			\centering
			\resizebox{\linewidth}{!}{%
				\begin{tikzpicture}[scale=1]
					\draw[->, line width=1pt] (-5,0) -- (5,0);
					\draw[->, line width=1pt] (0,-4) -- (0,4);
					
					\draw[dashed, line width=1pt] (1,-3.5) -- (1,3.5);
					\draw[dashed, line width=1pt] (-1,-3.5) -- (-1,3.5);
					
					\draw[postaction={decorate},thick] (1,1) .. controls (2,0.8) and (2.5,0.8) .. (3,1);
					\draw[postaction={decorate},thick] (1,-1) .. controls (2,-0.8) and (2.5,-0.8) .. (3,-1);
					\draw[postaction={decorate},thick] (-1,1) .. controls (-2,0.8) and (-2.5,0.8) .. (-3,1);
					\draw[postaction={decorate}, thick] (-1,-1) .. controls (-2,-0.8) and (-2.5,-0.8) .. (-3,-1);
					\draw[thick] (5,4) .. controls (4,3.6) and (3.5,3) .. (3,2.5);
					\draw[thick] (5,-4) .. controls (4,-3.6) and (3.5,-3) .. (3,-2.5);
					\draw[thick] (-5,4) .. controls (-4,3.6) and (-3.5,3) .. (-3,2.5);
					\draw[thick] (-5,-4) .. controls (-4,-3.6) and (-3.5,-3) .. (-3,-2.5);
					\draw[postaction={decorate}, postaction={decorate},thick] (1,3) .. controls (2,2.9) and (2.5,2.7) .. (3,2.5);
					\draw[postaction={decorate},postaction={decorate},thick] (1,-3) .. controls (2,-2.9) and (2.5,-2.7) .. (3,-2.5);
					\draw[postaction={decorate},postaction={decorate},thick] (-1,3) .. controls (-2,2.9) and (-2.5,2.7) .. (-3,2.5);
					\draw[postaction={decorate},postaction={decorate},thick] (-1,-3) .. controls (-2,-2.9) and (-2.5,-2.7) .. (-3,-2.5);
					\draw[postaction={decorate},postaction={decorate},thick] (3,2.5) .. controls (3.5,2) and (3.5,1.5) .. (3,1);
					\draw[postaction={decorate},postaction={decorate},thick] (-3,2.5) .. controls (-3.5,2) and (-3.5,1.5) .. (-3,1);
					\draw[postaction={decorate},postaction={decorate},thick] (3,-2.5) .. controls (3.5,-2) and (3.5,-1.5) .. (3,-1);
					\draw[postaction={decorate},postaction={decorate},thick] (-3,-2.5) .. controls (-3.5,-2) and (-3.5,-1.5) .. (-3,-1);
					\draw[postaction={decorate},postaction={decorate},postaction={decorate},postaction={decorate},thick] (3,1) .. controls (3.5,0.7) and (3.5,-0.7) .. (3,-1);
					\draw[postaction={decorate},postaction={decorate},postaction={decorate}, postaction={decorate},thick] (-3,1) .. controls (-3.5,0.7) and (-3.5,-0.7) .. (-3,-1);

					\node at (0,0) [below left, font=\large] {$\mathbf{0}$};
					\filldraw[black] (0,0) circle (1.25pt);
					\node at (5,0) [below, font=\large] {$\operatorname{Re}(z)$};
					\node at (0,4) [right, font=\large] {$\operatorname{Im}(z)$};
					
					\node at (3.37,0) [below right, font=\small] {$\zeta_1$};
					\node at (-3.37,0) [below left, font=\small] {$-\zeta_1$};
					\filldraw[black] (3.37,0) circle (1.25pt);
					\filldraw[black] (-3.37,0) circle (1.25pt);
					
					\node at (1,1) [left,font=\scriptsize] {$E_1$};
					\node at (1,3) [left,font=\scriptsize] {$E_2$};
					\filldraw[black] (1,1) circle (1.25pt);
					\filldraw[black] (1,3) circle (1.25pt);
					
					\node at (1,-1) [left,font=\scriptsize] {$\bar{E}_1$};
					\node at (1,-3) [left,font=\scriptsize] {$\bar{E}_2$};
					\filldraw[black] (1,-1) circle (1.25pt);
					\filldraw[black] (1,-3) circle (1.25pt);
					
					\node at (-1,1) [right,font=\scriptsize] {$-\bar{E}_1$};
					\node at (-1,3) [right,font=\scriptsize] {$-\bar{E}_2$};
					\filldraw[black] (-1,1) circle (1.25pt);
					\filldraw[black] (-1,3) circle (1.25pt);
					
					\node at (-1,-1) [right,font=\scriptsize] {$-E_1$};
					\node at (-1,-3) [below right,font=\scriptsize] {$-E_2$};
					\filldraw[black] (-1,-1) circle (1.25pt);
					\filldraw[black] (-1,-3) circle (1.25pt);
					
					\node at (3,2.5) [right,font=\scriptsize] {$b_2$};
					\node at (3,-2.5) [right,font=\scriptsize] {$\bar{b}_2$};
					\filldraw[black] (3,2.5) circle (1.25pt);
					\filldraw[black] (3,-2.5) circle (1.25pt);
					\node at (-3,2.5) [left,font=\scriptsize] {$-\bar{b}_2$};
					\node at (-3,-2.5) [left,font=\scriptsize] {$-b_2$};
					\filldraw[black] (-3,2.5) circle (1.25pt);
					\filldraw[black] (-3,-2.5) circle (1.25pt);
					
					\node at (3,1) [right,font=\scriptsize] {$b_1$};
					\node at (3,-1) [right,font=\scriptsize] {$\bar{b}_1$};
					\filldraw[black] (3,1) circle (1.25pt);
					\filldraw[black] (3,-1) circle (1.25pt);
					\node at (-3,1) [left,font=\scriptsize] {$-\bar{b}_1$};
					\node at (-3,-1) [left,font=\scriptsize] {$-b_1$};
					\filldraw[black] (-3,1) circle (1.25pt);
					\filldraw[black] (-3,-1) circle (1.25pt);
				\end{tikzpicture}
			}
			\caption{$\xi_0<\xi<\xi_{\mathrm{vac}}$.}
			\label{fig:sign-genus-3-gen}
		\end{subfigure}
		
		\caption{Signature charts of $\operatorname{Im} g$ in the generic case.
			In panel (a), the points $\zeta_1^2(\xi)$ is the zero of the genus-1 polynomial $P_3(\lambda,\xi)$.
			In panel (b), $B_j(\xi)=b_j(\xi)^2$, $j=1,2$, are the moving quotient branch points.
			Panel (a) shows the initial genus-1 sector, while panel (b) shows the genus-3 sector.}
		\label{fig:signature-two-sectors}
	\end{figure}
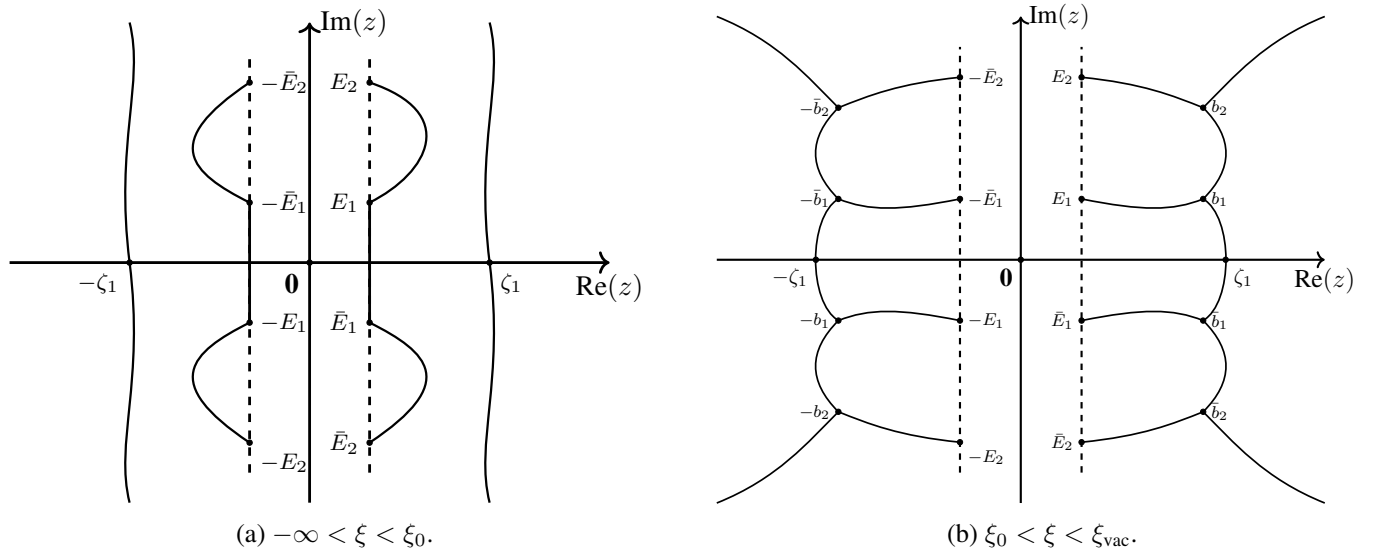

	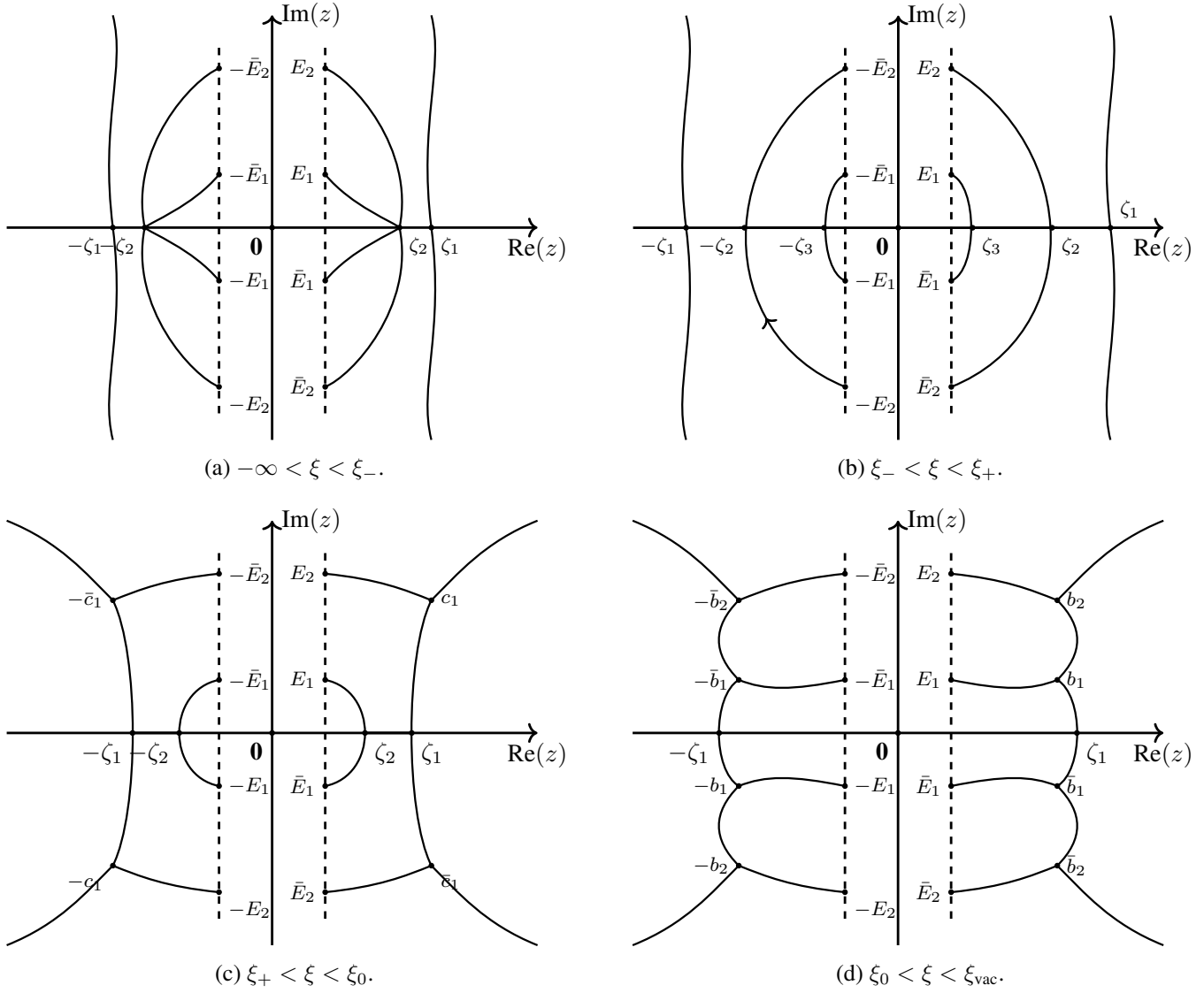
\begin{figure}[ht]
		\centering
		\begin{subfigure}[t]{0.48\textwidth}
			\centering
			\resizebox{\linewidth}{!}{%
				\begin{tikzpicture}[scale=0.75]
					\draw[thick] (-3,0) .. controls (-3.25,2) and (-2.75,3) .. (-3,4);
					\draw[thick] (-3,0) .. controls (-2.75,-2) and (-3.25,-3) .. (-3,-4);
					
					\draw[thick] (3,0) .. controls (2.75,2) and (3.25,3) .. (3,4);
					\draw[thick] (3,0) .. controls (3.25,-2) and (2.75,-3) .. (3,-4);

					\draw[->, line width=1pt] (-5,0) -- (5,0);
					\draw[->, line width=1pt] (0,-4) -- (0,4);
					
					\draw[dashed, line width=1pt] (1,-3.5) -- (1,3.5);
					\draw[dashed, line width=1pt] (-1,-3.5) -- (-1,3.5);
					
					\draw[postaction={decorate},thick] (1,1) .. controls (1.5,0.4) and (2.3,0.1) .. (2.4,0);
					\draw[postaction={decorate},thick] (2.4,0) .. controls (2.7,1.5) and (1.5,2.8) .. (1,3);

					\draw[postaction={decorate},thick] (1,-1) .. controls (1.5,-0.4) and (2.3,-0.1) .. (2.4,0);
					\draw[postaction={decorate},thick] (2.4,0) .. controls (2.7,-1.5) and (1.5,-2.8) .. (1,-3);

					\draw[postaction={decorate},thick] (-1,1) .. controls (-1.5,0.4) and (-2.3,0.1) .. (-2.4,0);
					\draw[postaction={decorate},thick] (-2.4,0) .. controls (-2.7,1.5) and (-1.5,2.8) .. (-1,3);
					
					\draw[postaction={decorate},thick] (-1,-1) .. controls (-1.5,-0.4) and (-2.3,-0.1) .. (-2.4,0);
					\draw[postaction={decorate},thick] (-2.4,0) .. controls (-2.7,-1.5) and (-1.5,-2.8) .. (-1,-3);
					
					\node at (0,0) [below left, font=\small] {$\mathbf{0}$};
					\filldraw[black] (0,0) circle (1.25pt);
					\node at (5,0) [below, font=\small] {$\operatorname{Re}(z)$};
					
					\node at (0,4) [right, font=\small] {$\operatorname{Im}(z)$};
					\node at (3,0) [below right, font=\scriptsize] {$\zeta_1$};
					\node at (-3,0) [below left, font=\scriptsize] {$-\zeta_1$};
					\filldraw[black] (3,0) circle (1.25pt);
					\filldraw[black] (-3,0) circle (1.25pt);
					\node at (2.4,0) [below right, font=\scriptsize] {$\zeta_2$};
					\node at (-2.4,0) [below left, font=\scriptsize] {$-\zeta_2$};
					\filldraw[black] (2.4,0) circle (1.25pt);
					\filldraw[black] (-2.4,0) circle (1.25pt);

					\node at (1,1) [left,font=\scriptsize] {$E_1$};
					\node at (1,3) [left,font=\scriptsize] {$E_2$};
					\filldraw[black] (1,1) circle (1.25pt);
					\filldraw[black] (1,3) circle (1.25pt);
					
					\node at (1,-1) [left,font=\scriptsize] {$\bar{E}_1$};
					\node at (1,-3) [left,font=\scriptsize] {$\bar{E}_2$};
					\filldraw[black] (1,-1) circle (1.25pt);
					\filldraw[black] (1,-3) circle (1.25pt);
					
					\node at (-1,1) [right,font=\scriptsize] {$-\bar{E}_1$};
					\node at (-1,3) [right,font=\scriptsize] {$-\bar{E}_2$};
					\filldraw[black] (-1,1) circle (1.25pt);
					\filldraw[black] (-1,3) circle (1.25pt);
					
					\node at (-1,-1) [right,font=\scriptsize] {$-E_1$};
					\node at (-1,-3) [below right,font=\scriptsize] {$-E_2$};
					\filldraw[black] (-1,-1) circle (1.25pt);
					\filldraw[black] (-1,-3) circle (1.25pt);
				\end{tikzpicture}
			}
			\caption{$-\infty<\xi<\xi_-$.}
			\label{fig:sign-genus-1}
		\end{subfigure}
		\hfill
		\begin{subfigure}[t]{0.48\textwidth}
			\centering
			\resizebox{\linewidth}{!}{%
				\begin{tikzpicture}[scale=0.75]
					\draw[thick] (-4,0) .. controls (-4.25,2) and (-3.75,3) .. (-4,4);
					\draw[thick] (-4,0) .. controls (-3.75,-2) and (-4.25,-3) .. (-4,-4);
					
					\draw[thick] (4,0) .. controls (3.75,2) and (4.25,3) .. (4,4);
					\draw[thick] (4,0) .. controls (4.25,-2) and (3.75,-3) .. (4,-4);
					
					\draw[->, line width=1pt] (-5,0) -- (5,0);
					\draw[->, line width=1pt] (0,-4) -- (0,4);
					
					\draw[dashed, line width=1pt] (1,-3.5) -- (1,3.5);
					\draw[dashed, line width=1pt] (-1,-3.5) -- (-1,3.5);
					
					\draw[postaction={decorate},postaction={decorate}, thick] (1,1) .. controls (1.5,0.8) and (1.5,-0.8) .. (1,-1);
					\draw[postaction={decorate}, postaction={decorate}, thick] (1,3) .. controls (3.5,1.5) and (3.5,-2) .. (1,-3);

					\draw[postaction={decorate},postaction={decorate}, thick] (-1,1) .. controls (-1.5,0.8) and (-1.5,-0.8) .. (-1,-1);
					\draw[postaction={decorate},postaction={decorate}, decoration={markings, mark=at position 0.75 with {\arrow{<}}},thick] (-1,3) .. controls (-3.5,1.5) and (-3.5,-2) .. (-1,-3);

					\draw[postaction={decorate}, line width=1pt] (1.4,0) -- (3,0);
					\draw[postaction={decorate}, line width=1pt] (-1.4,0) -- (-3,0);
					
					\node at (0,0) [below left, font=\small] {$\mathbf{0}$};
					\filldraw[black] (0,0) circle (1.25pt);
					\node at (5,0) [below, font=\small] {$\operatorname{Re}(z)$};
					
					\node at (0,4) [right, font=\small] {$\operatorname{Im}(z)$};
					\node at (4,0) [above right, font=\scriptsize] {$\zeta_1$};
					\node at (-4,0) [below left, font=\scriptsize] {$-\zeta_1$};
					\filldraw[black] (4,0) circle (1.25pt);
					\filldraw[black] (-4,0) circle (1.25pt);
					
					\node at (2.9,0) [below right, font=\scriptsize] {$\zeta_2$};
					\node at (-2.9,0) [below left, font=\scriptsize] {$-\zeta_2$};
					\filldraw[black] (2.9,0) circle (1.25pt);
					\filldraw[black] (-2.9,0) circle (1.25pt);
					
					\node at (1.4,0) [below right, font=\scriptsize] {$\zeta_3$};
					\node at (-1.4,0) [below left, font=\scriptsize] {$-\zeta_3$};
					\filldraw[black] (1.4,0) circle (1.25pt);
					\filldraw[black] (-1.4,0) circle (1.25pt);

					\node at (1,1) [left,font=\scriptsize] {$E_1$};
					\node at (1,3) [left,font=\scriptsize] {$E_2$};
					\filldraw[black] (1,1) circle (1.25pt);
					\filldraw[black] (1,3) circle (1.25pt);
					
					\node at (1,-1) [left,font=\scriptsize] {$\bar{E}_1$};
					\node at (1,-3) [left,font=\scriptsize] {$\bar{E}_2$};
					\filldraw[black] (1,-1) circle (1.25pt);
					\filldraw[black] (1,-3) circle (1.25pt);
					
					\node at (-1,1) [right,font=\scriptsize] {$-\bar{E}_1$};
					\node at (-1,3) [right,font=\scriptsize] {$-\bar{E}_2$};
					\filldraw[black] (-1,1) circle (1.25pt);
					\filldraw[black] (-1,3) circle (1.25pt);
					
					\node at (-1,-1) [right,font=\scriptsize] {$-E_1$};
					\node at (-1,-3) [below right,font=\scriptsize] {$-E_2$};
					\filldraw[black] (-1,-1) circle (1.25pt);
					\filldraw[black] (-1,-3) circle (1.25pt);
				\end{tikzpicture}
			}
			\caption{$\xi_-<\xi<\xi_+$.}
			\label{fig:sign-genus-1s}
		\end{subfigure}
		
		\vspace{2mm}
		
		\begin{subfigure}[t]{0.48\textwidth}
			\centering
			\resizebox{\linewidth}{!}{%
				\begin{tikzpicture}[scale=0.75]
					\draw[->, line width=1pt] (-5,0) -- (5,0);
					\draw[->, line width=1pt] (0,-4) -- (0,4);
					
					\draw[dashed, line width=1pt] (1,-3.5) -- (1,3.5);
					\draw[dashed, line width=1pt] (-1,-3.5) -- (-1,3.5);
					
					\draw[postaction={decorate},postaction={decorate},line width=1pt] (1.75,0) -- (2.63,0);
					\draw[postaction={decorate},postaction={decorate},line width=1pt] (-1.75,0) -- (-2.63,0);

					
					\draw[postaction={decorate},thick] (1,1) .. controls (2,0.8) and (2,-0.8) .. (1,-1);
					
					\draw[postaction={decorate},postaction={decorate}, thick] (-1,1) .. controls (-2,0.8) and (-2,-0.8) .. (-1,-1);
					
					\draw[thick] (5,4) .. controls (4,3.6) and (3.5,3) .. (3,2.5);
					\draw[thick] (5,-4) .. controls (4,-3.6) and (3.5,-3) .. (3,-2.5);
					\draw[postaction={decorate},postaction={decorate},thick] (1,3) .. controls (2,2.9) and (2.5,2.7) .. (3,2.5);
					\draw[postaction={decorate},postaction={decorate},thick] (1,-3) .. controls (2,-2.9) and (2.5,-2.7) .. (3,-2.5);
					\draw[postaction={decorate},postaction={decorate},postaction={decorate},postaction={decorate},thick] (3,2.5) .. controls (2.5,1.5) and (2.5,-1.5) .. (3,-2.5);
					
					\draw[thick] (-5,4) .. controls (-4,3.6) and (-3.5,3) .. (-3,2.5);
					\draw[thick] (-5,-4) .. controls (-4,-3.6) and (-3.5,-3) .. (-3,-2.5);
					\draw[postaction={decorate},postaction={decorate},thick] (-1,3) .. controls (-2,2.9) and (-2.5,2.7) .. (-3,2.5);
					\draw[postaction={decorate},postaction={decorate},thick] (-1,-3) .. controls (-2,-2.9) and (-2.5,-2.7) .. (-3,-2.5);
					\draw[postaction={decorate},postaction={decorate},postaction={decorate},postaction={decorate},thick] (-3,2.5) .. controls (-2.5,1.5) and (-2.5,-1.5) .. (-3,-2.5);

					\node at (0,0) [below left, font=\small] {$\mathbf{0}$};
					\filldraw[black] (0,0) circle (1.25pt);
					\node at (5,0) [below, font=\small] {$\operatorname{Re}(z)$};
					\node at (0,4) [right, font=\small] {$\operatorname{Im}(z)$};
					
					\node at (2.63,0) [below right, font=\small] {$\zeta_1$};
					\node at (-2.63,0) [below left, font=\small] {$-\zeta_1$};
					\filldraw[black] (2.63,0) circle (1.25pt);
					\filldraw[black] (-2.63,0) circle (1.25pt);
					
					\node at (1.75,0) [below right, font=\small] {$\zeta_2$};
					\node at (-1.75,0) [below left, font=\small] {$-\zeta_2$};
					\filldraw[black] (1.75,0) circle (1.25pt);
					\filldraw[black] (-1.75,0) circle (1.25pt);

					\node at (1,1) [left,font=\scriptsize] {$E_1$};
					\node at (1,3) [left,font=\scriptsize] {$E_2$};
					\filldraw[black] (1,1) circle (1.25pt);
					\filldraw[black] (1,3) circle (1.25pt);
					
					\node at (1,-1) [left,font=\scriptsize] {$\bar{E}_1$};
					\node at (1,-3) [left,font=\scriptsize] {$\bar{E}_2$};
					\filldraw[black] (1,-1) circle (1.25pt);
					\filldraw[black] (1,-3) circle (1.25pt);
					
					\node at (-1,1) [right,font=\scriptsize] {$-\bar{E}_1$};
					\node at (-1,3) [right,font=\scriptsize] {$-\bar{E}_2$};
					\filldraw[black] (-1,1) circle (1.25pt);
					\filldraw[black] (-1,3) circle (1.25pt);
					
					\node at (-1,-1) [right,font=\scriptsize] {$-E_1$};
					\node at (-1,-3) [below right,font=\scriptsize] {$-E_2$};
					\filldraw[black] (-1,-1) circle (1.25pt);
					\filldraw[black] (-1,-3) circle (1.25pt);
					
					\node at (3,2.5) [right,font=\scriptsize] {$c_1$};
					\node at (3,-2.5) [below right,font=\scriptsize] {$\bar{c}_1$};
					\filldraw[black] (3,2.5) circle (1.25pt);
					\filldraw[black] (3,-2.5) circle (1.25pt);
					\node at (-3,2.5) [left,font=\scriptsize] {$-\bar{c}_1$};
					\node at (-3,-2.5) [below left,font=\scriptsize] {$-c_1$};
					\filldraw[black] (-3,2.5) circle (1.25pt);
					\filldraw[black] (-3,-2.5) circle (1.25pt);
				\end{tikzpicture}
			}
			\caption{$\xi_+<\xi<\xi_0$.}
			\label{fig:sign-genus-2}
		\end{subfigure}
		\hfill
		\begin{subfigure}[t]{0.48\textwidth}
			\centering
			\resizebox{\linewidth}{!}{%
				\begin{tikzpicture}[scale=0.75]
					\draw[->, line width=1pt] (-5,0) -- (5,0);
					\draw[->, line width=1pt] (0,-4) -- (0,4);
					
					\draw[dashed, line width=1pt] (1,-3.5) -- (1,3.5);
					\draw[dashed, line width=1pt] (-1,-3.5) -- (-1,3.5);
					
					\draw[postaction={decorate},thick] (1,1) .. controls (2,0.8) and (2.5,0.8) .. (3,1);
					\draw[postaction={decorate},thick] (1,-1) .. controls (2,-0.8) and (2.5,-0.8) .. (3,-1);
					\draw[postaction={decorate},thick] (-1,1) .. controls (-2,0.8) and (-2.5,0.8) .. (-3,1);
					\draw[postaction={decorate}, thick] (-1,-1) .. controls (-2,-0.8) and (-2.5,-0.8) .. (-3,-1);
					\draw[thick] (5,4) .. controls (4,3.6) and (3.5,3) .. (3,2.5);
					\draw[thick] (5,-4) .. controls (4,-3.6) and (3.5,-3) .. (3,-2.5);
					\draw[thick] (-5,4) .. controls (-4,3.6) and (-3.5,3) .. (-3,2.5);
					\draw[thick] (-5,-4) .. controls (-4,-3.6) and (-3.5,-3) .. (-3,-2.5);
					\draw[postaction={decorate}, postaction={decorate},thick] (1,3) .. controls (2,2.9) and (2.5,2.7) .. (3,2.5);
					\draw[postaction={decorate},postaction={decorate},thick] (1,-3) .. controls (2,-2.9) and (2.5,-2.7) .. (3,-2.5);
					\draw[postaction={decorate},postaction={decorate},thick] (-1,3) .. controls (-2,2.9) and (-2.5,2.7) .. (-3,2.5);
					\draw[postaction={decorate},postaction={decorate},thick] (-1,-3) .. controls (-2,-2.9) and (-2.5,-2.7) .. (-3,-2.5);
					\draw[postaction={decorate},postaction={decorate},thick] (3,2.5) .. controls (3.5,2) and (3.5,1.5) .. (3,1);
					\draw[postaction={decorate},postaction={decorate},thick] (-3,2.5) .. controls (-3.5,2) and (-3.5,1.5) .. (-3,1);
					\draw[postaction={decorate},postaction={decorate},thick] (3,-2.5) .. controls (3.5,-2) and (3.5,-1.5) .. (3,-1);
					\draw[postaction={decorate},postaction={decorate},thick] (-3,-2.5) .. controls (-3.5,-2) and (-3.5,-1.5) .. (-3,-1);
					\draw[postaction={decorate},postaction={decorate},postaction={decorate},postaction={decorate},thick] (3,1) .. controls (3.5,0.7) and (3.5,-0.7) .. (3,-1);
					\draw[postaction={decorate},postaction={decorate},postaction={decorate}, postaction={decorate},thick] (-3,1) .. controls (-3.5,0.7) and (-3.5,-0.7) .. (-3,-1);

					\node at (0,0) [below left, font=\small] {$\mathbf{0}$};
					\filldraw[black] (0,0) circle (1.25pt);
					\node at (5,0) [below, font=\small] {$\operatorname{Re}(z)$};
					\node at (0,4) [right, font=\small] {$\operatorname{Im}(z)$};
					
					\node at (3.37,0) [below right, font=\small] {$\zeta_1$};
					\node at (-3.37,0) [below left, font=\small] {$-\zeta_1$};
					\filldraw[black] (3.37,0) circle (1.25pt);
					\filldraw[black] (-3.37,0) circle (1.25pt);
					
					\node at (1,1) [left,font=\scriptsize] {$E_1$};
					\node at (1,3) [left,font=\scriptsize] {$E_2$};
					\filldraw[black] (1,1) circle (1.25pt);
					\filldraw[black] (1,3) circle (1.25pt);
					
					\node at (1,-1) [left,font=\scriptsize] {$\bar{E}_1$};
					\node at (1,-3) [left,font=\scriptsize] {$\bar{E}_2$};
					\filldraw[black] (1,-1) circle (1.25pt);
					\filldraw[black] (1,-3) circle (1.25pt);
					
					\node at (-1,1) [right,font=\scriptsize] {$-\bar{E}_1$};
					\node at (-1,3) [right,font=\scriptsize] {$-\bar{E}_2$};
					\filldraw[black] (-1,1) circle (1.25pt);
					\filldraw[black] (-1,3) circle (1.25pt);
					
					\node at (-1,-1) [right,font=\scriptsize] {$-E_1$};
					\node at (-1,-3) [below right,font=\scriptsize] {$-E_2$};
					\filldraw[black] (-1,-1) circle (1.25pt);
					\filldraw[black] (-1,-3) circle (1.25pt);
					
					\node at (3,2.5) [right,font=\scriptsize] {$b_2$};
					\node at (3,-2.5) [right,font=\scriptsize] {$\bar{b}_2$};
					\filldraw[black] (3,2.5) circle (1.25pt);
					\filldraw[black] (3,-2.5) circle (1.25pt);
					\node at (-3,2.5) [left,font=\scriptsize] {$-\bar{b}_2$};
					\node at (-3,-2.5) [left,font=\scriptsize] {$-b_2$};
					\filldraw[black] (-3,2.5) circle (1.25pt);
					\filldraw[black] (-3,-2.5) circle (1.25pt);
					
					\node at (3,1) [right,font=\scriptsize] {$b_1$};
					\node at (3,-1) [right,font=\scriptsize] {$\bar{b}_1$};
					\filldraw[black] (3,1) circle (1.25pt);
					\filldraw[black] (3,-1) circle (1.25pt);
					\node at (-3,1) [left,font=\scriptsize] {$-\bar{b}_1$};
					\node at (-3,-1) [left,font=\scriptsize] {$-b_1$};
					\filldraw[black] (-3,1) circle (1.25pt);
					\filldraw[black] (-3,-1) circle (1.25pt);
				\end{tikzpicture}
			}
			\caption{$\xi_0<\xi< \xi_{\mathrm{vac}}$.}
			\label{fig:sign-genus-3}
		\end{subfigure}
		
		\caption{Signature charts of $\operatorname{Im} g$ in the degenerate case.
			The points $\zeta_j^2(\xi)$, $j=1,2,3$, are the roots of $P_3(\lambda,\xi)$ in the quotient $\lambda$-plane.
			Panels (a)--(d) show, respectively, the initial genus-1 sector, the genus-$1_s$ sector, the genus-2 sector, and the genus-3 sector.}
		\label{fig:signature-four-sectors}
	\end{figure}

	\subsection{Mechanism of the phase transitions}
	\label{phase-transitions}
	
	We now describe the mechanism selecting the different long-time phase diagrams. The starting point is the genus-1 quotient differential introduced above. Since the critical points of the Abelian phase are determined by
	\[
	d\widehat{\Phi}_1(\xi,\lambda)=0
	\quad\Longleftrightarrow\quad
	P_3(\lambda,\xi)=0,
	\]
	the topology of the signature chart of
	\[
	h_1(\xi,\lambda)
	:=\operatorname{Im}\widehat{\Phi}_1(\xi,\lambda)
	\]
	is governed by the roots of the cubic polynomial \(P_3(\lambda,\xi)\).
	
	The coefficients of polynomial \(P_3(\lambda,\xi)\) are obtained from the matching condition at
	infinity and the genus-1 Boutroux condition. In particular,
	\[
	P_3(\lambda,\xi)
	=
	\lambda^3+
	\left(\frac{\xi}{4}-S_1\right)\lambda^2
	+
	\left(S_2-\frac{S_1\xi}{4}\right)\lambda
	+w_0+w_1\xi ,
	\]
	where \(w_0,w_1\) are determined by the period integrals
	\[
	I_k=\oint_{\widehat a_1}
	\frac{\lambda^k}{\widehat{\mathcal R}_1(\lambda)}d\lambda,
	\qquad k=0,1,2,3 .
	\]
	Hence the polynomial \(P_3(\lambda,\xi)\) has real coefficients, and its roots are either three real
	roots or one real root together with a non-real conjugate pair.
	
	The transition of the critical graph is detected by two scalar
	quantities. The first one is the discriminant
	\[
	D_3(\xi)=\operatorname{disc}_{\lambda}P_3(\lambda,\xi).
	\]
	Its sign determines the configuration of the critical points:
	\[
	D_3(\xi)>0
	\quad\Longleftrightarrow\quad
	\text{three distinct real critical points},
	\]
	while
	\[
	D_3(\xi)<0
	\quad\Longleftrightarrow\quad
	\text{one real critical point and one conjugate pair}.
	\]
	Therefore \(D_3(\xi)=0\) describes collisions of critical points.
	
	The second mechanism occurs when a non-real critical point reaches the
	zero level set of the phase. Let \(d(\xi)\) be the upper half-plane root
	of \(P_3(\lambda,\xi)\) in the region \(D_3(\xi)<0\), and define
	\[
	H_1(\xi)
	=
	\operatorname{Im}\widehat{\Phi}_1(\xi,d(\xi)).
	\]
	A transversal crossing
	\[
	H_1(\xi_0)=0,\qquad H_1'(\xi_0)\neq0
	\]
	means that the critical value crosses the zero level and the genus-1 signature chart changes its admissibility.
	
	Before analyzing these two transitions, we impose the following
	nondegeneracy assumptions. On every interval where a fixed genus
	description is used, we assume:
	
	\begin{enumerate}
		\item no critical point collides with a fixed or existing moving branch
		point;
		
		\item all critical points are simple except at the explicitly described
		transition values;
		
		\item no additional finite saddle connection occurs for the quadratic
		differential $(d\widehat{\Phi}_n)^2$;
		
		\item all transitions are transversal.
	\end{enumerate}
	
	Under these assumptions, the signature chart is stable away from the
	transition points. More precisely, the zero level set
	\[
	\Gamma_1(\xi)
	=
	\{\lambda:\operatorname{Im}\widehat{\Phi}_1(\xi,\lambda)=0\}
	\]
	is isotopic as \(\xi\) varies inside such an interval. Therefore, once the
	signature inequalities required for the nonlinear steepest descent method
	are verified at one reference value, they remain valid throughout the
	whole interval.
	
	There are two possible local mechanisms for changing the genus.
	
	\paragraph{Opening from a real critical collision.} Suppose that
	\[
	D_3(\xi_*)=0,\qquad D_3'(\xi_*)\neq0 .
	\]
	Then two real critical points collide at a double critical point.
	After the collision, a new conjugate pair of moving branch points may
	open. 
	Consequently, a real critical collision increases the quotient genus by
	one:
	\[
	n\longrightarrow n+1 .
	\]
	
	\paragraph{Opening from a non-real critical crossing.}
	If a simple non-real critical point reaches the zero level,
	then both it and its conjugate contribute to the opening process. Two
	conjugate pairs of moving branch points are generated, and therefore
	\[
	n\longrightarrow n+2 .
	\]
	
	These two mechanisms completely determine the possible phase diagrams
	under the above nondegeneracy assumptions.
	
	\paragraph{Generic case:} Assume that
	\[
	D_3(\xi)<0,
	\qquad \xi<\xi_{\mathrm{vac}},
	\]
	and that \(H_1(\xi)\) is strictly monotone with opposite signs at
	\(-\infty\) and \(\xi_{\mathrm{vac}}^{-}\). Then there exists a unique
	transition value \(\xi_0\) satisfying
	\[
	H_1(\xi_0)=0,
	\qquad D_3(\xi_0)<0 .
	\]
	
	For
	\(\xi<\xi_0\), the genus-1 signature chart remains admissible.
	At \(\xi=\xi_0\), a non-real critical point crosses the zero level and
	opens two conjugate pairs of branch points. Hence the genus of the associated quotient Riemann surface
	increases from one to three. No further transition occurs before
	\(\xi_{\mathrm{vac}}\). Therefore,
	\[
	n=1\longrightarrow n=3\longrightarrow\mathrm{vacuum},
	\]
	as depicted in Figure~\ref{fig:sign-genus-1-gen}.
	\paragraph{Degenerate case:} Assume that \(D_3(\xi)\) has a single maximum and satisfies
	\[
	\lim_{\xi\to-\infty}D_3(\xi)<0,
	\qquad
	D_3(\xi_c)>0,
	\qquad
	\lim_{\xi\uparrow\xi_{\mathrm{vac}}}D_3(\xi)<0 .
	\]
	Then there exist two simple zeros $\xi_-$ and $\xi_+ $ satisfying
	\[
	\xi_-<\xi_+ .
	\]
	\par
	At \(\xi=\xi_-\), two critical points collide, but the resulting
	signature chart remains admissible. The genus of quotient Riemann surface therefore is still
	one, although the critical graph changes. We denote this sector by
	\(1_s\).
	
	At \(\xi=\xi_+\), the genus-\(1_s\) chart loses admissibility. A real
	critical collision opens one conjugate pair of moving branch points and
	produces a genus-2 continuation:
	\[
	n=1_s\longrightarrow n=2 .
	\]
	
	For the genus-2 phase, let \(C_1(\xi)\) be the moving branch point and
	define the residual polynomial
	\[
	Q_2(\lambda,\xi)
	=
	\frac{P_4(\lambda,\xi)}
	{(\lambda-C_1(\xi))
		(\lambda-\overline{C_1(\xi)})}.
	\]
	The next transition is detected by
	\[
	\mathcal D_2(\xi)
	=
	\operatorname{disc}_{\lambda}Q_2(\lambda,\xi).
	\]
	Assuming that \(\mathcal D_2(\xi)\) decreases monotonically and changes sign,
	there exists a unique $\xi_0>\xi_+$
	such that $\mathcal D_2(\xi_0)=0$.
	At this point another real critical collision occurs and opens a second
	conjugate pair of branch points, giving the genus-3 phase.
	
	Hence the degenerate phase diagram as depicted in Figure~\ref{fig:signature-four-sectors} is
	\[
	n=1
	\longrightarrow
	n=1_s
	\longrightarrow
	n=2
	\longrightarrow
	n=3
	\longrightarrow
	\mathrm{vacuum}.
	\]

	\subsection{The \(g\)-functions in the finite-genus sectors}
	
	We now describe the \(g\)-function construction in the finite-genus
	sectors. For each sector described by finite-genus solution, introduce a quotient curve
	\begin{equation}\label{quotient-curve-n}
		\widehat{\mathcal X}_n:
		\ 
		\widehat{\mathcal R}_n(\lambda)^2
		=
		\prod_{j=1}^{2n+2}(\lambda-\alpha_j),
		\qquad
		n=1,2,3,
	\end{equation}
	where the branch points occur in Schwarz conjugate pairs and the branch
	of \(\widehat{\mathcal R}_n\) is normalized by
	\[
	\widehat{\mathcal R}_n(\lambda)
	=
	\lambda^{n+1}
	\left(1+\mathcal O(\lambda^{-1})\right),
	\qquad
	\lambda\to\infty.
	\]
	All branch points are assumed
	to be distinct. The corresponding Abelian differential is in the following form
	\begin{equation}\label{Abelian-differential-n}
		d\widehat\Phi_n(\xi,\lambda)
		=
		4\frac{P_{n+2}(\lambda,\xi)}
		{\widehat{\mathcal R}_n(\lambda)}\,d\lambda,
	\end{equation}
	where \(P_{n+2}(\lambda,\xi)\) is a monic polynomial of degree \(n+2\).
	The Schwarz symmetry is imposed in the form
	\[
	\overline{P_{n+2}(\overline\lambda,\xi)}
	=
	P_{n+2}(\lambda,\xi),
	\]
	so that the coefficients of \(P_{n+2}\) are real.
	
	The coefficients of \(P_{n+2}\), together with the moving branch
	points when \(n\geq2\), are fixed by a closed system of modulation
	conditions. First, the differential must reproduce the quotient phase
	at infinity:
	\begin{equation}\label{infity-condition-n}
		d\widehat\Phi_n(\xi,\lambda)
		=
		\left(
		4\lambda+\xi+\mathcal O(\lambda^{-2})
		\right)d\lambda,
		\qquad
		\lambda\to\infty.
	\end{equation}
	Second, the Boutroux conditions are imposed:
	\begin{equation}\label{Boutroux-conditions}
		\operatorname{Im}
		\oint_{\widehat a_j}
		d\widehat\Phi_n(\xi,\lambda)
		=
		0,
		\qquad
		j=1,\ldots,n.
	\end{equation}
	
	For \(n\geq2\), the moving branch points must also be soft endpoints of
	the Abelian phase. To state this condition uniformly, set
	\[
	\beta_1=C_1
	\quad\text{for }n=2,
	\qquad
	(\beta_1,\beta_2)=(B_1,B_2)
	\quad\text{for }n=3.
	\]
	The additional modulation conditions are
	\begin{equation}\label{soft-edge-conditions}
		P_{n+2}(\beta_k(\xi),\xi)=0,
		\qquad
		k=1,\ldots,n-1.
	\end{equation}
	The corresponding equations at
	\(\overline{\beta_k(\xi)}\) follow from Schwarz symmetry.
	The necessity of \eqref{soft-edge-conditions} follows from the local
	behavior of the phase. Indeed, near a moving branch point \(\beta_k\),
	write
	\[
	\widehat{\mathcal R}_n(\lambda)
	=
	(\lambda-\beta_k)^{1/2}
	\widehat{\mathcal R}_{n,k}(\lambda),
	\qquad
	\widehat{\mathcal R}_{n,k}(\beta_k)\neq0.
	\]
	If \(P_{n+2}(\beta_k,\xi)\neq0\), then
	\[
	\widehat\Phi_n(\xi,\lambda)
	-
	\widehat\Phi_n(\xi,\beta_k)
	=
	\frac{8P_{n+2}(\beta_k,\xi)}
	{\widehat{\mathcal R}_{n,k}(\beta_k)}
	(\lambda-\beta_k)^{1/2}
	\left(1+\mathcal O(\lambda-\beta_k)\right).
	\]
	On the other
	hand, if
	\[
	P_{n+2}(\beta_k,\xi)=0,
	\qquad
	\partial_\lambda P_{n+2}(\beta_k,\xi)\neq0,
	\]
	then
	\begin{equation}\label{soft-edge-expansion}
		\widehat\Phi_n(\xi,\lambda)
		-
		\widehat\Phi_n(\xi,\beta_k)
		=
		\frac{8\,\partial_\lambda
			P_{n+2}(\beta_k,\xi)}
		{3\widehat{\mathcal R}_{n,k}(\beta_k)}
		(\lambda-\beta_k)^{3/2}
		\left(1+\mathcal O(\lambda-\beta_k)\right).
	\end{equation}
	The nonvanishing derivative in
	\eqref{soft-edge-expansion} holds in the interior of an open sector;
	its failure signals a further degeneration of the modulation system.
	
	Let \(\Sigma_{\mathrm{ban}}^n\) and
	\(\Sigma_{\mathrm{gap}}^n\) denote the bands and gaps of the resulting
	configuration. Since
	\(\widehat{\mathcal R}_{n,+}=-\widehat{\mathcal R}_{n,-}\) on a band,
	one has
	\[
	d\widehat\Phi_{n,+}
	+
	d\widehat\Phi_{n,-}=0.
	\]
	In terms of the original \(z\)-variable, the corresponding
	\(g\)-function is
	\begin{equation}\label{g-function-n}
		g_n(\xi,z)
		=
		\widehat\Phi_n(\xi,z^2)
		-
		\theta(\xi,z).
	\end{equation}
	By \eqref{infity-condition-n}, it follows that
	\[
	g_n(\xi,z)
	=
	g_{\infty,n}(\xi)
	+
	\mathcal O(z^{-2}),
	\qquad
	z\to\infty.
	\]
	On the remaining arcs, the sign of
	\(\operatorname{Im}\widehat\Phi_n\) is required to be such that the
	corresponding exponential factors decay as \(t\to+\infty\).
	
	It is important to distinguish the branch points of
	\(\widehat{\mathcal X}_n\) from the zeros of \(P_{n+2}\). The branch
	points are endpoints of the main bands and determine the genus of the
	quotient curve. The zeros of \(P_{n+2}\) are the critical points of
	the Abelian phase:
	\[
	d\widehat\Phi_n(\xi,\lambda)=0
	\quad\Longleftrightarrow\quad
	P_{n+2}(\lambda,\xi)=0.
	\]
	The moving branch points satisfy both conditions: they are branch
	points of the curve and, by \eqref{soft-edge-conditions}, zeros of the
	numerator. The fixed points
	\(A_1,A_2,\overline A_1,\overline A_2\), however, are not required to
	be zeros of \(P_{n+2}\). The remaining zeros of \(P_{n+2}\) determine
	the topology of the critical graph of
	\(\operatorname{Im}\widehat\Phi_n\).
	
	\paragraph{Genus-1 sector.}
	In the genus-1 sector, the quotient curve only has the four fixed
	branch points
	\[
	\mathcal A_1
	=
	\{A_1,A_2,\overline A_1,\overline A_2\}.
	\]
	Thus \(\widehat{\mathcal R}_1(\lambda)\) is defined by
	\eqref{quotient-curve-t1}, and
	\[
	d\widehat\Phi_1(\xi,\lambda)
	=
	4\frac{P_3(\lambda,\xi)}
	{\widehat{\mathcal R}_1(\lambda)}\,d\lambda.
	\]
	There are three real coefficients in \(P_3\). The infinity
	condition \eqref{infity-condition-n} and the Boutroux condition \eqref{Boutroux-conditions} determine them uniquely,
	provided that the corresponding period denominator is nonzero. No
	soft edge condition is imposed, since no moving branch point is
	present.
	The zeros of \(P_3\) determine the genus-1 critical graph. This sector
	remains valid as long as these zeros are simple, remain separated from
	the fixed branch points, and produce the required strict sign
	distribution. Its possible degenerations are the critical point
	collision
	\[
	D_3(\xi)
	=
	\operatorname{disc}_{\lambda}P_3(\lambda,\xi)
	=
	0
	\]
	and the zero-level crossing
	\[
	H_1(\xi)
	=
	\operatorname{Im}
	\widehat\Phi_1(\xi,d(\xi))
	=
	0.
	\]
	In the degenerate phase diagram, the first collision changes the
	critical graph without changing the branch point set. The quotient
	curve still has genus one, and the resulting sector is denoted by
	\(n=1_s\).
	
	\paragraph{Genus-2 sector.}
	The genus-2 sector appears when one conjugate pair of moving branch
	points must be introduced. Write
	\[
	\mathcal A_2
	=
	\mathcal A_1
	\cup
	\{C_1(\xi),\overline{C_1(\xi)}\},
	\]
	Then the quotient curve is
	\begin{equation}\label{quotient-curve-t2}
		\widehat{\mathcal R}_2(\lambda)^2
		=
		\prod_{\alpha\in\mathcal A_2}
		(\lambda-\alpha),
	\end{equation}
	and the corresponding differential is
	\[
	d\widehat\Phi_2(\xi,\lambda)
	=
	4\frac{P_4(\lambda,\xi)}
	{\widehat{\mathcal R}_2(\lambda)}\,d\lambda.
	\]
	The complete genus-2 modulation system consists of the two infinity
	conditions, the two Boutroux conditions, and
	\begin{equation}\label{genus-two-soft-edge}
		P_4(C_1(\xi),\xi)=0.
	\end{equation}
	Its Schwarz conjugate gives
	\(P_4(\overline{C_1(\xi)},\xi)=0\). Thus the four real coefficients of
	\(P_4\) and the two real components of \(C_1\) are determined by six
	real equations. 
	
	At the transition from the genus-\(1_s\) sector, the new conjugate
	branch points are born from a real double critical point \(c_+\). The
	matching with the genus-1 differential requires
	\[
	C_1(\xi),\overline{C_1(\xi)}
	\longrightarrow c_+,
	\]
	and
	\[
	\widehat{\mathcal R}_2(\lambda)
	\longrightarrow
	(\lambda-c_+)\widehat{\mathcal R}_1(\lambda),
	\qquad
	P_4(\lambda,\xi)
	\longrightarrow
	(\lambda-c_+)P_3(\lambda,\xi).
	\]
	Since \(P_3(c_+,\xi_+)=0\), the limiting numerator has the required
	double zero at the opening point. Away from the transition, the
	condition
	\(\partial_\lambda P_4(C_1,\xi)\neq0\) gives the local
	\(3/2\)-behavior in \eqref{soft-edge-expansion}.
	
	\paragraph{Genus-3 sector.}
	In the genus-3 sector, two conjugate pairs of moving branch points are
	present:
	\[
	\mathcal A_3
	=
	\mathcal A_1
	\cup
	\{
	B_1(\xi),B_2(\xi),
	\overline{B_1(\xi)},\overline{B_2(\xi)}
	\}.
	\]
	Thus the quotient curve is 
	\[
	\widehat{\mathcal R}_3(\lambda)^2
	=
	\prod_{\alpha\in\mathcal A_3}
	(\lambda-\alpha),
	\quad
	d\widehat\Phi_3(\xi,\lambda)
	=
	4\frac{P_5(\lambda,\xi)}
	{\widehat{\mathcal R}_3(\lambda)}\,d\lambda.
	\]
	The complete modulation system consists of the two infinity
	conditions, the three Boutroux conditions, and
	\begin{equation}\label{genus-three-soft-edge}
		P_5(B_1(\xi),\xi)
		=
		P_5(B_2(\xi),\xi)
		=
		0.
	\end{equation}
	Together with the conjugate equations, this gives nine real equations
	for the five real coefficients of \(P_5\) and the four real components
	of \(B_1,B_2\).
	
	In the generic case, the genus-3 branch points are opened from the
	non-real genus-1 critical points. At the transition value \(\xi_0\), one has
	\[
	B_1(\xi),B_2(\xi)
	\longrightarrow d(\xi_0),
	\qquad
	\overline{B_1(\xi)},\overline{B_2(\xi)}
	\longrightarrow\overline{d(\xi_0)}.
	\]
	The degeneration is compatible with the genus-1 differential in the
	form
	\[
	\widehat{\mathcal R}_3(\lambda)
	\longrightarrow
	(\lambda-d(\xi_0))
	(\lambda-\overline{d(\xi_0)})
	\widehat{\mathcal R}_1(\lambda),
	\quad
	P_5(\lambda,\xi)
	\longrightarrow
	(\lambda-d(\xi_0))
	(\lambda-\overline{d(\xi_0)})
	P_3(\lambda,\xi_0).
	\]
	Since \(d(\xi_0)\) and \(\overline{d(\xi_0)}\) are already zeros of
	\(P_3\), the limiting numerator has double zeros at the two opening
	points, as required.
	
	In the degenerate case, one conjugate pair is inherited from the
	genus-2 sector, after a relabeling of the moving endpoints, while the
	second pair is opened at the next degeneration of the genus-2
	critical graph.
	
	The genus-3 sector terminates at the right
	vacuum boundary
	\[
	\xi=\xi_{\mathrm{vac}}.
	\]
	For \(\xi>\xi_{\mathrm{vac}}\), the original jump matrices are
	exponentially close to the identity, and no finite-gap \(g\)-function
	is required.

	\subsection{Scalar normalization and model problems solved by Riemann theta function}
	
	After the transformation based on \(g\)-function is introduced, the exponential growth on the
	main bands has been removed. However, the mother body density factors
	\(\varrho(z)\) and \(\varrho^*(z)\) still occur in the off-diagonal
	entries. We now absorb these factors by a scalar normalization and
	reduce the leading problem to a finite-gap model problem with constant jump.
	
	We formulate the construction uniformly in the
	sectors of effective genus \(n\) for \(n=1,2,3\). Although the matrix RHP is
	written in the original \(z\)-plane, all Abelian data are defined on
	the quotient curve \(\widehat{\mathcal X}_n\).
	
	Let \(\Sigma_{\operatorname{ban}}^{n-}\) denote the bands on which the
	transformed jump has a lower-left entry, and let
	\(\Sigma_{\operatorname{ban}}^{n+}\) denote their Schwarz-conjugate
	bands, on which the transformed jump has an upper-right entry. Let
	\(\Sigma_{\operatorname{gap}}^{n,j}\), \(j=1,\ldots,n\), denote the gap
	contours. We introduce a scalar function \(F_n=F_n(\xi,z)\), analytic
	and nonzero away from the bands and gaps, satisfying
	\begin{equation}
		\left\{
		\begin{aligned}
			F_{n,+}F_{n,-}
			&=\varrho^{-1},
			&&z\in\Sigma_{\operatorname{ban}}^{n-},\\
			F_{n,+}F_{n,-}
			&=\varrho^*,
			&&z\in\Sigma_{\operatorname{ban}}^{n+},\\
			F_{n,+}F_{n,-}^{-1}
			&=e^{\kappa_{n,j}},
			&&z\in\Sigma_{\operatorname{gap}}^{n,j},
			\quad j=1,\ldots,n.
		\end{aligned}
		\right.
		\label{eq:Fn-jump-conditions}
	\end{equation}
	Here and below the \(\xi\)-dependence of the quantities in
	\eqref{eq:Fn-jump-conditions} is suppressed when no confusion can
	arise.
	
	The constants \(\kappa_{n,j}\) are chosen so that \(F_n\) is
	single-valued on the quotient surface. Equivalently, after fixing
	compatible branches of \(\log\varrho\) and \(\log\varrho^*\), we impose
	the full period conditions
	\begin{equation}
		\oint_{\widehat a_j}d\log F_n=0,
		\qquad
		j=1,\ldots,n.
		\label{eq:Fn-period}
	\end{equation}
	The Schwarz symmetry reduces \eqref{eq:Fn-period} to an
	\(n\)-dimensional real linear system for
	\(\kappa_{n,1},\ldots,\kappa_{n,n}\). At infinity,
	\begin{equation}
		F_n(\xi,z)
		=
		F_{\infty,n}(\xi)+\mathcal O(z^{-1}),
		\qquad
		z\to\infty,
		\label{eq:Fn-infinity}
	\end{equation}
	where \(F_{\infty,n}(\xi)\neq0\). After conjugation by
	\(F_n^{\sigma_3}\), the density factors disappear and the remaining jumps are constant.
	
	Let
	\[
	\widehat{\boldsymbol\nu}
	=
	(\widehat\nu_1,\ldots,\widehat\nu_n)^{\mathsf T}
	\]
	be the normalized holomorphic differentials on
	\(\widehat{\mathcal X}_n\):
	\[
	\oint_{\widehat a_j}\widehat\nu_k=\delta_{jk},
	\qquad j,k=1,\ldots,n.
	\]
	The period matrix is
	\begin{equation}\label{period-matrix-t}
		\widehat{\mathcal P}_{jk}
		=
		\oint_{\widehat b_j}\widehat\nu_k,
		\qquad
		j,k=1,\ldots,n,
	\end{equation}
	and the corresponding Riemann theta function is
	\[
	\Theta(\mathbf z;\widehat{\mathcal P})
	=
	\sum_{\mathbf m\in\mathbb Z^n}
	\exp\left\{
	2\pi i\mathbf m^{\mathsf T}\mathbf z
	+
	\pi i\mathbf m^{\mathsf T}
	\widehat{\mathcal P}\mathbf m
	\right\}.
	\]
	
	The Abel map, based at \(\infty^+\), is
	\[
	\widehat{\mathcal A}_n(P)
	=
	\int_{\infty^+}^{P}\widehat{\boldsymbol\nu}.
	\]
	For the two points \(P^\pm(z)\) lying over \(\lambda=z^2\), set
	\begin{equation}\label{limit-of-abel-map}
		\widehat{\mathcal A}_n^\pm(z)
		=
		\widehat{\mathcal A}_n(P^\pm(z)),
		\qquad
		\widehat{\mathcal A}_{n,\infty}
		=
		\lim_{z\to\infty}
		\widehat{\mathcal A}_n^-(z).
	\end{equation}
	The constant jumps on the gaps are encoded by
	\begin{equation}\label{eq:Delta-n-vector}
		\boldsymbol\Delta_n(\xi,t)
		=
		\frac{
			t\boldsymbol\Omega_n(\xi)
			-i\boldsymbol\kappa_n(\xi)
		}{2\pi},
	\end{equation}
	where
	\(\boldsymbol\kappa_n
	=(\kappa_{n,1},\ldots,\kappa_{n,n})^{\mathsf T}\).
	For each \(\xi\), choose a nonspecial divisor
	\[
	\mathcal D_n=P_{1,n}+\cdots+P_{n,n}
	\]
	of degree \(n\), depending smoothly on \(\xi\), and define
	\[
	\mathbf d_n
	=
	-\widehat{\mathcal A}_n(\mathcal D_n)-\mathbf K_n
	\quad
	\mod
	\bigl(
	\mathbb Z^n+\widehat{\mathcal P}\mathbb Z^n
	\bigr),
	\]
	where \(\mathbf K_n\) is the vector of Riemann constants. The divisor is
	chosen so that the theta quotients below have no nonremovable poles.
	Define
	\begin{equation}\label{eq:Qnpm}
		Q_n^\pm(z)
		=
		\frac{
			\Theta\bigl(
			\widehat{\mathcal A}_n^\pm(z)
			+\mathbf d_n-\boldsymbol\Delta_n;
			\widehat{\mathcal P}
			\bigr)
		}{
			\Theta\bigl(
			\widehat{\mathcal A}_n^\pm(z)
			+\mathbf d_n;
			\widehat{\mathcal P}
			\bigr)
		}.
	\end{equation}
	Let \(\alpha_n(z)\) be the fourth-root algebraic factor producing the
	sheet exchange across the bands, normalized by
	\[
	\alpha_n(z)=1+\mathcal O(z^{-1}),
	\qquad z\to\infty.
	\]
	Introduce first the unnormalized  matrix in terms of Riemann theta function
	\begin{equation}\label{eq:theta-model-matrix}
		\mathcal N_n(z)
		=
		\frac{
			\Theta(\mathbf d_n;\widehat{\mathcal P})
		}{
			2\Theta(
			\mathbf d_n-\boldsymbol\Delta_n;
			\widehat{\mathcal P})
		}
		\begin{pmatrix}
			Q_n^+(\alpha_n+\alpha_n^{-1})
			&
			-iQ_n^-(\alpha_n-\alpha_n^{-1})
			\\[1mm]
			iQ_n^+(\alpha_n-\alpha_n^{-1})
			&
			Q_n^-(\alpha_n+\alpha_n^{-1})
		\end{pmatrix}.
	\end{equation}
	The theta quotients reproduce the diagonal jumps on the gap, whereas the
	fourth-root factor produces the constant off-diagonal jumps on the
	bands.
	
	Since
	\[
	\widehat{\mathcal A}_n^+(z)\to0,
	\qquad
	\widehat{\mathcal A}_n^-(z)
	\to\widehat{\mathcal A}_{n,\infty},
	\qquad z\to\infty,
	\]
	one obtains
	\[
	\mathcal N_n(\infty)
	=
	\begin{pmatrix}
		1&0\\
		0&T_n(\xi,t)
	\end{pmatrix},
	\]
	where
	\begin{equation}\label{eq:Tn}
		T_n(\xi,t)
		=
		\frac{
			\Theta(\mathbf d_n;\widehat{\mathcal P})
			\Theta\bigl(
			\widehat{\mathcal A}_{n,\infty}
			+\mathbf d_n-\boldsymbol\Delta_n;
			\widehat{\mathcal P}
			\bigr)
		}{
			\Theta\bigl(
			\mathbf d_n-\boldsymbol\Delta_n;
			\widehat{\mathcal P}
			\bigr)
			\Theta\bigl(
			\widehat{\mathcal A}_{n,\infty}
			+\mathbf d_n;
			\widehat{\mathcal P}
			\bigr)
		}.
	\end{equation}
	Therefore the normalized outer model is
	\begin{equation}\label{eq:normalized-outer-model}
		m_n^{\mathrm{out}}(z)
		=
		\mathcal N_n(\infty)^{-1}\mathcal N_n(z),
	\end{equation}
	where satisfies
	\[
	m_n^{\mathrm{out}}(z)
	=
	\mathbb I+\mathcal O(z^{-1}),
	\qquad z\to\infty.
	\]
	
	Assume, in addition, that the nonspecial divisor
	$\mathcal D_n$ can be chosen compatibly with the Schwarz symmetry
	and the analyticity requirements of the model problem so that
	\[
	\widehat{\mathcal A}_n(\mathcal D_n)+\mathbf K_n
	\equiv 0
	\quad
	\mod\bigl(
	\mathbb Z^n+\widehat{\mathcal P}\mathbb Z^n
	\bigr).
	\]
	Then $\mathbf d_n\equiv0$, and the evenness of the Riemann theta
	function reduces \eqref{eq:Tn} to 
	\begin{equation}\label{eq:Tn-dzero}
		T_n(\xi,t)
		=
		\frac{
			\Theta(0;\widehat{\mathcal P})
			\Theta\bigl(
			\widehat{\mathcal A}_{n,\infty}
			-\boldsymbol\Delta_n;
			\widehat{\mathcal P}
			\bigr)
		}{
			\Theta\bigl(
			\boldsymbol\Delta_n;
			\widehat{\mathcal P}
			\bigr)
			\Theta\bigl(
			\widehat{\mathcal A}_{n,\infty};
			\widehat{\mathcal P}
			\bigr)
		}.
	\end{equation}
	All displayed formulae are understood away from the theta divisor;
	the apparent singularities at exceptional divisor values are treated
	by the standard divisor continuation.
	
	Suppose that
	\begin{equation}\label{eq:alpha-general-expansion}
		\alpha_n(z)
		=
		1+\frac{i\widehat\alpha_n(\xi)}{z}
		+\mathcal O(z^{-2}),
		\qquad z\to\infty.
	\end{equation}
	Tracing back the scalar and \(g\)-function conjugations, the leading-order term of asymptotic solution is therefore
	\begin{equation}\label{eq:Un-general}
		\mathcal U_n(x,t)
		=
		2i\widehat\alpha_n(\xi)
		T_n(\xi,t)
		\bigl(F_{\infty,n}(\xi)\bigr)^2
		e^{2it g_{\infty,n}(\xi)}.
	\end{equation}
	
	In the genus-1 and genus-\(1_s\) sectors, take
	\begin{equation}\label{eq:alpha1}
		\alpha_1(z)
		=
		\left[
		\frac{
			(z-E_1)(z-\overline E_2)
			(z+\overline E_1)(z+E_2)
		}{
			(z+E_1)(z+\overline E_2)
			(z-\overline E_1)(z-E_2)
		}
		\right]^{1/4}.
	\end{equation}
	Then it follows that
	\begin{equation}\label{eq:U1}
		\mathcal U_1(x,t)
		=
		2i
		\bigl(
		\operatorname{Im}E_2-\operatorname{Im}E_1
		\bigr)
		T_1(\xi,t)
		\bigl(F_{\infty,1}(\xi)\bigr)^2
		e^{2it g_{\infty,1}(\xi)}.
	\end{equation}
	The genus-\(1_s\) sector has the same algebraic factor but a different
	band--gap decomposition, and therefore different values of
	\(T_1\), \(F_{\infty,1}\), and \(g_{\infty,1}\).
	
	In the genus-2 sector, write
	\[
	C_1(\xi)=c_1(\xi)^2,
	\qquad
	\operatorname{Im}c_1(\xi)>0,
	\]
	and define
	\begin{equation}\label{eq:alpha2}
		\alpha_2(z)
		=
		\alpha_1(z)
		\left[
		\frac{
			(z-c_1)(z+\overline c_1)
		}{
			(z+c_1)(z-\overline c_1)
		}
		\right]^{1/4}.
	\end{equation}
	Then it follows that
	\begin{equation}\label{eq:U2}
		\mathcal U_2(x,t)
		=
		2i
		\bigl(
		\operatorname{Im}E_2-\operatorname{Im}E_1
		-\operatorname{Im}c_1
		\bigr)
		T_2(\xi,t)
		\bigl(F_{\infty,2}(\xi)\bigr)^2
		e^{2it g_{\infty,2}(\xi)}.
	\end{equation}
	
	In the genus-3 sector, let
	\[
	B_j(\xi)=b_j(\xi)^2,
	\qquad
	\operatorname{Im}b_j(\xi)>0,
	\qquad j=1,2,
	\]
	and set
	\begin{equation}\label{eq:alpha3}
		\alpha_3(z)
		=
		\alpha_1(z)
		\left[
		\frac{
			(z+b_1)(z-\overline b_1)
			(z-b_2)(z+\overline b_2)
		}{
			(z-b_1)(z+\overline b_1)
			(z+b_2)(z-\overline b_2)
		}
		\right]^{1/4}.
	\end{equation}
	Consequently, we have
	\begin{equation}\label{eq:U3}
		\mathcal U_3(x,t)
		=
		2i
		\bigl(
		\operatorname{Im}E_2-\operatorname{Im}E_1
		+\operatorname{Im}b_1-\operatorname{Im}b_2
		\bigr)
		T_3(\xi,t)
		\bigl(F_{\infty,3}(\xi)\bigr)^2
		e^{2it g_{\infty,3}(\xi)}.
	\end{equation}
	
	The functions
	\(\mathcal U_1,\mathcal U_2,\mathcal U_3\), together with the
	genus-\(1_s\) variant of \(\mathcal U_1\), provide the leading
	finite-gap terms in the corresponding long-time asymptitics of each sectors, which are listed in Theorem \ref{time degen}.

	\subsection{Local parametrices and error analysis}
	
	We now complete the nonlinear steepest descent analysis in the finite-genus
	sectors. Fix a compact set of values of $\xi$ contained in one of the open
	sectors. We assume throughout that the branch points remain mutually
	separated, the modulation Jacobian is nonzero, the strict sign inequalities
	hold, and the theta denominators occurring in $m_n^{\mathrm{out}}$ stay away
	from zero. All estimates below are uniform under these assumptions.
	
	After the $g$-function transformation, the scalar normalization by
	$F_n^{\sigma_3}$, and the opening of lenses, the leading jumps on the bands
	are solved by $m_n^{\mathrm{out}}$. On every remaining contour separated from
	the band endpoints, the jump matrices contain factors
	$e^{\pm2it\widehat\Phi_n}$ whose moduli are exponentially small by the
	admissible signature table. Local parametrices are nevertheless required at
	all band endpoints. The fixed endpoints are hard edges and are described by
	Bessel parametrices, whereas the moving endpoints introduced by the
	modulation equations are soft edges and are described by Airy parametrices.
	
	\subsubsection{Fixed endpoints and the genus-1 sectors}
	
	Let
	\[
	\mathcal P_{\mathrm{fix}}
	=
	\{\pm E_1,\pm E_2,
	\pm\overline E_1,\pm\overline E_2\}.
	\]
	These points are present in every finite-genus sector. Let
	$p\in\mathcal P_{\mathrm{fix}}$, and choose the local sheet adjacent to the
	corresponding band. Since $p^2$ is a fixed branch point and, away from the
	transition rays,
	\[
	P_{n+2}(p^2,\xi)\neq0,
	\]
	we have
	\begin{equation}\label{eq:local-phase-fixed-endpoint}
		\widehat\Phi_n(\xi,z^2)
		-
		\widehat\Phi_n(\xi,p^2)
		=
		\beta_{n,p}(\xi)(z-p)^{1/2}
		\bigl(1+\mathcal O(z-p)\bigr),
		\qquad z\to p,
	\end{equation}
	where $\beta_{n,p}(\xi)\neq0$. The square root and the limiting value of the
	phase are taken on the chosen local sheet.
	
	Choose $\varepsilon_{n,p}\in\{\pm1\}$ and the local branches so that
	\begin{equation}\label{eq:local-hard-phase}
		\phi_{n,p}(z)
		:=
		i\varepsilon_{n,p}
		\left(
		\widehat\Phi_n(\xi,z^2)
		-
		\widehat\Phi_n(\xi,p^2)
		\right)
	\end{equation}
	produces the decaying exponentials on the two lens boundaries. The conformal
	coordinate is
	\begin{equation}\label{eq:bessel-coordinate}
		\zeta^{\mathrm B}_{n,p}(z)
		=
		\frac{t^2}{4}\phi_{n,p}(z)^2,
		\qquad
		2\bigl(\zeta^{\mathrm B}_{n,p}(z)\bigr)^{1/2}
		=
		t\phi_{n,p}(z).
	\end{equation}
	By \eqref{eq:local-phase-fixed-endpoint},
	$\zeta^{\mathrm B}_{n,p}$ is conformal in a sufficiently small fixed disk
	$D_p$ centered at $p$.
	
	The endpoint condition on the density is
	\[
	\varrho(z)
	=
	(z-p)^{1/2}\varrho_p(z),
	\qquad
	\varrho_p(p)\neq0,
	\]
	with the Schwarz-conjugate analogue on the conjugate bands. Consequently, the
	local RHP is the standard hard-edge Bessel problem of
	order $1/2$, up to a constant conjugation determined by the orientation of the
	band and the triangularity of the jump. Denote its solution by
	$\Psi_{\mathrm{Bes}}$. We use the normalization
	\begin{equation}\label{eq:bessel-asymptotics}
		\Psi_{\mathrm{Bes}}(\zeta)
		=
		\zeta^{-\frac14\sigma_3}M_{\mathrm B}
		\left(
		\mathbb I+\mathcal O(\zeta^{-1/2})
		\right)
		e^{-2\zeta^{1/2}\sigma_3},
		\qquad \zeta\to\infty,
	\end{equation}
	where $M_{\mathrm B}$ is a fixed unimodular matrix. Its behavior at
	$\zeta=0$ has the quarter-root form required by the scalar factor $F_n$ and
	the outer model.
	
	The local parametrix can therefore be written as
	\begin{equation}\label{eq:bessel-local-parametrix}
		P_p^{\mathrm B}(z)
		=
		E_p^{\mathrm B}(z)
		\Psi_{\mathrm{Bes}}
		\bigl(\zeta^{\mathrm B}_{n,p}(z)\bigr)
		e^{t\phi_{n,p}(z)\sigma_3}
		\mathcal C_{n,p}^{\mathrm B}(z),
		\qquad z\in D_p.
	\end{equation}
	Here $\mathcal C_{n,p}^{\mathrm B}(z)$ is analytic and invertible in
	$D_p$ and contains the nonzero analytic factors left in the local triangular
	jumps, together with the constant band phase and orientation matrices. The
	prefactor $E_p^{\mathrm B}$ is analytic and invertible in $D_p$ and is chosen
	so that the quarter-root singularity in \eqref{eq:bessel-asymptotics}
	cancels that of $m_n^{\mathrm{out}}$. Since
	$|\zeta^{\mathrm B}_{n,p}|\asymp t^2$ on $\partial D_p$, we obtain
	\begin{equation}\label{eq:bessel-matching}
		P_p^{\mathrm B}(z)
		=
		m_n^{\mathrm{out}}(z)
		\left(
		\mathbb I+\mathcal O(t^{-1})
		\right),
		\qquad z\in\partial D_p.
	\end{equation}

	\subsubsection{Moving endpoints and the Airy model}
	
	For $n=2,3$, let $\mathcal P_{\mathrm{mov}}^{(n)}$ denote the moving
	endpoints in the original $z$-plane:
	\[
	\mathcal P_{\mathrm{mov}}^{(2)}
	=
	\{\pm c_1,\pm\overline c_1\},
	\qquad
	\mathcal P_{\mathrm{mov}}^{(3)}
	=
	\{\pm b_1,\pm b_2,
	\pm\overline b_1,\pm\overline b_2\}.
	\]
	Let $p\in\mathcal P_{\mathrm{mov}}^{(n)}$. The soft-edge conditions
	\eqref{soft-edge-conditions} and the nondegeneracy condition
	$\partial_\lambda P_{n+2}(p^2,\xi)\neq0$ imply
	\begin{equation}\label{eq:local-phase-moving-endpoint}
		\widehat\Phi_n(\xi,z^2)
		-
		\widehat\Phi_n(\xi,p^2)
		=
		\gamma_{n,p}(\xi)(z-p)^{3/2}
		\bigl(1+\mathcal O(z-p)\bigr),
		\qquad z\to p,
	\end{equation}
	where $\gamma_{n,p}(\xi)\neq0$. Define
	\begin{equation}\label{eq:local-airy-phase}
		\phi_{n,p}(z)
		=
		i\varepsilon_{n,p}
		\left(
		\widehat\Phi_n(\xi,z^2)
		-
		\widehat\Phi_n(\xi,p^2)
		\right),
		\qquad
		\varepsilon_{n,p}\in\{\pm1\},
	\end{equation}
	with the sign and branch selected according to the local signature chart.
	The coordinate in the Airy model is
	\begin{equation}\label{eq:zeta-a}
		\zeta^{\mathrm A}_{n,p}(z)
		=
		\left[
		\frac{3t}{2}\phi_{n,p}(z)
		\right]^{2/3},
		\qquad
		\frac{2}{3}
		\bigl(\zeta^{\mathrm A}_{n,p}(z)\bigr)^{3/2}
		=
		t\phi_{n,p}(z).
	\end{equation}
	It is conformal in a fixed disk $D_p$ and satisfies
	$|\zeta^{\mathrm A}_{n,p}|\asymp t^{2/3}$ on $\partial D_p$.
	
	We use the following constant jump of the Airy model. Let
	$\omega=e^{2\pi i/3}$ and take the rays
	\[
	\Gamma_0:\arg\zeta=0,
	\quad
	\Gamma_+:\arg\zeta=\frac{2\pi}{3},
	\quad
	\Gamma_\pi:\arg\zeta=\pi,
	\quad
	\Gamma_-:\arg\zeta=-\frac{2\pi}{3}
	\]
	to be oriented from the origin to infinity.
	
	\begin{problem}\label{Airy}
		Find a $2\times2$ matrix $\Psi_{\mathrm{Ai}}(\zeta)$ with the following
		properties.
		\begin{enumerate}
			\item $\Psi_{\mathrm{Ai}}(\zeta)$ is analytic for
			$\zeta\in \mathbb C\setminus
			(\Gamma_0\cup\Gamma_+\cup\Gamma_\pi\cup\Gamma_-)$.
			
			\item Its boundary values satisfy the constant jump relations
			\[
			\Psi_{\mathrm{Ai},+}(\zeta)
			=
			\Psi_{\mathrm{Ai},-}(\zeta)
			\begin{cases}
				\begin{pmatrix}1&1\\0&1\end{pmatrix},
				&\zeta\in\Gamma_0,\\[4mm]
				\begin{pmatrix}1&0\\-1&1\end{pmatrix},
				&\zeta\in\Gamma_+\cup\Gamma_-,\\[4mm]
				\begin{pmatrix}0&-1\\1&0\end{pmatrix},
				&\zeta\in\Gamma_\pi.
			\end{cases}
			\]
			
			\item As $\zeta\to\infty$, with the powers taken on the principal
			branch and continued sectorially,
			\begin{equation}\label{eq:Airy-asymptotics}
				\Psi_{\mathrm{Ai}}(\zeta)
				=
				\zeta^{-\frac14\sigma_3}\frac{e^{-\pi i/4}}{\sqrt2}
				\begin{pmatrix}
					1&i\\-1&i
				\end{pmatrix}
				\left[
				\mathbb I+
				\mathcal O(\zeta^{-3/2})
				\right]
				e^{-\frac23\zeta^{3/2}\sigma_3}.
			\end{equation}
			
			\item $\Psi_{\mathrm{Ai}}(\zeta)$ is bounded as $\zeta\to0$ away from the
			rays.
		\end{enumerate}
	\end{problem}
	
	For completeness, this model has the following explicit solution. Define
	\[
	y_0(\zeta)=\operatorname{Ai}(\zeta),
	\qquad
	y_1(\zeta)=\omega\operatorname{Ai}(\omega\zeta),
	\qquad
	y_2(\zeta)=\omega^2\operatorname{Ai}(\omega^2\zeta),
	\]
	and
	$\mathbf y_j=(y_j,y_j')^{\mathsf T}~(j=0,1,2)$. Since
	$\mathbf y_0+\mathbf y_1+\mathbf y_2=0$, the matrix
	\begin{equation}\label{eq:explicit-Airy-model}
		\Psi_{\mathrm{Ai}}(\zeta)
		=
		\sqrt{2\pi}\,e^{-\pi i/4}
		\begin{cases}
			(\mathbf y_0,-\mathbf y_2),
			&0<\arg\zeta<\frac{2\pi}{3},\\[1mm]
			(-\mathbf y_1,-\mathbf y_2),
			&\frac{2\pi}{3}<\arg\zeta<\pi,\\[1mm]
			(-\mathbf y_2,\mathbf y_1),
			&-\pi<\arg\zeta<-\frac{2\pi}{3},\\[1mm]
			(\mathbf y_0,\mathbf y_1),
			&-\frac{2\pi}{3}<\arg\zeta<0,
		\end{cases}
	\end{equation}
	solves RHP~\ref{Airy}; here $(\mathbf a,\mathbf b)$ denotes the
	matrix with columns $\mathbf a$ and $\mathbf b$. In particular,
	$\det\Psi_{\mathrm{Ai}}(\zeta)\equiv1$.
	
	After a $z$-independent constant conjugation which accounts for the local
	orientation and the constant band phase, the local Airy parametrix has the
	form
	\begin{equation}\label{eq:Airy-parametrix-a}
		P_p^{\mathrm A}(z)
		=
		E_p^{\mathrm A}(z)
		\Psi_{\mathrm{Ai}}
		\bigl(\zeta^{\mathrm A}_{n,p}(z)\bigr)
		e^{t\phi_{n,p}(z)\sigma_3}
		\mathcal C_{n,p}^{\mathrm A}(z),
		\qquad z\in D_p.
	\end{equation}
	Here $\mathcal C_{n,p}^{\mathrm A}(z)$ is analytic and invertible in
	$D_p$ and incorporates the nonzero analytic jump coefficients, the local
	orientation matrices, and the constant value of the phase. The prefactor
	$E_p^{\mathrm A}$ is analytic and invertible in $D_p$ and is chosen so that
	the factor
	$(\zeta^{\mathrm A}_{n,p})^{-\sigma_3/4}$ in
	\eqref{eq:Airy-asymptotics} cancels the fourth-root singularity of
	$m_n^{\mathrm{out}}$. By \eqref{eq:zeta-a}, the exponential in
	\eqref{eq:Airy-asymptotics} cancels the exponential in
	\eqref{eq:Airy-parametrix-a}. Hence
	\begin{equation}\label{eq:Airy-matching}
		P_p^{\mathrm A}(z)
		=
		m_n^{\mathrm{out}}(z)
		\left(
		\mathbb I+
		\mathcal O(t^{-1})
		\right),
		\qquad z\in\partial D_p.
	\end{equation}

	\subsubsection{Global error problem}
	Let
	$\sharp\in\{1,1_s,2,3\}$, and write $\mathcal T_\sharp$ and
	$m_\sharp^{\mathrm{out}}$ for the corresponding final transformation and
	outer model. 
	For $n=2,3$, let
	\[
	\mathfrak D^{\mathrm B}
	=
	\bigcup_{p\in\mathcal P_{\mathrm{fix}}}D_p,
	\quad 
	\mathfrak D_n^{\mathrm A}
	=
	\bigcup_{p\in\mathcal P_{\mathrm{mov}}^{(n)}}D_p,
	\]
	where all disks are chosen mutually disjoint. The global approximation is
	\begin{equation}\label{eq:global-app-model}
		m_\sharp^{\mathrm{app}}(z)
		=
		\begin{cases}
			P_p^{\mathrm B}(z),
			&z\in D_p,\quad p\in\mathcal P_{\mathrm{fix}},\\[1mm]
			P_p^{\mathrm A}(z),
			&z\in D_p,\quad p\in\mathcal P_{\mathrm{mov}}^{(n)},\\[1mm]
			m_\sharp^{\mathrm{out}}(z),
			&z\notin
			\bigl(\mathfrak D^{\mathrm B}
			\cup\mathfrak D_\sharp^{\mathrm A}\bigr).
		\end{cases}
	\end{equation}
	Define
	\begin{equation}\label{eq:error-matrix-En}
		\mathcal E_\sharp(z)
		=
		\mathcal T_\sharp(z)
		\bigl(m_\sharp^{\mathrm{app}}(z)\bigr)^{-1}.
	\end{equation}
	On the lens contours outside the disks, we have
	\begin{equation}\label{eq:error-jump-lens}
		J_{\mathcal E_\sharp}(z)
		=
		\mathbb I+
		\mathcal O(e^{-ct}),
		\qquad c>0.
	\end{equation}
	On every Bessel or Airy circle,
	\[
	J_{\mathcal E_\sharp}(z)
	=
	\mathbb I+
	\mathcal O(t^{-1}).
	\]
	Consequently,
	\begin{equation}\label{eq:error-jump-norm}
		\left\|
		J_{\mathcal E_\sharp}-\mathbb I
		\right\|_{L^2\cap L^\infty}
		=
		\mathcal O(t^{-1}).
	\end{equation}
	The theory for small-norm RHP yields $\mathcal E_\sharp(z)
	=
	\mathbb I+
	\mathcal O\!\left(
	\frac{t^{-1}}{1+|z|}
	\right)$.
	In particular, the coefficient of $z^{-1}$ in the expansion of
	$\mathcal E_\sharp$ is $\mathcal O(t^{-1})$. Combining this estimate with
	\eqref{eq:Un-general}, we obtain, uniformly for $\xi$ in compact subsets of
	each open finite-genus sector,
	\begin{equation}\label{eq:finite-genus-final-asymptotics}
		u(x,t)
		=
		\mathcal U_\sharp(x,t)+\mathcal O(t^{-1}),
		\qquad
		\sharp\in\{1,1_s,2,3\},
	\end{equation}
	where $\mathcal U_{1_s}$ denotes the genus-1 expression with the
	band--gap data of the special signature chart.

	\section{Kinetic description of the DNLS soliton gas}
	\label{sec:dnls-kinetic}
	
	This section gives the kinetic equation of the DNLS soliton gas in each sector of the long-time asymptotics, which further proves that the soliton gas constructed in this work are regular soliton gas \cite{Grava-R}. The derivation is local in an open asymptotic sector. Thus we fix such a sector and denote by \(n\in\{1,2,3\}\) the corresponding effective quotient genus. In the rapidly decaying sector, the spectral density is identically zero, and the kinetic description is trivial.
	\par
	The natural spectral variable for the kinetic theory of the DNLS soliton gas is
	$
	\lambda=z^2.
	$
	Recall that
	$
	A_j=E_j^2$, $j=1,2$,
	and that $\Lambda$ is the image of the mother body segment $\mathscr I=[E_1,E_2]$ under $z\mapsto z^2$. The conjugate arc is denoted by \(\Lambda^*\). In a genus-\(n\) modulated sector we write \(\Lambda_n(\xi)\) for the union of the quotient bands of the corresponding \(g\)-function model. In the genus-1 sectors, \(\Lambda_n(\xi)\) reduces to the original quotient band \(\Lambda\); in the higher genus sectors it also contains the moving quotient bands generated by the additional branch points. The orientation of every component is inherited from the corresponding band in the \(z\)-plane.

	Let
	\[
	d\widehat\Phi_n(\xi,\lambda)
	=
	\widehat\Omega_n(\xi,\lambda)\,d\lambda
	\]
	be the quotient differential used in the construction of the genus-\(n\) \(g\)-function in the preceding section, where $d\widehat\Phi_n(\xi,\lambda)$ is given in \eqref{Abelian-differential-n}. On the upper support \(\Lambda_n\), define the cumulative spectral distribution by
	\begin{equation}\label{cumulative-spectral-distribution}
		\mathcal N_n(x,t,\lambda)
		:=
		\frac{t}{\pi}\widehat\Omega_{n,+}(\xi,\lambda),
		\qquad
		\xi=\frac{x}{t},
		\qquad
		\lambda\in\Lambda_n.
	\end{equation}
	Here \(+\) denotes the boundary value from the fixed side of the oriented arc \(\Lambda_n\).
	
	In the genus-1 sector, this definition gives the explicit formula
	\[
	\mathcal N_1(x,t,\lambda)
	=
	\frac{4t}{\pi}
	\frac{P_3(\lambda,\xi)}
	{\widehat {\mathcal{R}}_{1,+}(\lambda)},
	\qquad
	\lambda\in\Lambda_n,
	\]
	where
	$
	\widehat {\mathcal{R}}_1(\lambda)^2
	=
	(\lambda-A_1)(\lambda-A_2)
	(\lambda-\overline{A}_1)(\lambda-\overline{A}_2),
	$
	and \(P_3(\lambda,\xi)\) is the cubic polynomial appearing in the genus-1 quotient differential. In higher genus sectors, the same formula is replaced by
	\[
	\mathcal N_n(x,t,\lambda)
	=
	\frac{t}{\pi}
	\frac{d\widehat\Phi_n}{d\lambda}(\xi,\lambda)_+,
	\qquad n=2,3.
	\]
	It is characterized by the logarithmic potential relation
	\[
	\widehat\phi_n(x,t,\lambda_0)
	:=
	i\int_{\Lambda_n(\xi)}
	\log\left(
	\frac{\lambda_0-\overline{\mu}}
	{\lambda_0-\mu}
	\right)
	\mathcal N_n(x,t,\mu)\,d\mu,
	\qquad
	\lambda_0\in\mathbb C\setminus\bigl(\Lambda_n(\xi)\cup\Lambda_n^*(\xi)\bigr),
	\]
	where the branch of the logarithm is chosen so that the logarithm tends to zero as \(\lambda_0\to\infty\). Equivalently, \(\mathcal N_n(x,t,\lambda)\) is the density whose Cauchy logarithmic potential reproduces the quotient \(g\)-function. In particular,
	\[
	\widehat\phi_{n,+}(x,t,\lambda)+\widehat\phi_{n,-}(x,t,\lambda)
	=
	-2x\lambda-4t\lambda^2,
	\qquad
	\lambda\in\Lambda_n(\xi).
	\]
	This is the band condition inherited from the quotient \(g\)-function construction.
	
	The dressed phase of a point with quotient spectral parameter \(\lambda_0\) is then
	\begin{equation}
		\widehat\Psi_n(x,t,\lambda_0)
		:=
		\widehat\phi_n(x,t,\lambda_0)
		+
		x\lambda_0+2t\lambda_0^2.
	\end{equation}
	Thus \(\widehat\Psi_n\) is analytic for \(\lambda_0\notin\Lambda_n(\xi)\cup\Lambda_n(\xi)^*\), satisfies the normalization
	\[
	\widehat\Psi_n(x,t,\lambda_0)
	=
	x\lambda_0+2t\lambda_0^2+\mathcal O(\lambda_0^{-1}),
	\qquad
	\lambda_0\to\infty,
	\]
	and has the band relation
	\[
	\widehat\Psi_{n,+}(x,t,\lambda)
	+
	\widehat\Psi_{n,-}(x,t,\lambda)
	=
	0,
	\qquad
	\lambda\in\Lambda_n(\xi).
	\]
	
	The origin of the logarithmic term is the fourfold spectral symmetry of the DNLS equation. For a finite reflectionless configuration, a test soliton with spectral point \(z_0\) interacts not only with a pole \(s\), but also with its symmetric partners \(-s\), \(\overline{s}\), and \(-\overline{s}\). The corresponding product of elementary phase-shift factors has the form
	\[
	\frac{(z_0-\overline{s})(z_0+\overline{s})}
	{(z_0-s)(z_0+s)}
	=
	\frac{z_0^2-\overline{s}^{\,2}}
	{z_0^2-s^2}
	=
	\frac{\lambda_0-\overline{\mu}}
	{\lambda_0-\mu},
	\qquad
	\lambda_0=z_0^2,\quad \mu=s^2.
	\]
	Passing from the discrete product to the continuum limit gives precisely the logarithmic potential above. This is the specific mechanism behind the kinetic kernel.
	
	Define the local kinetic density by
	\begin{equation}\label{kinetic-density}
		\mathfrak f_n(x,t,\lambda)
		:=
		\partial_x\mathcal N_n(x,t,\lambda).
	\end{equation}
	All \(x\)- and \(t\)-derivatives are total derivatives in the sector. The effective velocity of the spectral component \(\lambda\in\Lambda_n(\xi)\) is defined by
	\begin{equation}\label{eff velocity}
		v_{{\rm eff},n}(x,t,\lambda)
		:=
		-\frac{\partial_t\mathcal N_n(x,t,\lambda)}
		{\partial_x\mathcal N_n(x,t,\lambda)}
		=
		-\frac{\partial_t\mathcal N_n(x,t,\lambda)}
		{\mathfrak f_n(x,t,\lambda)},
	\end{equation}
	whenever \(\mathfrak f_n\ne0\), and by continuous extension otherwise. Equivalently, we have
	\[
	\partial_t\mathcal N_n(x,t,\lambda)
	+
	v_{{\rm eff},n}(x,t,\lambda)
	\partial_x\mathcal N_n(x,t,\lambda)
	=
	0.
	\]
	Differentiating this identity with respect to \(x\) yields the local conservation law
	\[
	\partial_t\mathfrak f_n(x,t,\lambda)
	+
	\partial_x\Bigl(
	v_{{\rm eff},n}(x,t,\lambda)\mathfrak f_n(x,t,\lambda)
	\Bigr)
	=
	0,
	\qquad
	\lambda\in\Lambda_n(\xi).
	\]
	
	We next consider a distinguished test soliton with spectral parameter
	\[
	\lambda_0=z_0^2,\qquad \operatorname{Im}\lambda_0\ne0,
	\]
	lying off the active spectral support. Its center is determined by a level curve
	\[
	\operatorname{Im}\widehat\Psi_n(x,t,\lambda_0)=\mathrm{constant}.
	\]
	Therefore, if \(x=x(t)\) is the trajectory of the test soliton, its velocity is
	\[
	v_{\rm test}(x,t;\lambda_0)
	:=
	x'(t)
	=
	-
	\frac{\partial_t\operatorname{Im}\widehat\Psi_n(x,t,\lambda_0)}
	{\partial_x\operatorname{Im}\widehat\Psi_n(x,t,\lambda_0)}.
	\]
	Introduce the real logarithmic kernel
	\[
	\mathcal K(\lambda,\mu)
	:=
	\log\left|
	\frac{\lambda-\overline{\mu}}
	{\lambda-\mu}
	\right|.
	\]
	Taking the imaginary part of the definition of \(\widehat\Psi_n\) gives
	\begin{equation}
		\operatorname{Im}\widehat\Psi_n(x,t,\lambda_0)
		=
		\int_{\Lambda_n(\xi)}
		\mathcal K(\lambda_0,\mu)
		\mathcal N_n(x,t,\mu)\,d\mu
		+
		x\operatorname{Im}\lambda_0
		+
		2t\operatorname{Im}(\lambda_0^2).
	\end{equation}
	
	Along the trajectory of the test soliton,
	$
	\frac{d}{dt}
	\operatorname{Im}\widehat\Psi_n(x(t),t,\lambda_0)=0.
	$
	Hence
	\[
	0=
	\int_{\Lambda_n(\xi)}
	\mathcal K(\lambda_0,\mu)
	\Bigl[
	\partial_t\mathcal N_n(x,t,\mu)
	+
	v_{\rm test}(x,t;\lambda_0)
	\partial_x\mathcal N_n(x,t,\mu)
	\Bigr]\,d\mu
	+
	v_{\rm test}(x,t;\lambda_0)\operatorname{Im}\lambda_0
	+
	2\operatorname{Im}(\lambda_0^2).
	\]
	Using \eqref{kinetic-density} and \eqref{eff velocity},
	we obtain
	\begin{equation}\label{test-velocity}
		v_{\rm test}(x,t;\lambda_0)
		=
		v_0(\lambda_0)
		+
		\frac{1}{\operatorname{Im}\lambda_0}
		\int_{\Lambda_n(\xi)}
		\mathcal K(\lambda_0,\mu)
		\Bigl[
		v_{{\rm eff},n}(x,t,\mu)
		-
		v_{\rm test}(x,t;\lambda_0)
		\Bigr]
		\mathfrak f_n(x,t,\mu)\,d\mu,
	\end{equation}
	where the free velocity is
	\[
	v_0(\lambda)
	:=
	-2\frac{\operatorname{Im}(\lambda^2)}{\operatorname{Im}\lambda}
	=
	-4\operatorname{Re}\lambda.
	\]
	Thus the bare velocity is precisely the one-soliton velocity of the DNLS equation \eqref{dnls3} which is in agreement with \eqref{eq:free-soliton-velocity-a}.
	
	Let \(\lambda\in\Lambda_n(\xi)\). From the band condition for \(\widehat\phi_n\), we have
	\[
	2\int_{\Lambda_n(\xi)}
	\mathcal K(\lambda,\mu)
	\mathcal N_n(x,t,\mu)\,d\mu
	=
	-2x\operatorname{Im}\lambda
	-
	4t\operatorname{Im}(\lambda^2).
	\]
	Differentiating this identity with respect to \(t\) and \(x\) respectively and substituting \eqref{kinetic-density} and \eqref{eff velocity},
	it is obtained that
	\[
	\int_{\Lambda_n(\xi)}
	\mathcal K(\lambda,\mu)
	v_{{\rm eff},n}(x,t,\mu)
	\mathfrak f_n(x,t,\mu)\,d\mu
	=
	2\operatorname{Im}(\lambda^2),
	\quad
	\int_{\Lambda_n(\xi)}
	\mathcal K(\lambda,\mu)
	\mathfrak f_n(x,t,\mu)\,d\mu
	=
	-\operatorname{Im}\lambda.
	\]
	Combining the two identities gives
	\[
	\int_{\Lambda_n(\xi)}
	\mathcal K(\lambda,\mu)
	\Bigl[
	v_{{\rm eff},n}(x,t,\mu)
	-
	v_{{\rm eff},n}(x,t,\lambda)
	\Bigr]
	\mathfrak f_n(x,t,\mu)\,d\mu
	=
	2\operatorname{Im}(\lambda^2)
	+
	v_{{\rm eff},n}(x,t,\lambda)\operatorname{Im}\lambda.
	\]
	Solving this relation for \(v_{{\rm eff},n}(x,t,\lambda)\), the kinetic equation of the DNLS soliton gas is derived 
	\begin{equation}\label{efftive-velocity}
		v_{{\rm eff},n}(x,t,\lambda)
		=
		-4\operatorname{Re}\lambda
		+
		\frac{1}{\operatorname{Im}\lambda}
		\int_{\Lambda_n(\xi)}
		\log\left|
		\frac{\lambda-\overline{\mu}}
		{\lambda-\mu}
		\right|
		\Bigl[
		v_{{\rm eff},n}(x,t,\mu)
		-
		v_{{\rm eff},n}(x,t,\lambda)
		\Bigr]
		\mathfrak f_n(x,t,\mu)\,d\mu,
		\qquad
		\lambda\in\Lambda_n(\xi).
	\end{equation}

	\section{Fredholm determinants and the continuum \texorpdfstring{$\tau$}{tau}-function}
	\label{sec:tau-function}
	
	The preceding sections describe the continuum DNLS soliton gas through a compactly supported $\bar\partial$-problem and, in the elliptic setting, through the Riemann--Hilbert reduction of mother body. The purpose of this section is complementary: an operator-theoretic characterization of the same continuum object is given.
	\par
	There are three points to be established. First, the continuum $\bar\partial$-problem can be rewritten as a Fredholm equation of block Hilbert--Schmidt type, so that its solvability is governed by the non-vanishing of a Fredholm determinant. Second, this determinant is obtained as the continuum limit of the finite dimensional determinants associated with the pure $N$-soliton residue systems. Third, the squared modulus of the DNLS equation \eqref{dnls3} is recovered from the logarithmic derivative of the resulting determinant ratio.
	\par
	The spectral symmetry associated with the DNLS equation \eqref{dnls3} plays a structural role in this construction. In the original $z$-plane, the discrete spectrum occurs in quadruples
	\[
	z,\quad -z,\quad \overline z,\quad -\overline z.
	\]
	The quotient map $\lambda=z^2$ identifies the two points $z$ and $-z$. Their contribution to the Cauchy kernel, however, is not lost; it is retained in the characteristic numerator $2z_+(\lambda)$ of the quotient operator introduced below.
	
	Let
	$
	\pi:\mathbb C_z\to\mathbb C_\lambda,
	\ 
	\pi(z)=z^2,
	$
	be the quotient map. We assume that the condensation domain $\mathcal D_+$ is contained in the first quadrant and stays a positive distance away from the coordinate axes. Then the restriction of $\pi$ to $\mathcal D_+$ is one-to-one, and its image is contained in the upper half-plane. Set
	\[
	\mathcal D:=\pi(\mathcal D_+)
	=
	\{\lambda\in\mathbb C_\lambda:\lambda=z^2,\ z\in\mathcal D_+\},
	\qquad
	\mathcal D^*:=\overline{\mathcal D}.
	\]
	Thus $\mathcal D\Subset\mathbb C^+$ and $\mathcal D^*\Subset\mathbb C^-$; in particular,
	$
	\operatorname{dist}(\mathcal D,\mathcal D^*)>0.
	$
	We denote by $z_+(\lambda)$ the inverse branch of the square root on $\mathcal D$ characterized by
	\[
	z_+(\lambda)^2=\lambda,
	\qquad
	z_+(\lambda)\in\mathcal D_+,
	\qquad
	\lambda\in\mathcal D.
	\]
	Then $z_+(\lambda)$ is the representative of the quotient class ${z,-z}$ lying on the selected sheet $\mathcal D_+$.
	
	Let $dA_\lambda=d^2\lambda$ denote planar Lebesgue measure in $\lambda$-plane. We use the normalized area measure
	\[
	d\nu(\lambda):=\frac{1}{\pi}dA_\lambda.
	\]
	Let $f(z,\bar z)$ be the density appearing in the original compactly supported $\bar\partial$-problem on $\mathcal D_+$. We denote by $\widehat f$ its push-forward to the quotient domain $\mathcal D$. More precisely, for $\lambda\in\mathcal D$,
	\[
	\widehat f(\lambda)
	:=
	\frac{
		f\bigl(z_+(\lambda),\overline{z_+(\lambda)}\bigr)
	}{
		4|z_+(\lambda)|^2
	}.
	\]
	The denominator is the Jacobian factor for the change of variables $\lambda=z^2$, since
	$
	dA_\lambda
	=\left|\frac{d\lambda}{dz}\right|^2dA_z
	=
	4|z|^2dA_z.
	$
	Equivalently, on the selected sheet,
	$
	f(z,\bar z)dA_z
	=
	\widehat f(\lambda)dA_\lambda.
	$
	
	The Schwarz conjugate quotient density on $\mathcal D^*$ is defined by
	$
	\widehat f^{*}(\omega)
	:=
	\overline{\widehat f(\overline\omega)},
	\ 
	\omega\in\mathcal D^*.
	$
	Thus the density on the lower quotient domain is determined from the density on $\mathcal D$ by Schwarz reflection.

	\begin{assumption}
		\label{ass:tau-density}
		The quotient densities satisfy
		\[
		\widehat f\in L^\infty(\mathcal D)\cap L^2(\mathcal D),
		\qquad
		\widehat f^{\,*}\in L^\infty(\mathcal D^*)\cap L^2(\mathcal D^*).
		\]
		Moreover, the branches of \(\sqrt{\widehat f}\) on \(\mathcal D\) and \(\sqrt{\widehat f^{\,*}}\) on \(\mathcal D^*\) are fixed once and for all.
	\end{assumption}
	
	The choice of the square-root branches is immaterial for the determinant. Changing a branch conjugates the operator below by multiplication by a sign and leaves the Fredholm determinant unchanged.
	
	\begin{remark}
		The notation in this section separates three different densities. The function \(f\) is the density in the two-dimensional \(z\)-plane \(\bar\partial\)-problem. The function \(\widehat f\) is its quotient push-forward to the \(\lambda\)-plane. The mother body density \(\varrho\), used in the elliptic RHP, is a different object obtained after collapsing the \(\bar\partial\)-support to the mother body of the condensation domain.
	\end{remark}
	
	\subsection{Block Hilbert--Schmidt operators and regularized determinants}
	\label{subsec:block-HS-determinant}
	
	We first isolate the operator-theoretic mechanism \cite{Simon} needed later.
	Let \(H_+\) and \(H_-\) be complex Hilbert spaces. If
	$
	Q:H_-\to H_+
	$
	is Hilbert--Schmidt, denote by
	$
	Q^{\mathsf T}:H_+\to H_-
	$
	the transpose with respect to the bilinear pairing, not the Hermitian adjoint. Thus, if \(Q\) is represented by a kernel \(Q(\lambda,\omega)\), then \(Q^{\mathsf T}\) is represented by the transposed kernel with the two variables interchanged in the bilinear pairing.
	
	Define the block operator
	\[
	\mathcal K_Q
	:=
	\begin{pmatrix}
		0 & -Q\\
		Q^{\mathsf T} & 0
	\end{pmatrix}
	:
	H_+\oplus H_-\to H_+\oplus H_-.
	\]
	
	\begin{prop}
		\label{prop:block-det}
		Let \(Q:H_-\to H_+\) be Hilbert--Schmidt map. Then \(\mathcal K_Q\) is Hilbert--Schmidt, while
		\[
		Q^{\mathsf T}Q:H_-\to H_-,
		\qquad
		QQ^{\mathsf T}:H_+\to H_+
		\]
		are trace class \cite{Trace}. Moreover,
		\begin{equation}\label{Hilbert--Carleman determinant}
			\det_2(\mathbb I-\mathcal K_Q)
			=
			\det_{H_-}(\mathbb I+Q^{\mathsf T}Q)
			=
			\det_{H_+}(\mathbb I+QQ^{\mathsf T}),
		\end{equation}
		where \(\substack{\mathrm{\det}\\ 2}\) is the Hilbert--Carleman determinant and the determinants on the right-hand side are ordinary Fredholm determinants.
	\end{prop}
	
	\begin{proof}
		Let \(\mathfrak S_p\) denote the Schatten class; in particular,
		\(\mathfrak S_2\) is the Hilbert--Schmidt class and
		\(\mathfrak S_1\) is the trace class. Since \(Q\in\mathfrak S_2(H_-,H_+)\), its transpose \(Q^{\mathsf T}\in\mathfrak S_2(H_+,H_-)\). Hence
		$
		\mathcal K_Q=
		$
		is Hilbert--Schmidt on \(H_+\oplus H_-\). Moreover, products of two Hilbert--Schmidt operators are trace class, so
		\[
		Q^{\mathsf T}Q\in\mathfrak S_1(H_-),
		\qquad
		QQ^{\mathsf T}\in\mathfrak S_1(H_+).
		\]
		
		We first prove the determinant identity when \(Q\) has finite rank. In this case \(\mathcal K_Q\) is finite rank, hence trace class. Since \(\mathcal K_Q\) has zero diagonal blocks, we have
		$
		\operatorname{Tr}\mathcal K_Q=0.
		$
		Therefore, by the definition of the regularized determinant,
		\[
		\det_2(\mathbb I-\mathcal K_Q)
		=
		\det(\mathbb I-\mathcal K_Q)e^{\operatorname{Tr}\mathcal K_Q}
		=
		\det(\mathbb I-\mathcal K_Q).
		\]
		On the finite-dimensional range generated by \(Q\) and \(Q^{\mathsf T}\), we may use the Schur complement:
		\[
		\mathbb -\mathcal K_Q
		=
		\begin{pmatrix}
			\mathbb I & Q\\
			-Q^{\mathsf T} & \mathbb I
		\end{pmatrix}.
		\]
		Taking the Schur complement with respect to the upper-left block gives
		\[
		\det(\mathbb I-\mathcal K_Q)
		=
		\det(\mathbb I)\det\bigl(\mathbb I+Q^{\mathsf T}Q\bigr)
		=
		\det_{H_-}\bigl(\mathbb I+Q^{\mathsf T}Q\bigr).
		\]
		Taking instead the Schur complement with respect to the lower-right block, it yields that \eqref{Hilbert--Carleman determinant}.
		
		Now let \(Q\in\mathfrak S_2(H_-,H_+)\) be arbitrary. Choose finite-rank operators \(Q_m\) such that
		$
		\|Q_m-Q\|_{\mathfrak S_2}\to0.
		$
		Set
		\[
		\mathcal K_m:=
		\begin{pmatrix}
			0 & -Q_m\\
			Q_m^{\mathsf T} & 0
		\end{pmatrix}.
		\]
		Then
		$
		\|\mathcal K_m-\mathcal K_Q\|_{\mathfrak S_2}\to0.
		$
		Since the regularized determinant is continuous in the Hilbert--Schmidt norm,
		\[
		\det_2(\mathbb I-\mathcal K_m)\to \det_2(\mathbb I-\mathcal K_Q).
		\]
		On the other hand,
		\[
		Q_m^{\mathsf T}Q_m-Q^{\mathsf T}Q
		=
		(Q_m^{\mathsf T}-Q^{\mathsf T})Q_m
		+
		Q^{\mathsf T}(Q_m-Q).
		\]
		Then we get
		$
		\|Q_m^{\mathsf T}Q_m-Q^{\mathsf T}Q\|_{\mathfrak S_1}
		\to 0.
		$
		Hence, by continuity of the Fredholm determinant in the trace norm, it follows that
		\[
		\det_{H_-}(\mathbb I+Q_m^{\mathsf T}Q_m)
		\to
		\det_{H_-}(\mathbb I+Q^{\mathsf T}Q).
		\]
		Similarly,
		$
		\det_{H_+}(\mathbb I+Q_mQ_m^{\mathsf T})
		\to
		\det_{H_+}(\mathbb I+QQ^{\mathsf T}).
		$
		Passing to the limit in the finite-rank identity proves \eqref{Hilbert--Carleman determinant}.
	\end{proof}
	
	\begin{prop}\label{prop:block-invertibility}
		The block operator $I-\mathcal K_Q$ is invertible if and only if either, and hence both, of the Schur-complement operators
		\[
		\mathbb I+Q^{\mathsf T}Q:H_-\to H_-,
		\qquad
		\mathbb I+QQ^{\mathsf T}:H_+\to H_+
		\]
		are invertible. In particular, since $Q^{\mathsf T}Q$ and $QQ^{\mathsf T}$ are trace class,
		$
		\det_{H_-} (\mathbb I+Q^{\mathsf T}Q)\ne0
		$
		is equivalent to the invertibility of $\mathbb I-\mathcal K_Q$.
	\end{prop}
	
	\begin{proof}
		Use the block factorization
		\[
		\mathbb I-\mathcal K_Q=\begin{pmatrix}
			\mathbb I & Q\\
			-Q^{\mathsf T} & \mathbb I
		\end{pmatrix}
		=
		\begin{pmatrix}
			\mathbb I & 0\\
			-Q^{\mathsf T} & \mathbb I
		\end{pmatrix}
		\begin{pmatrix}
			\mathbb I & Q\\
			0 & \mathbb I+Q^{\mathsf T}Q
		\end{pmatrix}.
		\]
		The first factor is always invertible, with inverse
		$
		\begin{pmatrix}
			\mathbb I & 0\\
			Q^{\mathsf T} & \mathbb I
		\end{pmatrix}.
		$
		Therefore $\mathbb I-\mathcal K_Q$ is invertible if and only if $\mathbb I+Q^{\mathsf T}Q$ is invertible on $H_-$.

		Similarly,
		\[
		\mathbb I-\mathcal K_Q=\begin{pmatrix}
			\mathbb I & Q\\
			-Q^{\mathsf T} & \mathbb I
		\end{pmatrix}
		=
		\begin{pmatrix}
			\mathbb I & Q\\
			0 & \mathbb I
		\end{pmatrix}
		\begin{pmatrix}
			\mathbb I+QQ^{\mathsf T} & 0\\
			-Q^{\mathsf T} & \mathbb I
		\end{pmatrix},
		\]
		and the first factor is always invertible. Hence \(\mathbb I-\mathcal K_Q\) is invertible if and only if \(\mathbb I+QQ^{\mathsf T}\) is invertible on \(H_+\).
		
		Finally, since \(Q\) and \(Q^{\mathsf T}\) are Hilbert--Schmidt operators, the products \(Q^{\mathsf T}Q\) and \(QQ^{\mathsf T}\) are trace class. For a trace-class operator \(\cdot\), the Fredholm determinant satisfies
		\[
		\det(\mathbb I+\cdot)\ne0
		\quad\Longleftrightarrow\quad
		\mathbb I+\cdot\text{ is invertible}.
		\]
		
	\end{proof}

	\subsection{The quotient Fredholm operator of the DNLS equation}
	\label{subsec:dnls-quotient-operator}
	
	We now construct the operator associated with the continuum soliton gas of the DNLS equation. Set
	\[
	H_+:=L^2(\mathcal D,d\nu),
	\qquad
	H_-:=L^2(\mathcal D^*,d\nu).
	\]
	Define
	$
	Q_+(x,t):H_-\to H_+
	$
	by
	\begin{equation}\label{H-S positive}
		(Q_+(x,t)h)(\lambda)
		=
		\int_{\mathcal D^*}
		Q_+(x,t;\lambda,\omega)h(\omega)\,d\nu(\omega),
		\qquad
		\lambda\in\mathcal D,
	\end{equation}
	where
	\begin{equation}\label{kernel of H-S}
		Q_+(x,t;\lambda,\omega)
		=
		\frac{
			2i\,z_+(\lambda)
			\sqrt{\widehat f(\lambda)}
			\sqrt{\widehat f^{\,*}(\omega)}
			e^{i\widehat\theta(x,t,\lambda)-i\widehat\theta(x,t,\omega)}
		}{
			\lambda-\omega
		},
	\end{equation}
	whose transpose
	$
	Q_+^{\mathsf T}(x,t):H_+\to H_-
	$
	is defined with respect to the bilinear pairing
	$
	\langle g,h\rangle
	:=
	\int g(\lambda)h(\lambda)\,d\nu(\lambda),
	$
	that is,
	\begin{equation}\label{H-S positive trans}
		(Q_+^{\mathsf T}(x,t)g)(\omega)
		=
		\int_{\mathcal D}
		Q_+(x,t;\lambda,\omega)g(\lambda)\,d\nu(\lambda),
		\qquad
		\omega\in\mathcal D^*.
	\end{equation}
	The kernel can be written in the form
	\begin{equation}\label{H-S IIKS}
		Q_+(x,t;\lambda,\omega)
		=
		\frac{\phi_+(\lambda)\psi_-(\omega)}
		{\lambda-\omega},
	\end{equation}
	with
	\begin{equation}\label{kernel-IIKS}
		\phi_+(\lambda)
		=
		2i\,z_+(\lambda)
		\sqrt{\widehat f(\lambda)}
		e^{i\widehat\theta(x,t,\lambda)},
		\qquad
		\psi_-(\omega)
		=
		\sqrt{\widehat f^{\,*}(\omega)}
		e^{-i\widehat\theta(x,t,\omega)}.
	\end{equation}
	There is no diagonal singularity, since \(\mathcal D\cap\mathcal D^*=\varnothing\).
	
	\begin{prop}
		\label{prop:Q-HS}
		Under Assumption~\ref{ass:tau-density}, the operator \(Q_+(x,t)\) is Hilbert--Schmidt for every real \(x,t\). Moreover, \(Q_+^{\mathsf T}(x,t)Q_+(x,t)\) is trace class on \(H_-\).
	\end{prop}
	
	\begin{proof}
		Since \(\operatorname{dist}(\mathcal D,\mathcal D^*)>0\), there exists
		\(c_0>0\) such that
		\[
		|\lambda-\omega|\ge c_0,
		\qquad
		\lambda\in\mathcal D,\quad \omega\in\mathcal D^*.
		\]
		Moreover, \(z_+\) is bounded on the compact set \(\mathcal D\). Hence
		$
		C_0:=
		\frac{4\sup_{\lambda\in\mathcal D}|z_+(\lambda)|^2}{c_0^2}
		<\infty.
		$
		From the definition of the kernel, it is obtained that 
		\[
		|Q_+(x,t;\lambda,\omega)|^2
		\le
		C_0
		|\widehat f(\lambda)|
		|\widehat f^{\,*}(\omega)|
		\left|
		e^{i\widehat\theta(x,t,\lambda)-i\widehat\theta(x,t,\omega)}
		\right|^2.
		\]
		For fixed real \(x,t\), the exponential factor is bounded on the compact set \(\mathcal D\times\mathcal D^*\). Therefore, it follows that
		\[
		\|Q_+(x,t)\|_{\mathfrak S_2}^2
		=
		\int_{\mathcal D}
		\int_{\mathcal D^*}
		|Q_+(x,t;\lambda,\omega)|^2
		d\nu(\omega)d\nu(\lambda)
		<\infty.
		\]
		Thus \(Q_+(x,t)\) is Hilbert--Schmidt. The trace-class property of \(Q_+^{\mathsf T}Q_+\) follows because a product of Hilbert--Schmidt operators is trace class.
		
	\end{proof}
	
	\begin{remark}
		The factor \(2z_+(\lambda)\) is the genuine contribution of the DNLS equation. It comes from combining the two \(z\)-plane Cauchy factors associated with a pair of symmetric poles:
		\[
		\frac{1}{z-s}+\frac{1}{z+s}
		=
		\frac{2z}{z^2-s^2}
		=
		\frac{2z_+(\lambda)}{\lambda-\mu},
		\qquad
		\lambda=z^2,\quad \mu=s^2.
		\]
		Thus the quotient operator is not the standard operator of the NLS equation written in another variable; it retains the memory of the \(z\mapsto -z\) symmetry through the numerator \(2z_+(\lambda)\).
	\end{remark}
	
	Define the positive quotient \(\tau\)-function as
	\begin{equation}\label{tau-fun-posi}
		\tau^+(x,t)
		:=
		\det_{H_-}
		\left(\mathbb I+Q_+^{\mathsf T}(x,t)Q_+(x,t)\right).
	\end{equation}
	Equivalently, by Proposition~\ref{prop:block-det},
	$
	\tau^+(x,t)
	=
	\det_2(\mathbb I-\mathcal K_+(x,t)),
	$
	where
	\begin{equation}\label{K-operator}
		\mathcal K_+(x,t)
		=
		\begin{pmatrix}
			0 & -Q_+(x,t)\\
			Q_+^{\mathsf T}(x,t) & 0
		\end{pmatrix}.
	\end{equation}
	\begin{theorem}
		\label{thm:tau-solvability}
		Under Assumption~\ref{ass:tau-density}, for fixed real (x,t),
		$
		\tau^+(x,t)\ne0
		$
		if and only if the block operator
		\[
		\mathbb I-\mathcal K_+(x,t):H_+\oplus H_-\to H_+\oplus H_-
		\]
		is invertible. Hence \(\tau^+(x,t)\ne0\) is the Fredholm solvability condition for the positive quotient system.
	\end{theorem}
	
	\begin{proof}
		By Proposition~\ref{prop:Q-HS}, the operator \(Q_+^{\mathsf T}Q_+\) is trace class. Hence the Fredholm determinant defining \(\tau^+\) is well-defined, and it vanishes exactly when \(\mathbb I+Q_+^{\mathsf T}Q_+\) is not invertible.
		By Proposition~\ref{prop:block-invertibility}, this is equivalent to the invertibility of
		$
		\mathbb I-\mathcal K_+(x,t)
		$
		on \(H_+\oplus H_-\). The claim follows.
	\end{proof}
	
	We also need the conjugate determinant. Let
	$
	z_+^*(\omega)
	:=
	\overline{z_+(\overline\omega)},
	\ 
	\omega\in\mathcal D^*.
	$
	Define
	$
	Q_-(x,t):H_+\to H_-
	$
	by
	\begin{equation}\label{H-S negative}
		(Q_-(x,t)h)(\omega)
		=
		\int_{\mathcal D}
		Q_-(x,t;\omega,\lambda)h(\lambda)\,d\nu(\lambda),
	\end{equation}
	where
	\begin{equation}\label{kernel of negative H-S}
		Q_-(x,t;\omega,\lambda)
		=
		\frac{
			2i\,z_+^*(\omega)
			\sqrt{\widehat f^{\,*}(\omega)}
			\sqrt{\widehat f(\lambda)}
			e^{-i\widehat\theta(x,t,\omega)+i\widehat\theta(x,t,\lambda)}
		}{
			\omega-\lambda
		}.
	\end{equation}
	The conjugate \(\tau\)-function is
	\begin{equation}\label{tau-fun-nega}
		\tau^-(x,t)
		:=
		\det_{H_+}
		\left(\mathbb I+Q_-^{\mathsf T}(x,t)Q_-(x,t)\right).
	\end{equation}
	For real \(x,t\), the Schwarz symmetry gives
	$
	\tau^-(x,t)=\overline{\tau^+(x,t)}
	$
	whenever the square-root branches are chosen compatibly. We keep the two symbols \(\tau^+\) and \(\tau^-\) because they arise from the two conjugate residue systems and enter the reconstruction formula as a ratio.

	\subsection{Fredholm limit of the finite residue systems}
	\label{subsec:fredholm-solvability}
	
	We now connect the continuum determinant with the finite reflectionless problem. Let
	\[
	z_1,\ldots,z_N\in \mathcal D_+,
	\qquad
	\lambda_j:=z_j^2,
	\qquad
	z_+(\lambda_j)=z_j.
	\]
	Let \(c_j\) be the corresponding norming constants and set
	\[
	\alpha_j(x,t)
	:=
	c_j e^{2i\widehat\theta(x,t,\lambda_j)},
	\qquad
	\alpha_j^*(x,t)
	:=
	\overline{c_j}\,
	e^{-2i\widehat\theta(x,t,\overline{\lambda}_j)}.
	\]
	Define
	\begin{equation}\label{diag-A}
		A_N(x,t):=\operatorname{diag}(\alpha_1,\ldots,\alpha_N),
		\qquad
		A_N^*(x,t):=\operatorname{diag}(\alpha_1^*,\ldots,\alpha_N^*).
	\end{equation}
	Introduce the Cauchy-type matrix with entries
	\begin{equation}\label{Cauchy matrix}
		(\mathcal C_N)_{jk}
		:=
		\frac{2z_j}{\lambda_j-\overline{\lambda}_k},
		\qquad
		1\le j,k\le N.
	\end{equation}
	The factor \(2z_j\) is the finite-dimensional counterpart of the numerator \(2z_+(\lambda)\) in the quotient Fredholm kernel.
	\par
	Define
	\begin{equation}\label{discrete-tau-fun}
		\tau_N^+(x,t)
		:=
		\det_{\mathbb C^N}
		\left(
		\mathbb I_N
		-
		A_N^*(x,t)\mathcal C_N^{\mathsf T}
		A_N(x,t)\mathcal C_N
		\right).
	\end{equation}
	The conjugate determinant \(\tau_N^-(x,t)\) is obtained by applying Schwarz conjugation to the finite residue system.
	
	\begin{prop}
		\label{prop:finite-det}
		The determinant \(\tau_N^+(x,t)\) is the determinant of the finite algebraic system obtained from the pure \(N\)-soliton residue conditions. In particular, the solution 
		\(u_N(x,t)\) corresponding reflectionless potential of the DNLS equation \eqref{dnls3} satisfies
		\begin{equation}\label{discrete-construction-formula}
			|u_N(x,t)|^2
			=
			2i\,\partial_x
			\log
			\frac{\tau_N^+(x,t)}{\tau_N^-(x,t)}
		\end{equation}
		at every point where \(\tau_N^+(x,t)\tau_N^-(x,t)\ne0\).
	\end{prop}

	\begin{proof} 
		See Appendix \ref{proof-prop-7} for the detailed proof of this proposition.
	\end{proof}

	\subsection{From finite residue determinants to the continuum density}
	\label{subsec:finite-determinants-density}
	
	We next formulate the convergence statement needed to pass from pure \(N\)-soliton determinants to the continuum \(\tau\)-function.

	\begin{prop}\label{prop:verification-HS-convergence}
		Suppose that there exist $0<\beta\leq 1$ and a function $g\in C^\beta(\mathrm{cl}(\mathcal{D}_+))$
		such that
		\[
		g(z)^2=f(z,\overline z),
		\qquad
		z\in\mathrm{cl}(\mathcal{D}_+).
		\]
		On the Schwarz conjugate domain, choose the compatible branch
		$
		g^\sharp(z):=\overline{g(z)}.
		$
		Define the finite matrix
		\begin{equation}
			\mathbf Q_{+,N}(x,t)
			:=
			iA_N(x,t)^{1/2}C_N
			\bigl(A_N^*(x,t)\bigr)^{1/2},
		\end{equation}
		where the square roots are chosen according to
		\[
		\sqrt{\alpha_{j,N}(x,t)}
		=
		\left(\frac{\mathcal A}{\pi N}\right)^{1/2}
		g(z_{j,N})
		e^{i\widehat\theta(x,t,\lambda_{j,N})},
		\quad
		\sqrt{\alpha_{j,N}^*(x,t)}
		=
		\left(\frac{\mathcal A}{\pi N}\right)^{1/2}
		g^\sharp(z_{j,N})
		e^{-i\widehat\theta
			\left(x,t,\overline{\lambda_{j,N}}\right)}.
		\]
		Then there exist isometric embeddings
		\[
		\mathcal I_{+,N}:\mathbb C^N\longrightarrow H_+,
		\qquad
		\mathcal I_{-,N}:\mathbb C^N\longrightarrow H_-,
		\]
		such that, with
		\[
		Q_{+,N}(x,t)
		:=
		\mathcal I_{+,N}\mathbf Q_{+,N}(x,t)
		\mathcal I_{-,N}^{\dagger},
		\]
		one has, for every compact set $K\Subset\mathbb R^2$,
		\[
		\sup_{(x,t)\in K}
		\left\|
		Q_{+,N}(x,t)-Q_+(x,t)
		\right\|_{\mathfrak S_2}
		\longrightarrow 0,
		\quad
		\sup_{(x,t)\in K}
		\left\|
		\partial_xQ_{+,N}(x,t)
		-
		\partial_xQ_+(x,t)
		\right\|_{\mathfrak S_2}
		\longrightarrow 0.
		\]
		The analogous conclusions hold for the conjugate operators
		$Q_{-,N}\to Q_-$.
	\end{prop}
	
	\begin{proof}
		See Appendix \ref{proof-prop-7} for the detailed proof of this proposition.
	\end{proof}
	
	\begin{theorem}
		\label{thm:tau-limit}
		Under Assumptions~\ref{ass:tau-density}, the finite-dimensional determinants converge to the continuum determinants:
		\[
		\tau_N^+(x,t)\longrightarrow \tau^+(x,t),
		\qquad
		\tau_N^-(x,t)\longrightarrow \tau^-(x,t),
		\]
		locally uniformly for \((x,t)\in\mathbb R^2\). Furthermore, on every compact set on which
		$
		\tau^+(x,t)\tau^-(x,t)\ne0,
		$
		one has
		\[
		\partial_x
		\log
		\frac{\tau_N^+(x,t)}{\tau_N^-(x,t)}
		\longrightarrow
		\partial_x
		\log
		\frac{\tau^+(x,t)}{\tau^-(x,t)}
		\]
		locally uniformly.
	\end{theorem}
	
	\begin{proof}
		Proposition~\ref{prop:verification-HS-convergence} readily implies that
		\[
		Q_{+,N}^{\mathsf T}Q_{+,N}
		\longrightarrow
		Q_+^{\mathsf T}Q_+,
		\qquad
		\partial_x\bigl(Q_{+,N}^{\mathsf T}Q_{+,N}\bigr)
		\longrightarrow
		\partial_x\bigl(Q_+^{\mathsf T}Q_+\bigr)
		\]
		in trace norm, locally uniformly in $(x,t)$. The same statements hold for the conjugate operators. The Fredholm determinant is continuous with respect to the trace norm, hence
		\[
		\det(\mathbb I+Q_{+,N}^{\mathsf T}Q_{+,N})
		\to
		\det(\mathbb I+Q_+^{\mathsf T}Q_+).
		\]
		This gives \(\tau_N^+\to\tau^+\). The proof for \(\tau_N^-\to\tau^-\) is identical.
		
		The convergence of logarithmic derivatives follows from the trace-norm convergence of both the operators and their \(x\)-derivatives. Let \(T(x)\) be a trace-class operator differentiable in trace norm. Then 
		\[
		\partial_x\log\det(\mathbb I+T(x))
		=
		\operatorname{Tr}\left((\mathbb I+T(x))^{-1}\partial_xT(x)\right),
		\]
		whenever \(\mathbb I+T(x)\) is invertible. On compact sets where \(\tau^+\tau^-\ne0\), the inverses of the limiting operators are bounded. The trace-norm convergence and the resolvent identity imply convergence of the traces, and therefore convergence of the logarithmic derivatives.
	\end{proof}
	
	We are now ready to state the determinant representation of the continuum soliton gas.
	
	Let
	$
	\Omega_\tau
	:=
	\{(x,t):\tau^+(x,t)\tau^-(x,t)\ne0\}.
	$
	We prove the identity locally on \(\Omega_\tau\). Let \(K\Subset\Omega_\tau\) be compact. Since \(\tau^+\tau^-\) is continuous and nonzero on \(K\), there exists \(\delta_K>0\) such that
	\[
	|\tau^+(x,t)\tau^-(x,t)|\ge \delta_K,
	\qquad
	(x,t)\in K.
	\]
	By Theorem~\ref{thm:tau-limit}, we have
	$
	\tau_N^\pm\to\tau^\pm
	$
	locally uniformly. Hence, for all sufficiently large \(N\), we have
	\[
	\tau_N^+(x,t)\tau_N^-(x,t)\ne0,
	\qquad
	(x,t)\in K.
	\]
	Therefore Proposition~\ref{prop:finite-det} applies on \(K\), and gives
	\[
	|u_N(x,t)|^2
	=
	2i\,\partial_x
	\log
	\frac{\tau_N^+(x,t)}{\tau_N^-(x,t)},
	\qquad
	(x,t)\in K.
	\]
	
	We now pass to the limit on both sides. By the Theorem \ref{thm:finite-to-dbar},
	$
	u_N\to u
	$
	uniformly on \(K\). Consequently,
	$
	|u_N|^2\to |u|^2
	$
	uniformly on \(K\). Indeed,
	\[
	\sup_K\left||u_N|^2-|u|^2\right|
	\le
	\sup_K|u_N-u|
	\left(
	\sup_K|u_N|+\sup_K|u|
	\right),
	\]
	and the right-hand side tends to zero because \(u_N\to u\) uniformly on \(K\).
	
	On the other hand, Theorem~\ref{thm:tau-limit} gives the locally uniform convergence of logarithmic derivatives on \(K\):
	\[
	\partial_x
	\log
	\frac{\tau_N^+(x,t)}{\tau_N^-(x,t)}
	\longrightarrow
	\partial_x
	\log
	\frac{\tau^+(x,t)}{\tau^-(x,t)}.
	\]
	Taking the limit in the finite identity therefore yields
	\begin{equation}\label{continuum-reconstruction}
		|u(x,t)|^2
		=
		2i\,\partial_x
		\log
		\frac{\tau^+(x,t)}{\tau^-(x,t)},
		\qquad
		(x,t)\in K.
	\end{equation}
	Since \(K\Subset\Omega_\tau\) was arbitrary, the identity holds throughout \(\Omega_\tau\).
	
	\begin{remark}
		For real \(x,t\), the Schwarz symmetry gives \(\tau^-=\overline{\tau^+}\) under compatible choices of branches. Hence
		\[
		\frac{\tau^+}{\tau^-}
		=
		\frac{\tau^+}{\overline{\tau^+}}
		\]
		has modulus one wherever \(\tau^+\ne0\). Thus \(|u|^2\) is encoded in the \(x\)-variation of the phase of the determinant for the DNLS equation. This is different from determinant formulae in which the physical density is recovered from the logarithmic derivative of the modulus of a real positive determinant.
	\end{remark}
	
	\begin{remark}
		The determinant representation is not needed for the nonlinear steepest descent analysis in the previous sections. Its role is structural: it shows that the continuum soliton gas is not merely a formal large-\(N\) object, but is governed by a Fredholm determinant obtained as the limit of finite reflectionless residue determinants. The factor \(2z_+(\lambda)\) in the quotient kernel records the symmetry \(z\mapsto -z\), while the determinant ratio records the density \(|u|^2\) of the equation \eqref{dnls3}.
	\end{remark}

	\appendix
	\section{Gauge equivalence among the three DNLS equations}
	\label{app:gauge-equivalence}
	
	The purpose of this appendix is to record the gauge equivalence among the
	three DNLS equations.  The main consequence used in this paper is that the Kaup--Newell and Chen--Lee--Liu fields can be recovered from the Gerdjikov--Ivanov field by nonlocal phase transformations.
	
	\begin{prop}\label{prop:gauge-equivalence}
		Let $u(x,t)$ be a solution of the
		Gerdjikov--Ivanov equation \eqref{dnls3}.  Let $\Phi_u(x,t,z)$ be a
		fundamental solution of the Lax pair \eqref{eq:lax3}
		associated with the Gerdjikov--Ivanov equation, normalized so that
		$
		G(x,t):=\Phi_u(x,t,0)
		$
		is diagonal.  Write
		\[
		G(x,t)=G_{11}(x,t)^{\sigma_3}
		=
		\begin{pmatrix}
			G_{11}(x,t) & 0\\
			0 & G_{11}(x,t)^{-1}
		\end{pmatrix}.
		\]
		Then
		$
		G_{11}(x,t)
		=
		\exp\left\{
		\frac{i}{2}
		\int_{-\infty}^{x}|u(y,t)|^2\,dy
		\right\}.
		$
		Define
		\begin{equation}
			\label{app:gauge-relations}
			q(x,t):=u(x,t)G_{11}(x,t)^{-2}, 
			\quad
			Q(x,t):=u(x,t)G_{11}(x,t)^{-1}. 
		\end{equation}
		Then $q(x,t)$ solves the Kaup--Newell equation \eqref{dnls1}, and $Q(x,t)$
		solves the Chen--Lee--Liu equation \eqref{dnls2}.  Moreover,
		\[
		|q(x,t)|^2=|Q(x,t)|^2=|u(x,t)|^2.
		\]
		At the level of eigenfunctions, the corresponding gauge transformations \cite{Kitaev-Vartanian-1999} are
		\[
		\Phi_q(x,t,z):=G(x,t)^{-1}\Phi_u(x,t,z),
		\qquad
		\Phi_Q(x,t,z):=G(x,t)^{-1/2}\Phi_u(x,t,z),
		\]
		which transform the Lax pair of the Gerdjikov--Ivanov equation into the Lax pairs of the
		Kaup--Newell and Chen--Lee--Liu equations, respectively.
	\end{prop}
	
	\begin{proof}
		For a complex-valued function $f$, set
		\[
		P_f=
		\begin{pmatrix}
			0 & f\\
			-\bar f & 0
		\end{pmatrix},
		\qquad
		\sigma_3=
		\begin{pmatrix}
			1 & 0\\
			0 & -1
		\end{pmatrix}.
		\]
		The function $u=u(x,t)$ solves the Gerdjikov--Ivanov equation \eqref{dnls3} if and only if there exists a matrix-valued function $\Phi_u=\Phi_u(x,t,z)$, depending on the spectral parameter $z\in\mathbb{C}$, which satisfies the linear system
		\begin{equation}
			\label{eq:lax3}
			\partial_x\Phi_u=X_u\Phi_u,
			\qquad
			\partial_t\Phi_u=T_u\Phi_u,
		\end{equation}
		where
		\[
		X_u
		=
		-iz^2\sigma_3
		+zP_u
		+\frac{i}{2}|u|^2\sigma_3,
		\]
		and
		\[
		T_u
		=
		-2iz^4\sigma_3
		+2z^3P_u
		+iz^2|u|^2\sigma_3
		-iz(P_u)_x\sigma_3
		+
		\left(
		\frac{i}{4}|u|^4
		+\frac{1}{2}(u\bar u_x-\bar u u_x)
		\right)\sigma_3.
		\]
		Evaluating \eqref{eq:lax3} at $z=0$ gives
		\[
		(G_{11})_x
		=
		\frac{i}{2}|u|^2G_{11}.
		\]
		With the normalization $G_{11}(-\infty,t)=1$, we obtain
		\[
		G_{11}(x,t)
		=
		\exp\left\{
		\frac{i}{2}
		\int_{-\infty}^{x}|u(y,t)|^2\,dy
		\right\}.
		\]
		In particular, $|G_{11}(x,t)|=1$, and therefore the transformations
		\eqref{app:gauge-relations} preserve the density:
		\[
		|q(x,t)|=|Q(x,t)|=|u(x,t)|.
		\]
		
		We first derive the Lax pair.  Let
		\[
		g(x,t):=G_{11}(x,t),
		\qquad
		\Phi_q:=G^{-1}\Phi_u,
		\qquad
		q:=ug^{-2}.
		\]
		Then $G^{-1}P_uG=P_q.$
		Since $G^{-1}G_x=\frac{i}{2}|u|^2\sigma_3,$
		we have
		\[
		\partial_x\Phi_q
		=
		\left(
		-G^{-1}G_x+G^{-1}X_uG
		\right)\Phi_q
		=
		\left(
		-iz^2\sigma_3+zP_q
		\right)\Phi_q.
		\]
		For the $t$-part, it can be checked that
		\[
		G^{-1}(P_u)_xG
		=
		(P_q)_x+i|q|^2\sigma_3P_q.
		\]
		Substituting this identity into
		\[
		\partial_t\Phi_q
		=
		\left(
		-G^{-1}G_t+G^{-1}T_uG
		\right)\Phi_q
		\]
		gives
		\begin{subequations}
			\label{eq:lax1}
			\begin{align}
				\partial_x\Phi_q
				&=
				\left[
				-iz^2\sigma_3+zP_q
				\right]\Phi_q,
				\\
				\partial_t\Phi_q
				&=
				\left[
				-2iz^4\sigma_3
				+2z^3P_q
				+iz^2|q|^2\sigma_3
				-z|q|^2P_q
				-iz(P_q)_x\sigma_3
				\right]\Phi_q.
			\end{align}
		\end{subequations}
		This is the Lax pair of the Kaup--Newell equation, whose zero-curvature condition is just the equation
		\eqref{dnls1}.
		
		We next derive the Lax pair of the Chen--Lee--Liu equation.  Let
		\[
		H:=G^{1/2},
		\qquad
		\Phi_Q:=H^{-1}\Phi_u,
		\qquad
		Q:=ug^{-1}.
		\]
		Then
		\[
		H^{-1}P_uH=P_Q,
		\qquad
		H^{-1}H_x=\frac{i}{4}|Q|^2\sigma_3,
		\quad
		H^{-1}(P_u)_xH
		=
		(P_Q)_x+\frac{i}{2}|Q|^2\sigma_3P_Q.
		\]
		Moreover, since $u=Qg$, we have
		\[
		u\bar u_x-\bar u u_x
		=
		Q\bar Q_x-\bar Q Q_x-i|Q|^4.
		\]
		Using these identities in
		\[
		\partial_x\Phi_Q
		=
		\left(
		-H^{-1}H_x+H^{-1}X_uH
		\right)\Phi_Q,
		\quad
		\partial_t\Phi_Q
		=
		\left(
		-H^{-1}H_t+H^{-1}T_uH
		\right)\Phi_Q,
		\]
		it is derived that
		\begin{subequations}
			\label{eq:lax2}
			\begin{align}
				\partial_x\Phi_Q
				&=
				\left[
				-iz^2\sigma_3
				+zP_Q
				+\frac{i}{4}|Q|^2\sigma_3
				\right]\Phi_Q,
				\\
				\partial_t\Phi_Q
				&=
				\Bigg[
				-2iz^4\sigma_3
				+2z^3P_Q
				+iz^2|Q|^2\sigma_3
				-iz(P_Q)_x\sigma_3
				-\frac{1}{2}z|Q|^2P_Q
				\notag\\
				&\hspace{2.8cm}
				-\frac{i}{8}|Q|^4\sigma_3
				+\frac{1}{4}(Q\bar Q_x-\bar Q Q_x)\sigma_3
				\Bigg]\Phi_Q,
			\end{align}
		\end{subequations}
		which is just the Lax pair of the Chen--Lee--Liu equation
		\eqref{dnls2}.
	\end{proof}

	\section{Proof of Proposition \ref{prop:finite-det} and Proposition \ref{prop:verification-HS-convergence}} \label{proof-prop-7}
	
	\begin{proof}[\textbf{Proof of Proposition \ref{prop:finite-det}}]
		Let
		\[
		P:=(p_1,\ldots,p_N)^{\mathsf T},
		\qquad
		R:=(\overline q_1,\ldots,\overline q_N)^{\mathsf T},
		\qquad
		\mathbf 1:=(1,\ldots,1)^{\mathsf T},
		\]
		where \(p_j\) and \(q_j\), \(j=1,\ldots,N\), are those defined in Proposition~\ref{prop:pure-N-soliton}. 
		The factor \(2z_j\) in \eqref{Cauchy matrix} is the finite-dimensional trace of the symmetry \(z\mapsto -z\) of the DNLS equation, since
		\[
		\frac{1}{z_j-\overline z_k}
		+
		\frac{1}{z_j+\overline z_k}
		=
		\frac{2z_j}{z_j^2-\overline z_k^{,2}}
		=
		\frac{2z_j}{\lambda_j-\overline\lambda_k}.
		\]
		According to the finite algebraic system \eqref{eq:finite-algebraic-system}, we obtain
		\begin{equation}\label{P-R}
			P=-A_NKR,\quad
			R=A^*_N(\mathbf 1-K^{\mathsf T}P),
		\end{equation}
		where $A_N$ and $A^*_N$ are given in \eqref{diag-A}. Moreover,
		$
		R=A_N^*\mathbf 1+A_N^*K^{\mathsf T}A_NKR,
		$
		or equivalently
		\begin{equation}\label{matrix B_+}
			B_+R=A^*_N\mathbf 1,
			\qquad
			B_+:=\mathbb I_N-A^*_NK^{\mathsf T}A_NK.
		\end{equation}
		Thus
		$
		\tau_N^+(x,t)=\det B_+
		$
		is precisely the determinant of the finite algebraic system associated with the residue conditions.
		
		We next compute the \(x\)-derivative of \(\tau_N^+\). At points where \(\tau_N^+\ne0\), Jacobi's formula gives
		$
		\partial_x\log\tau_N^+
		=
		\operatorname{Tr}(B_+^{-1}\partial_xB_+).
		$
		Let
		$
		\Lambda:=\operatorname{diag}(\lambda_1,\ldots,\lambda_N),
		\ 
		\overline\Lambda:=\operatorname{diag}(\overline\lambda_1,\ldots,\overline\lambda_N).
		$
		Since
		$
		\partial_xA_N=2i\Lambda A_N,
		\ 
		\partial_xA^*_N=-2i\overline\Lambda A^*_N,
		$
		we have
		\[
		\partial_xB_+
		=
		2iA^*_N(\overline\Lambda K^{\mathsf T}-K^{\mathsf T}\Lambda)A_NK.
		\]
		For the matrix in parentheses,
		$
		\overline\Lambda K^{\mathsf T}-K^{\mathsf T}\Lambda
		=
		-2\mathbf 1\,Z^{\mathsf T}
		$
		with $Z:=(z_1,\ldots,z_N)^{\mathsf T}$.
		It follows that
		\[
		\partial_xB_+
		=
		-4iA^*_N\mathbf 1\,Z^{\mathsf T}A_NK.
		\]
		Therefore
		$
		\partial_x\log\tau_N^+
		=
		-4i\operatorname{Tr}
		\left(
		B_+^{-1}A^*_N\mathbf 1\,Z^{\mathsf T}A_NK
		\right).
		$
		Using cyclicity of the trace for this rank-one product gives
		\[
		\partial_x\log\tau_N^+
		=
		-4i\,Z^{\mathsf T}A_NKB_+^{-1}A^*_N\mathbf 1=
		-4i\,Z^{\mathsf T}A_NKR.
		\]
		Using \(P=-A_NKR\), we obtain
		$
		\partial_x\log\tau_N^+
		=
		4i\,Z^{\mathsf T}P
		=
		4i\sum_{j=1}^Nz_jp_j.
		$
		The Schwarz conjugate residue system gives, in the same way,
		$
		\partial_x\log\tau_N^-
		=
		-4i\sum_{j=1}^N\overline z_j\,\overline p_j.
		$
		Consequently,
		\begin{equation}\label{derivative-x-log}
			\partial_x\log\frac{\tau_N^+}{\tau_N^-}
			=
			4i\sum_{j=1}^N
			\left(
			z_jp_j+\overline z_j\,\overline p_j
			\right).
		\end{equation}
		
		It remains to identify this expression with \(|u_N|^2\). Expanding the reflectionless solution \(M^{(N)}(z)\) at infinity, one has
		\[
		M^{(N)}(z)
		=
		\mathbb I+\frac{M^{(N)}_1}{z}
		+\frac{M^{(N)}_2}{z^2}
		+O(z^{-3}),
		\qquad z\to\infty.
		\]
		From the rational ansatz for \(M^{(N)}(z)\), the coefficient of \(z^{-1}\) is
		$
		M^{(N)}_1
		=
		\begin{pmatrix}
			0 & -2\sum_{j=1}^N\overline q_j\\[2mm]
			2\sum_{j=1}^N q_j & 0
		\end{pmatrix}.
		$
		Thus the reconstruction formula gives
		$
		u_N(x,t)
		=
		2i(M^{(N)}_1)_{12}
		=
		-4i\sum_{j=1}^N\overline q_j.
		$
		
		The diagonal part of the \(z^{-2}\)-coefficients are
		$
		(M^{(N)}_2)_{11}=2\sum_{j=1}^Nz_jp_j$,
		and
		$(M^{(N)}_2)_{22}=2\sum_{j=1}^N\overline z_j\,\overline p_j.
		$
		Therefore, it follows that
		\[
		\operatorname{Tr}M^{(N)}_2
		=
		2\sum_{j=1}^N
		\left(
		z_jp_j+\overline z_j\,\overline p_j
		\right).
		\]
		Since $
		\det M_N(z)\equiv1,
		$
		one yields that
		$
		\operatorname{Tr}M^{(N)}_2+\det M^{(N)}_1=0.
		$
		Therefore
		$
		|u_N(x,t)|^2
		=
		-8\sum_{j=1}^N
		\left(
		z_jp_j+\overline z_j\,\overline p_j
		\right).
		$
		Combining this with \eqref{derivative-x-log}
		gives
		\[
		2i\,\partial_x\log\frac{\tau_N^+}{\tau_N^-}
		=
		-8\sum_{j=1}^N
		\left(
		z_jp_j+\overline z_j\,\overline p_j
		\right)
		=
		|u_N(x,t)|^2.
		\]
		This proves the identity.
	\end{proof}

	\begin{proof}[\textbf{Proof of Proposition \ref{prop:verification-HS-convergence}}]
		Since all the measures are supported on the fixed compact set
		$\overline{D_+}$, the weak convergence
		$\mu_N\rightharpoonup\mu$ implies
		\[
		W_2(\mu_N,\mu)\longrightarrow 0,
		\]
		where $W_2$ denotes the quadratic Wasserstein distance.
		Moreover, $\mu$ is absolutely continuous with respect to planar
		Lebesgue measure. Hence, for every $N$, there exists a Borel
		transport map
		\[
		T_N:\mathcal{D}_+\longrightarrow
		\{z_{1,N},\ldots,z_{N,N}\}
		\]
		such that $(T_N)_\#\mu=\mu_N$
		and
		\[
		\varepsilon_N^2
		:=
		\int_{D_+}|T_N(z)-z|^2\,d\mu(z)
		=
		W_2(\mu_N,\mu)^2
		\longrightarrow 0.
		\]
		Set $E_{j,N}
		:=
		T_N^{-1}\bigl(\{z_{j,N}\}\bigr)$,
		then
		\[
		\mu(E_{j,N})=\frac{1}{N},
		\qquad
		j=1,\ldots,N.
		\]
		Consequently, the functions
		\[
		e_{j,N}(z)
		:=
		\sqrt{N}\,\chi_{E_{j,N}}(z),
		\qquad
		j=1,\ldots,N,
		\]
		form an orthonormal system in $\mathcal H:=L^2(\mathcal D_+,d\mu)$.
		Define the isometry
		\[
		\mathcal J_N:\mathbb C^N\longrightarrow\mathcal H,
		\qquad
		\mathcal J_N(a_1,\ldots,a_N)
		:=
		\sum_{j=1}^Na_je_{j,N}.
		\]
		We next pull the quotient operators back to the original
		$z$-plane. Define
		\[
		U_+:H_+\longrightarrow\mathcal H,
		\qquad
		(U_+h)(z)
		:=
		2\left(\frac{\mathcal A}{\pi}\right)^{1/2}
		|z|\,h(z^2),
		\]
		and
		\[
		U_-:H_-\longrightarrow\mathcal H,
		\qquad
		(U_-h)(z)
		:=
		2\left(\frac{\mathcal A}{\pi}\right)^{1/2}
		|z|\,h(\overline z^{\,2}).
		\]
		The change of variables $\lambda=z^2$ gives $dA_\lambda=4|z|^2\,dA_z$,
		and therefore $U_+$ and $U_-$ are unitary.
		Let
		\[
		\widetilde Q_+(x,t)
		:=
		U_+Q_+(x,t)U_-^{-1}.
		\]
		A direct change of variables in the kernel formula for $Q_+$
		shows that $\widetilde Q_+$ is the integral operator on
		$\mathcal H$ with kernel
		\[
		\widetilde q_{x,t}(z,s)
		=
		\frac{2i\mathcal A}{\pi}
		\frac{
			z\,g(z)g^\sharp(s)
			e^{i\widehat\theta(x,t,z^2)
				-i\widehat\theta(x,t,\overline s^{\,2})}
		}{
			z^2-\overline s^{\,2}
		},
		\qquad
		z,s\in \mathcal{D}_+.
		\]
		Indeed,
		\[
		\bigl(\widetilde Q_+(x,t)h\bigr)(z)
		=
		\int_{\mathcal{D}_+}
		\widetilde q_{x,t}(z,s)h(s)\,d\mu(s).
		\]
		The entries of matrix $\mathbf Q_{+,N}$ are
		\[
		\begin{aligned}
			\bigl(\mathbf Q_{+,N}(x,t)\bigr)_{jk}
			&=
			\frac{
				2iz_{j,N}
				\sqrt{\alpha_{j,N}(x,t)}
				\sqrt{\alpha_{k,N}^*(x,t)}
			}{
				\lambda_{j,N}-\overline{\lambda_{k,N}}
			}
			\\
			&=
			\frac{1}{N}
			\widetilde q_{x,t}(z_{j,N},z_{k,N}).
		\end{aligned}
		\]
		Define
		\[
		\widetilde Q_{+,N}(x,t)
		:=
		\mathcal J_N\mathbf Q_{+,N}(x,t)\mathcal J_N^\dagger.
		\]
		Then $\widetilde Q_{+,N}$ is the integral operator on
		$\mathcal H$ whose kernel is
		\[
		\widetilde q^{(N)}_{x,t}(z,s)
		=
		\widetilde q_{x,t}\bigl(T_N(z),T_N(s)\bigr).
		\]
		Finally, set
		\[
		\mathcal I_{+,N}:=U_+^{-1}\mathcal J_N,
		\qquad
		\mathcal I_{-,N}:=U_-^{-1}\mathcal J_N.
		\]
		These maps are isometric, and
		\[
		Q_{+,N}
		=
		U_+^{-1}\widetilde Q_{+,N}U_-
		=
		\mathcal I_{+,N}\mathbf Q_{+,N}
		\mathcal I_{-,N}^\dagger.
		\]
		
		Since $\operatorname{dist}(\mathcal{D},\mathcal{D}^*)>0$,
		there exists $c_0>0$ such that
		\[
		|z^2-\overline s^{\,2}|
		\geq c_0,
		\qquad
		z,s\in\mathrm{cl}(\mathcal{D}_+).
		\]
		It follows from $g\in C^\beta(\mathrm{cl}(\mathcal{D}_+))$ that, for every
		compact set $K\Subset\mathbb R^2$, there exists $C_K>0$ such that
		\[
		\begin{aligned}
			\sup_{(x,t)\in K}
			\left|
			\widetilde q_{x,t}(z,s)
			-
			\widetilde q_{x,t}(z',s')
			\right|
			\leq
			C_K
			\left(
			|z-z'|^\beta+|s-s'|^\beta
			\right)
		\end{aligned}
		\]
		for all $z,z',s,s'\in\overline{\mathcal{D}_+}$.
		
		The Hilbert--Schmidt norm of an integral operator equals the
		$L^2$ norm of its kernel. Hence
		\[
		\begin{aligned}
			&
			\left\|
			\widetilde Q_{+,N}(x,t)
			-
			\widetilde Q_+(x,t)
			\right\|_{\mathfrak S_2}^2
			\\
			&\qquad=
			\int_{\mathcal{D}_+}\int_{\mathcal{D}_+}
			\left|
			\widetilde q_{x,t}
			\bigl(T_N(z),T_N(s)\bigr)
			-
			\widetilde q_{x,t}(z,s)
			\right|^2
			\,d\mu(s)\,d\mu(z).
		\end{aligned}
		\]
		Using the preceding Hölder estimate and Jensen's inequality, we
		obtain
		\[
		\begin{aligned}
			\sup_{(x,t)\in K}
			\left\|
			\widetilde Q_{+,N}(x,t)
			-
			\widetilde Q_+(x,t)
			\right\|_{\mathfrak S_2}^2
			&\leq
			C_K
			\int_{D_+}
			|T_N(z)-z|^{2\beta}\,d\mu(z)
			\\
			&\leq
			C_K\varepsilon_N^{2\beta}.
		\end{aligned}
		\]
		Differentiating the pulled-back kernel with respect to $x$ gives
		\[
		\partial_x\widetilde q_{x,t}(z,s)
		=
		-\frac{2\mathcal A}{\pi}
		z\,g(z)g^\sharp(s)
		e^{i\widehat\theta(x,t,z^2)
			-i\widehat\theta(x,t,\overline s^{\,2})}.
		\]
		In particular, the Cauchy denominator disappears. The family
		$\{\partial_x\widetilde q_{x,t}:(x,t)\in K\}$ satisfies the same
		uniform $C^\beta$ estimate as
		$\{\widetilde q_{x,t}:(x,t)\in K\}$. Repeating the preceding
		argument yields
		\[
		\sup_{(x,t)\in K}
		\left\|
		\partial_x\widetilde Q_{+,N}(x,t)
		-
		\partial_x\widetilde Q_+(x,t)
		\right\|_{\mathfrak S_2}
		\leq
		C_K\varepsilon_N^\beta.
		\]
		
		Since the Hilbert--Schmidt norm is invariant under unitary
		conjugation, we have
		\[
		\begin{aligned}
			&
			\left\|
			Q_{+,N}-Q_+
			\right\|_{\mathfrak S_2}
			+
			\left\|
			\partial_xQ_{+,N}-\partial_xQ_+
			\right\|_{\mathfrak S_2}
			\\
			&\qquad=
			\left\|
			\widetilde Q_{+,N}-\widetilde Q_+
			\right\|_{\mathfrak S_2}
			+
			\left\|
			\partial_x\widetilde Q_{+,N}
			-
			\partial_x\widetilde Q_+
			\right\|_{\mathfrak S_2}.
		\end{aligned}
		\]
		The asserted estimates now follow from
		$\varepsilon_N=W_2(\mu_N,\mu)\to0$.
		
		The proof for $Q_{-,N}\to Q_-$ is identical after applying
		Schwarz conjugation and using the compatible square-root
		branches.
	\end{proof}

	\paragraph{Data availability:}  The manuscript has no associated data.
	\paragraph{Declarations} 
	\paragraph{Conflict of interest:} 
	The authors declare that they have no conflict of interest.

	\paragraph{Acknowledgments:}
	This work is supported by the National Natural Science Foundation of China, Grant No. 12371247 and No. 12431008, Beijing Natural Science Foundation Grant No. 1262012 and No. JQ26004 and Key Project of the Natural Science Foundation of Inner Mongolia Autonomous Region Grant No. 2026ZD036.

\end{document}